\documentclass[11pt,oneside,a4paper]{amsart}

\usepackage[utf8]{inputenc}
\usepackage[english]{babel}
\usepackage[T1]{fontenc}
\usepackage{nicefrac}
\usepackage{microtype}
\usepackage{amsthm,amsmath,amsfonts,amssymb}
\usepackage[foot]{amsaddr}
\usepackage{geometry}
\usepackage[mathscr]{euscript}
\usepackage[bb=px]{mathalfa}
\usepackage[normalem]{ulem}

\usepackage{newpxtext}

\usepackage{mathtools}
\usepackage{enumerate}
\usepackage{graphicx}
\usepackage{subfigure}
\usepackage{tabularx}
\usepackage{booktabs}
\usepackage{hyperref}
\usepackage{footnote}

\usepackage[dvipsnames]{xcolor}
\colorlet{inlinkcolor}{purple}
\colorlet{exlinkcolor}{purple}
\colorlet{citecolor}{teal!75!blue}
\hypersetup{colorlinks=true,
            linkcolor=inlinkcolor,
            citecolor=citecolor,
            urlcolor=exlinkcolor,
            linktoc=page,
            breaklinks=true,
            plainpages=false}

\usepackage{listings}
\definecolor{codeblue}{rgb}{0.1,0.1,0.7}
\definecolor{codegreen}{rgb}{0,0.6,0}
\definecolor{codegray}{rgb}{0.5,0.5,0.5}
\definecolor{codepurple}{rgb}{0.58,0,0.82}
\definecolor{backcolour}{rgb}{0.95,0.95,0.92}

\lstdefinestyle{mypython}{
    backgroundcolor=\color{backcolour},
    commentstyle=\color{codegray}\itshape,
    keywordstyle=\color{codeblue}\bfseries,
    numberstyle=\tiny\color{codegray},
    stringstyle=\color{codepurple},
    basicstyle=\ttfamily\footnotesize,
    breakatwhitespace=false,
    breaklines=true,
    captionpos=b,
    keepspaces=true,
    showspaces=false,
    showstringspaces=false,
    showtabs=true,
    tabsize=2,
    language=Python
}

\newtheorem{theorem}{Theorem}[section]
\newtheorem{proposition}[theorem]{Proposition}
\newtheorem{lemma}[theorem]{Lemma}
\newtheorem{corollary}[theorem]{Corollary}
\theoremstyle{definition}
\newtheorem{definition}[theorem]{Definition}
\newtheorem{example}[theorem]{Example}
\newtheorem{assumption}{Assumption}
\theoremstyle{remark}
\newtheorem{remark}[theorem]{Remark}

\numberwithin{equation}{section}

\def\what{\widehat}
\let\originalparagraph\paragraph
\renewcommand{\paragraph}[1]{\vspace*{5pt}\originalparagraph{\bfseries #1}}

\usepackage{bm,bbm}

\def\vxi{{\bm{\xi}}}

\def\vepsilon{{\bm{\varepsilon}}}

\def\vb{\mathbf{b}}

\def\ve{\mathbf{e}}

\def\vs{\mathbf{s}}

\def\vu{\mathbf{u}}

\def\vx{\mathbf{x}}
\def\vy{\mathbf{y}}
\def\vz{\mathbf{z}}

\def\mB{\mathbf{B}}

\def\mF{\mathbf{F}}
\def\mG{\mathbf{G}}

\def\mI{\mathbf{I}}

\def\mU{\mathbf{U}}
\def\mV{\mathbf{V}}
\def\mW{\mathbf{W}}
\def\mX{\mathbf{X}}
\def\mY{\mathbf{Y}}
\def\mZ{\mathbf{Z}}

\def\calF{{\mathcal{F}}}
\def\calG{{\mathcal{G}}}

\def\calS{{\mathcal{S}}}

\def\bbE{{\mathbb{E}}}

\def\bbN{{\mathbb{N}}}

\def\bbR{{\mathbb{R}}}

\usepackage{mathrsfs}

\DeclareMathOperator*{\argmin}{arg\,min}

\usepackage{tikz}
\usetikzlibrary{positioning,fit,tikzmark,calc,arrows.meta,backgrounds}
\usepackage{pifont}
\usepackage{makecell}

\usepackage{algorithmic}
\usepackage{algorithm}

\DeclareMathOperator{\iid}{i.i.d.}

\DeclareMathOperator{\kl}{KL}
\DeclareMathOperator{\tv}{TV}
\DeclareMathOperator{\ent}{Ent}
\DeclareMathOperator{\PI}{PI}
\DeclareMathOperator{\LSI}{LSI}

\DeclareMathOperator{\var}{Var}
\DeclareMathOperator{\obs}{obs}
\def\d{\,\mathrm{d}}

\def\dt{\,\mathrm{d}t}
\def\ds{\,\mathrm{d}s}

\def\dz{\,\mathrm{d}z}

\def\almc{\mathtt{ALMC}}

\begin{document}

\title[Nonlinear Assimilation via Score-based Sequential Langevin Sampling]{Nonlinear Assimilation via Score-based Sequential Langevin Sampling}

\author[Z. Ding]{Zhao Ding$^{1}$}
\email{zd1998@whu.edu.cn}
\author[C. Duan]{Chenguang Duan$^{1}$}
\email{cgduan.math@whu.edu.cn}
\author[Y. Jiao]{Yuling Jiao$^{1}$}
\email{yulingjiaomath@whu.edu.cn}
\author[J.Z. Yang]{Jerry Zhijian Yang$^{1}$}
\email{zjyang.math@whu.edu.cn}
\author[C. Yuan]{Cheng Yuan$^{1,2}$}
\email{yuancheng@ccnu.edu.cn}
\author[P. Zhang]{Pingwen Zhang$^{1,3}$}
\email{pzhang@pku.edu.cn}
\address{1. Wuhan University, 2. Central China Normal University, 3. Peking University.}


\date{\today}


\begin{abstract}
This paper introduces score-based sequential Langevin sampling (SSLS), a novel approach to nonlinear data assimilation within a recursive Bayesian filtering framework. The proposed method decomposes the assimilation process into alternating prediction and update steps, using dynamic models for state prediction and incorporating observational data via score-based Langevin Monte Carlo during the updates. To overcome inherent challenges in highly non-log-concave posterior sampling, we integrate an annealing strategy into the update mechanism. Theoretically, we establish convergence guarantees for SSLS in total variation (TV) distance, yielding concrete insights into the algorithm's error behavior with respect to key hyperparameters. Crucially, our derived error bounds demonstrate the asymptotic stability of SSLS, guaranteeing that local posterior sampling errors do not accumulate indefinitely over time. Extensive numerical experiments across challenging scenarios, including high-dimensional systems, strong nonlinearity, and sparse observations, highlight the robust performance of the proposed method. Furthermore, SSLS effectively quantifies the uncertainty associated with state estimates, rendering it particularly valuable for reliable error calibration.
\end{abstract}

\maketitle


\section{Introduction}

\par Data assimilation aims to estimate the time-varying latent states given noisy observation data and the state transition dynamics~\cite{Law2015Data,Reich2015Probabilistic,Reich2019Data}. This task is essential in various application scenarios such as weather forecasting~\cite{Katsafados2020Numerical,Evensen2022Data}, digital twin technology~\cite{Thelen2022comprehensive1,Thelen2022comprehensive2}, and mathematical finance~\cite{Bhar2010Stochastic,Rudiger2012Pricing,Elliott2013option}. Despite its importance and widespread application, data assimilation remains a challenging task. The major difficulties in data assimilation lie in the nonlinear nature of both the state transition dynamics and the measurement model, as well as the high dimensionality of the state. Moreover, in practical scenarios, only noisy and sparse observation data are available, introducing further difficulties to the data assimilation. Apart from estimating the latent states, researchers also need to quantify the uncertainties of the estimated states, which is crucial for assessing and calibrating the estimation error~\cite{Sullivan2015Introduction}. These constraints and requirements pose significant challenges for data assimilation.

\par Although various widely-used methods have been proposed for data assimilation, none fully addresses the aforementioned challenges. These approaches generally fall into two categories: variational methods~\cite{Evensen2022Data} and Bayesian filtering~\cite{Sarkka20232023Bayesian}. Variational methods, such as 3D-Var and 4D-Var~\cite{LeDimet1986Variational}, estimate latent states through maximum-a-posteriori inference. In contrast, Bayesian filtering approaches, including the ensemble Kalman filter (EnKF)~\cite{houtekamer1998data} and particle filter (PF)~\cite{Gordon1993Novel,Kitagawa1996Monte,Doucet2001Sequential}, aim to sample from the posterior distribution. Despite their widespread adoption, both categories encounter significant limitations in complex assimilation scenarios. The fundamental rationale behind variational methods and EnKF relies on Gaussian approximations of the prior and likelihood~\cite{Sarkka20232023Bayesian}. Specifically, these methods assume both Gaussian prior and measurement noise, while linearizing the dynamics and measurement models. However, in highly nonlinear assimilation scenarios with non-Gaussian prior and likelihood, the true posterior may deviate substantially from a Gaussian distribution~\cite{Mandel2012convergence}, severely compromising the effectiveness of Gaussian approximations. The PF, while free from linear and Gaussian assumptions, encounters particle degeneracy or impoverishment in high-dimensional settings~\cite{Snyder2008Obstacles,Bengtsson2008Curse,Bickel2008Sharp}. This phenomenon occurs when the number of ensemble particles is limited: with high probability, multiple particles in the ensemble converge to identical values~\cite{Snyder2008Obstacles}. 

\par Recently, score-based generative models~\cite{Ho2020Denoising,Song2019Generative,song2021score} have emerged as a promising approach in data assimilation~\cite{Rozet2023score,li2024state,Bao2024score,si2024latent}, driven by their exceptional ability to learn and sample from complex distributions. While these approaches demonstrate encouraging empirical performance in certain nonlinear and high-dimensional problems, they face two key limitations: they either depend on restrictive assumptions about Gaussian priors and likelihoods~\cite{li2024state}, or they lack rigorous theoretical foundations~\cite{Rozet2023score,Bao2024score,si2024latent}. A detailed discussion of these limitations is provided in Section~\ref{section:related}.

\par In this work, we introduce a provable method for nonlinear and high-dimensional data assimilation that is both empirically validated and theoretically rigorous. Our main contributions are summarized as follows:
\begin{enumerate}[(i)]
\item We present a novel method for nonlinear assimilation, named \underline{\textbf{s}}core-based \underline{\textbf{s}}equential \underline{\textbf{L}}angevin \underline{\textbf{s}}ampling (SSLS), within a recursive Bayesian filtering framework. SSLS decomposes the assimilation process into a sequence of iterations invoking prediction and update steps. During the prediction step, we utilize the dynamics model to predict states, from which the score of the prior distribution can be learned. Subsequently, in the update step, we incorporate the observation data as the likelihood and sample from the posterior distribution using the score-based Langevin Monte Carlo. To improve convergence and facilitate multi-modal sampling, an annealing strategy is integrated into the Langevin algorithm.
\item We analyze the convergence of SSLS in total variation distance under mild conditions. 
Our theoretical results precisely characterize how the assimilation error depends on key hyperparameters, including the step size, the number of Langevin iterations, and 
the score matching tolerance, and provide explicit guidance for their selection (Theorem~\ref{theorem:section:convergence:assimilation}). A notable consequence is the long-horizon stability of SSLS: although the sequential assimilation error may grow over time, it remains uniformly bounded over all assimilation steps. We also establish a convergence guarantee for score-based Langevin sampling applied to posterior estimation (Theorem~\ref{theorem:section:convergence}), a result that is of independent interest beyond the data assimilation setting.
\item We utilize SSLS in various numerical examples to assess its performance and compare it with baseline approaches from multiple perspectives. According to our experimental results, SSLS yields significant advantages in high-dimensional and nonlinear data assimilation, even with only sparse observations. Furthermore, the standard deviation of SSLS accurately indicates estimation errors, highlighting the proficiency of our method in quantifying uncertainty.
\end{enumerate}

\subsection{Notations}\label{section:notations}
We now introduce some basic notations. The set of positive integers is denoted by $\bbN=\{1,2,\ldots\}$. Denote $\bbN_{0}=\{0\}\cup\bbN$ for convenience. For a positive integer $k\in\bbN$, let $[k]$ denote the set $\{1,\ldots,k\}$. We employ the notations $A\lesssim B$ and $B\gtrsim A$ to signify that there exists an absolute constant $c>0$ such that $A\leq cB$. In addition, $A\asymp B$ means both $A\lesssim B$ and $A\gtrsim B$. Denote by $\gamma_{d}(\cdot;\bm{\mu},\bm{\Sigma})$ the density of a $d$-dimensional Gaussian distribution $\mathcal{N}(\bm{\mu},\bm{\Sigma})$. Appendix~\ref{appendix:notations} summarizes the notations used in Sections~\ref{section:method} and~\ref{section:convergence} for easy reference and cross-checking.

\subsection{Organization}
The rest of this paper is organized as follows. Section~\ref{section:method} presents the score-based sequential Langevin sampling for data assimilation, while Section~\ref{section:convergence} provides a thorough theoretical guarantee. The efficiency of our methods is demonstrated through a series of numerical experiments in Section~\ref{section:experiments}. The related works are reviewed in Section~\ref{section:related}. Finally, Section~\ref{section:conclusion} summarizes the conclusions and outlines future work. The supplementary material provides a review of existing approaches, a notation summary, complete theoretical proofs, additional numerical experiments, and detailed experimental settings.


\section{Score-based Sequential Langevin Sampling}
\label{section:method}

\par This section begins with an introduction to data assimilation in Section~\ref{section:method:setup}, followed by a presentation of the recursive Bayesian filtering framework in Section~\ref{section:method:Bayes}. Then Sections~\ref{section:method:predict} and~\ref{section:method:update} propose the prediction and update procedures, respectively. The complete assimilation algorithm is summarized in Section~\ref{section:method:procedure}.

\subsection{Problem formulation}
\label{section:method:setup}

\par The data assimilation refers to a class of problems that aim to estimate the state of a time-varying system that is indirectly
observed through noisy measurements. Let $(\mX_{k})_{k\in\bbN}$ be a sequence of unobservable latent states taking values in $\bbR^{d}$, which satisfies the dynamics model
\begin{equation}\label{eq:dynamic}
\mX_{k+1}=\calF_{k}(\mX_{k},\mV_{k}).
\end{equation}
Here $k\in\bbN$ is the time index, $\calF_{k}$ is a time-dependent forward propagation operator, and $(\mV_{k})_{k\in\bbN}$ is a sequence of independent random variables with known distribution. The dynamics model~\eqref{eq:dynamic} implies that $(\mX_{k})_{k\in\bbN}$ is a non-homogeneous Markov chain, defined in terms of the transition probability density $\rho_{k}(\vx|\vx_{k}):=p_{\mX_{k+1}|\mX_{k}}(\vx|\vx_{k})$. The stochastic process $(\mY_{k})_{k\in\bbN}$ represents the indirect and noisy observations, linked with the latent states $(\mX_{k})_{k\in\bbN}$ by the measurement model 
\begin{equation}\label{eq:measurement}
\mY_{k}=\calG_{k}(\mX_{k},\mW_{k}),
\end{equation}
where $\calG_{k}$ is a time-dependent measurement operator, and $(\mW_{k})_{k\in\bbN}$ is a sequence of independent noise with known distribution. Denote the conditional density associated with the measurement model~\eqref{eq:measurement} by $g_{k}(\vy|\vx):=p_{\mY_{k}|\mX_{k}}(\vy|\vx)$, which is known as the measurement likelihood. The dynamics model~\eqref{eq:dynamic} together with the measurement model~\eqref{eq:measurement} characterize a state-space model illustrated in Figure~\ref{fig:model:assimilation}.

\begin{figure}[htbp]
\vspace*{10pt}
\centering 
\begin{tikzpicture}[
NodeState/.style={circle, draw=teal!50, fill=teal!5, thick, minimum size=13mm},
NodeObservation/.style={circle, draw=purple!50, fill=purple!5, thick, minimum size=13mm},
NodeBlank/.style={circle, draw=white, fill=white, thick, minimum size=13mm},
]
\node[NodeState](upper1){$\mX_{1}$};
\node[NodeBlank](upper2)[right=40pt of upper1]{$\cdots$};
\node[NodeState](upper3)[right=40pt of upper2]{$\mX_{k}$};
\node[NodeState](upper4)[right=40pt of upper3]{$\mX_{k+1}$};
\node[NodeBlank](upper5)[right=40pt of upper4]{$\cdots$};
\node[NodeObservation](lower1)[below=20pt of upper1]{$\mY_{1}$};
\node[NodeBlank](lower2)[below=20pt of upper2]{$\cdots$};
\node[NodeObservation](lower3)[below=20pt of upper3]{$\mY_{k}$};
\node[NodeObservation](lower4)[below=20pt of upper4]{$\mY_{k+1}$};
\node[NodeBlank](lower5)[below=20pt of upper5]{$\cdots$};
\node(state)[left=5pt of upper1]{Latent States};
\node(measurement)[left=5pt of lower1]{Observations};
\draw [->,line width=1pt](upper1)--node[above]{$\rho_{1}$}(upper2);
\draw [->,line width=1pt](upper2)--node[above]{$\rho_{k-1}$}(upper3);
\draw [->,line width=1pt](upper3)--node[above]{\tikzmarknode{transition}{$\rho_{k}$}}(upper4);
\draw [->,line width=1pt](upper4)--node[above]{$\rho_{k+1}$}(upper5);
\draw [->,line width=1pt](upper1)--node[left]{$g_{1}$}(lower1);
\draw [->,line width=1pt](upper3)--node[left]{\tikzmarknode{measurement}{$g_{k}$}}(lower3);
\draw [->,line width=1pt](upper4)--node[left]{$g_{k+1}$}(lower4);
\end{tikzpicture}
\begin{tikzpicture}[overlay,remember picture,nodes={align=left,inner ysep=1pt},<-]
\path ([yshift=5mm]transition.north)node[anchor=south east,color=black!75](commenttrans){\itshape state transition density};
\draw [<-,line width=0.5pt,black!75](transition.north)|-([color=black!75]commenttrans.south west);
\path ([xshift=-5mm]measurement.west)node[anchor=south east,color=black!75](commentmea){\itshape measurement likelihood};
\draw [<-,line width=0.5pt,black!75](measurement.west)--([color=black!75]commentmea.south west);
\end{tikzpicture}
\caption{An illustrative schematic of the state-space model. The latent states $(\mX_{k})_{k\in\bbN}$ are unobservable and evolves according to known transition densities $(\rho_{k})_{k\in\bbN}$, which are specified by a dynamics model~\eqref{eq:dynamic}. The observations $(\mY_{k})_{k\in\bbN}$ are linked with states by a known likelihood $g_{k}$ characterized by the measurement model~\eqref{eq:measurement}.}
\label{fig:model:assimilation}
\end{figure}

\par  The goal of the data assimilation is to estimate the posterior distribution of the latent state $\mX_{k+1}$ conditioned on all available observations $\mY_{[k+1]}$, that is, 
\begin{equation}\label{eq:assimilation}
\pi_{k+1}(\vx|\vy_{[k+1]}):=p_{\mX_{k+1}|\mY_{[k+1]}}(\vx|\vy_{[k+1]}), \quad k\in\bbN, \quad \vx\in\bbR^{d}.
\end{equation}
In practical scenarios, researchers predominantly focus on posterior sampling rather than posterior density estimation. This preference arises because posterior sampling provides direct access to statistical inference through the computation of essential measures such as means, standard deviations, and confidence intervals, which are crucial for decision-making and uncertainty quantification. Consequently, data assimilation can be reformulated as a sequence of posterior sampling problems.

\subsection{Recursive Bayesian filtering framework}
\label{section:method:Bayes}

\par In this subsection, we present the recursive Bayesian filtering framework~\cite{Sarkka20232023Bayesian} for data assimilation. Given the previous posterior distribution $\pi_{k}(\cdot|\vy_{[k]})$, the current state $\mX_{k+1}$ can be predicted using the dynamics model~\eqref{eq:dynamic}. The distribution of the predicted state given all historical measurements is given by
\begin{equation}\label{eq:prediction:distribution}
q_{k+1}(\vx|\vy_{[k]}):=p_{\mX_{k+1}|\mY_{[k]}}(\vx|\vy_{[k]})=\int\rho_{k}(\vx|\vx_{k})\pi_{k}(\vx_{k}|\vy_{[k]})\d\vx_{k}, \quad \vx\in\bbR^{d},
\end{equation}
where the Chapman-Kolmogorov identity is applied. The posterior distribution in~\eqref{eq:assimilation} can be expressed as the product of the measurement likelihood $g_{k+1}(\vy_{k+1}|\cdot)$~\eqref{eq:measurement} and the prediction distribution $q_{k+1}(\cdot|\vy_{[k]})$ via the Bayes' rule:
\begin{equation}\label{eq:bayes}
\pi_{k+1}(\vx|\vy_{[k+1]})\propto g_{k+1}(\vy_{k+1}|\vx)q_{k+1}(\vx|\vy_{[k]}), \quad \vx\in\bbR^{d},
\end{equation}
where we omit a constant independent of $\vx$. The prediction~\eqref{eq:prediction:distribution} and update~\eqref{eq:bayes} stages can be combined to characterize a recursion from the previous posterior $\pi_{k}(\cdot|\vy_{[k]})$ to the current posterior $\pi_{k+1}(\cdot|\vy_{[k+1]})$ as
\begin{equation}\label{eq:prediction:recursion}
\pi_{k+1}(\vx|\vy_{[k+1]})\propto g_{k+1}(\vy_{k+1}|\vx)\int\rho_{k}(\vx|\vx_{k})\pi_{k}(\vx_{k}|\vy_{[k]})\d\vx_{k}, \quad \vx\in\bbR^{d}.
\end{equation}
This recursion serves as the central object throughout our method, which enables us to decompose the data assimilation into a sequence of posterior sampling problems. Each of these sub-problems can be solved by alternating between prediction~\eqref{eq:prediction:distribution} and update~\eqref{eq:bayes} steps. The complete procedure of the prediction-update recursion is illustrated in Figure~\ref{fig:assimilation:ALMC}. We will present these two steps in detail as Sections~\ref{section:method:predict} and~\ref{section:method:update}, respectively. 

\begin{figure}[htbp]
\centering 
\begin{tikzpicture}[
NodePurple/.style={rectangle, draw=purple!50, fill=purple!5, thick, minimum width=65mm, minimum height=5mm},
NodeBlack/.style={rectangle, draw=black!50, fill=black!5, thick, minimum width=65mm, minimum height=5mm},
NodeTeal/.style={rectangle, draw=teal!50, fill=teal!5, thick, minimum width=65mm, minimum height=5mm},
NodeComment/.style={minimum size=8mm}]
\node[NodePurple](prediction1){$\what{\mX}_{k}^{1},\ldots,\what{\mX}_{k}^{n}\sim\what{\pi}_{k}(\cdot|\vy_{[k]})$};
\node[NodePurple](prediction2)[below=30pt of prediction1]{$\underline{\mX}_{k+1}^{1},\ldots,\underline{\mX}_{k+1}^{n}\sim\what{q}_{k+1}(\cdot|\vy_{[k]})$};
\node[NodePurple](prediction3)[below=30pt of prediction2]{\tikzmarknode{prior}{$\what{\vs}_{k+1}(\cdot,\vy_{[k]})\approx\nabla_{\vx}\log\what{q}_{k+1}(\cdot|\vy_{[k]})$}};
\node[NodeBlack](likelihood)[right=30pt of prediction3]{\tikzmarknode{loglikelihood}{$\nabla_{\vx}\log g_{k+1}(\vy_{k+1}|\cdot)$}};
\node[NodeTeal](update1)[right=30pt of prediction2]{\tikzmarknode{posterior}{$\what{\vb}_{k+1}(\cdot,\vy_{[k+1]})\approx\nabla_{\vx}\log\what{\pi}_{k+1}(\cdot|\vy_{[k+1]})$}};
\node[NodeTeal](update2)[right=30pt of prediction1]{$\what{\mX}_{k+1}^{1},\ldots,\what{\mX}_{k+1}^{n}\sim\what{\pi}_{k+1}(\cdot|\vy_{[k+1]})$};
\node[NodeComment](prediction)[above=0pt of prediction1]{\textcolor{purple!80}{Prediction}};
\node[NodeComment](update)[above=0pt of update2]{\textcolor{teal!80}{Update}};
\node[draw,dashed,draw=purple!80, thick,fit={(prediction1) (prediction2) (prediction3) (prediction)},minimum width=70mm](){};
\node[draw,dashed,teal!80, thick,fit={(update1) (update2) (update)},minimum width=70mm](){};
\draw [->,black,line width=1.5pt,align=right](prediction1.south)--node[left]{dynamics \\ model}(prediction2.north);
\draw [->,black,line width=1.5pt,align=right](prediction2.south)--node[left]{score \\ matching}(prediction3.north);
\draw [->,black,line width=1.5pt,align=left](update1.north)--node[right]{Langevin \\ sampling}(update2.south);
\draw [->,black,line width=1.5pt,align=left](likelihood.north)--node[right]{Bayes' rule}(update1.south);
\draw [-,black,line width=1.5pt,align=left]([xshift=15mm]prediction3.north)|-($(likelihood.north)!0.5!(update1.south)$);
\draw [->,black,dashed,line width=1.5pt](update2.west)--(prediction1.east);
\end{tikzpicture}
\begin{tikzpicture}[overlay,remember picture,nodes={align=left,inner ysep=1pt},<-]
\path ([yshift=-5mm,xshift=-18mm]prior.south)node[anchor=north west,color=black!75](commentprior){\itshape prediction score};
\draw [<-,line width=0.5pt,black!75]([xshift=-18mm,yshift=-1mm]prior.south)|-(commentprior.south east);
\path ([yshift=-5mm]loglikelihood.south)node[anchor=north east,color=black!75](commentloglikelihood){\itshape gradient of the log-likelihood};
\draw [<-,line width=0.5pt,black!75]([yshift=-1mm]loglikelihood.south)|-(commentloglikelihood.south west);
\path ([yshift=5mm,xshift=-20mm]posterior.north)node[anchor=south east,color=black!75](commentposterior){\itshape posterior score};
\draw [<-,line width=0.5pt,black!75]([xshift=-20mm,yshift=1mm]posterior.north)|-(commentposterior.south west);
\end{tikzpicture}
\vspace*{15pt}
\caption{Schematic representation of score-based sequential Langevin sampling. (Left) The prediction step involves sampling from the approximated prediction distribution and estimating the prediction score. (Right) The posterior score is then obtained by combining the prediction score with the gradient of the log-likelihood. The update step samples from the posterior distribution using ALMC. Combining these two phases characterizes a recursion from the previous posterior  to the current posterior.}
\label{fig:assimilation:ALMC}
\end{figure}

\subsection{Prediction and score matching}
\label{section:method:predict}

\par This section focuses on estimating the score, i.e., the gradient of log-density, of the prediction distribution $q_{k+1}(\cdot|\vy_{[k]})$~\eqref{eq:prediction:distribution}. The prediction score estimator will be utilized in the update step for sampling through Langevin-type algorithms, as demonstrated in the subsequent subsection.

\par Given that the exact previous posterior distribution $\pi_{k}(\cdot|\vy_{[k]})$ in~\eqref{eq:prediction:distribution} is intractable, and only an estimator $\what{\pi}_{k}(\cdot|\vy_{[k]})$ is available within the recursive Bayesian filtering framework, we substitute the exact posterior distribution in~\eqref{eq:prediction:distribution} with its estimator to derive the approximated prediction distribution:
\begin{equation}\label{eq:prediction:distribution:hat}
\what{q}_{k+1}(\vx|\vy_{[k]}):=\int\rho_{k}(\vx|\vx_{k})\what{\pi}_{k}(\vx_{k}|\vy_{[k]})\d\vx_{k}\approx q_{k+1}(\vx|\vy_{[k]}), \quad \vx\in\bbR^{d}.
\end{equation}
This approximation closely resembles the prediction distribution~\eqref{eq:prediction:distribution} when $\what{\pi}_{k}(\cdot|\vy_{[k]})$ provides an accurate approximation of the previous posterior distribution $\pi_{k}(\cdot|\vy_{[k]})$. The error of this approximation is analyzed in Theorem~\ref{theorem:section:convergence}. Our task thus becomes estimating the score for the approximated prediction distribution~\eqref{eq:prediction:distribution:hat}.

\par According to the dynamics model~\eqref{eq:dynamic}, a particle approximation to the approximated prediction distribution~\eqref{eq:prediction:distribution:hat} can be constructed as
\begin{equation}\label{eq:predict}
\underline{\mX}_{k+1}^{i}=\calF_{k}(\what{\mX}_{k}^{i},\mV_{k}^{i}), \quad 1\leq i\leq n,
\end{equation}
where $\what{\mX}_{k}^{1},\ldots,\what{\mX}_{k}^{n}$ are independent random variables drawn from the previous estimated posterior $\what{\pi}_{k}(\cdot|\vy_{[k]})$, and $\mV_{k}^{1},\ldots,\mV_{k}^{n}$ are independent random copies of $\mV_{k}$. However, in regions where the approximated prediction density is low, score matching using the particles~\eqref{eq:predict} fails to accurately estimate the score due to insufficient prediction samples~\cite{Song2019Generative}.

\par\textbf{Gaussian smoothing.} 
To address this limitation, we incorporate Gaussian smoothing into the score matching procedure, building upon the approach developed by~\cite{Song2019Generative}. For a fixed smoothing level $\sigma>0$, define the Gaussian smoothed counterpart of~\eqref{eq:prediction:distribution:hat} as 
\begin{equation}\label{eq:smooth:density}
q_{k+1}^{\sigma}(\vx|\vy_{[k]})=\int\gamma_{d}(\vx;\vx_{0},\sigma^{2}\mI_{d})\what{q}_{k+1}(\vx_{0}|\vy_{[k]})\d\vx_{0}\approx\what{q}_{k+1}(\vx|\vy_{[k]}), \quad \vx\in\bbR^{d},
\end{equation}
where $\gamma_{d}(\cdot;\vx_{0},\sigma^{2}\mI_{d})$ represents the density of a $d$-dimensional Gaussian distribution with mean $\vx_{0}$ and covariance matrix $\sigma^{2}\mI_{d}$. The Gaussian smoothing serves two important purposes. First, it fills in low density regions in the original approximated prediction distribution~\eqref{eq:prediction:distribution:hat}, making the estimation of the smoothed density~\eqref{eq:smooth:density} more tractable than estimating the original distribution~\cite{Song2019Generative}. Second, for sufficiently small $\sigma>0$, the score function of the smoothed distribution approximates that of the original prediction density $\what{q}_{k+1}(\cdot|\vy_{[k]})$~\cite{Tang2024Adaptivity}. Consequently, the score function of~\eqref{eq:smooth:density} serves as an effective surrogate that closely approximates the original score while being easier to estimate. The error of Gaussian smoothing has been investigated by~\cite[Theorem 1]{Tang2024Adaptivity}. We show the empirical effectiveness of the Gaussian smoothing in Appendix~\ref{appendix:experiment:sensitivity}.

\begin{remark}[Inflation]
The Gaussian smoothing~\eqref{eq:smooth:density} is commonly known as inflation in the field of data assimilation, and has demonstrated empirical success in practical applications~\cite{Anderson1999Monte,Anderson2009Spatially,Sacher2008Sampling,Evensen2022Data}. Inflation serves a main purpose to mitigate the excessive reduction of variance resulting from spurious correlations in the update step. 
\end{remark}

\par\textbf{Denoising score matching.}
Three mainstream approaches exist for score matching: implicit score matching~\cite{hyvarinen2005Estimation}, sliced score matching~\cite{song2020Sliced}, and denoising score matching~\cite{Vincent2011Connection}. Among these, we adopt denoising score matching because it eliminates the need to compute the gradient of the score network, unlike the other two methods which require this computation.

\par Following denoising score matching~\cite{Vincent2011Connection}, the score function of the smoothed density~\eqref{eq:smooth:density} minimizes the objective functional:
\begin{equation*}
L_{k+1}(\vs)=\bbE_{\underline{\mX}_{k+1}\sim\what{q}_{k+1}(\cdot|\vy_{[k]})}\bbE_{\vepsilon\sim\mathcal{N}(\bm{0},\mI_{d})}\big[\|\sigma\vs(\underline{\mX}_{k+1}+\sigma\vepsilon,\vy_{[k]})+\vepsilon\|_{2}^{2}\big].
\end{equation*}
Since this population risk is analytically intractable in practical applications, we estimate the score function through empirical risk minimization:
\begin{equation}\label{eq:dsm}
\what{\vs}_{k+1}(\cdot,\vy_{[k]})\in\argmin_{\vs\in\calS}\what{L}_{k+1}(\vs)=\frac{1}{n}\sum_{i=1}^{n}\|\sigma\vs(\underline{\mX}_{k+1}^{i}+\sigma\vepsilon_{i},\vy_{[k]})+\vepsilon_{i}\|_{2}^{2},
\end{equation}
where $\calS$ is a deep neural network class, $\{\underline{\mX}_{k+1}^{i}\}_{i=1}^{n}$ is a set of independent predicted states defined as~\eqref{eq:predict}, and $\{\vepsilon_{i}\}_{i=1}^{n}$ is a set of independent standard Gaussian variables. 

\subsection{Update via Langevin Sampling}
\label{section:method:update}

\par This subsection introduces a Langevin algorithm to sample from the posterior distribution $\pi_{k+1}(\cdot|\vy_{[k+1]})$ in~\eqref{eq:bayes}. We begin by introducing the Langevin diffusion associated with the target posterior distribution, defined as the solution to the following stochastic differential equation:
\begin{equation}\label{eq:method:update:LD}
\d\mZ_{t}=\nabla_{\vx}\log\pi_{k+1}(\mZ_{t}|\vy_{[k+1]})\dt+\sqrt{2}\d\mB_{t},  \quad \mZ_{0}\sim q_{k+1}(\cdot|\vy_{[k]}),
\end{equation}
where $(\mB_{t})_{t\geq 0}$ is a Brownian motion. Classical theory establishes that when $\pi_{k+1}(\cdot|\vy_{[k+1]})$ satisfies a functional inequality such as the log-Sobolev inequality, the law of Langevin diffusion \eqref{eq:method:update:LD} converges exponentially fast to the target distribution $\pi_{k+1}(\cdot|\vy_{[k+1]})$~\cite{Bakr2014Analysis}. By applying the Bayes' rule~\eqref{eq:bayes}, we can approximate the drift term of the Langevin diffusion~\eqref{eq:method:update:LD} as:
\begin{equation*}
\nabla_{\vx}\log\pi_{k+1}(\vx|\vy_{[k+1]})\approx\nabla_{\vx}\log g_{k+1}(\vy_{k+1}|\vx)+\what{\vs}_{k+1}(\vx,\vy_{[k]}), \quad \vx\in\bbR^{d},
\end{equation*}
where the second term represents the score estimated in the prediction step \eqref{eq:dsm}. Sampling from the posterior distribution requires simulating the Langevin diffusion with this estimated score. However, in most cases, the Langevin diffusion cannot be simulated analytically. We employ the Euler-Maruyama discretization to approximate the Langevin diffusion~\eqref{eq:method:update:LD} with the estimated score, leading to the Langevin Monte Carlo (LMC).

\par\textbf{Limitations of vanilla Langevin Monte Carlo.}
Notice that the initial distribution of Langevin diffusion \eqref{eq:method:update:LD} is chosen as the prediction distribution $q_{k+1}(\cdot|\vy_{[k]})$, which can be practically implemented using the approximated prediction distribution~\eqref{eq:prediction:distribution:hat}. This choice is necessitated by the fact that the prediction distribution represents our only available knowledge about the state variables $\mX_{k}$. However, this initialization strategy may become inefficient when there is a substantial discrepancy between the target posterior distribution $\pi_{k+1}(\cdot|\vy_{[k+1]})$ and the prediction distribution $q_{k+1}(\cdot|\vy_{[k]})$. 

Specifically, regions of high prediction density may not coincide with regions of high posterior density, particularly when the likelihood is highly informative or concentrated in regions where the prediction density is low. This misalignment creates two significant limitations for the vanilla LMC: it wastes computational resources exploring regions with high prediction density but low likelihood, and it may fail to locate important regions of the posterior distribution where the prediction density is low but the likelihood is high.

\par\textbf{Annealing strategy.} 
To overcome these limitations, we incorporate an annealing strategy into the Langevin algorithm. The rationale behind annealing involves gradually transitioning from the prediction distribution to the target posterior distribution~\cite{DelMoral2006Sequential,Kantas2014Sequential,Beskos2015Sequential,Brosse2018Normalizing,Song2019Generative,Ge2020Estimating,Jalal2021Robust,wu2024annealing}. Specifically, we construct a sequence of interpolations between these two distributions
\begin{equation}\label{eq:interpolation}
\pi_{k+1}^{m}(\vx|\vy_{[k+1]})\propto\pi_{k+1}(\vx|\vy_{[k+1]})^{\beta_{m}}q_{k+1}(\vx|\vy_{[k]})^{1-\beta_{m}}, \quad 0\leq m\leq M, \quad \vx\in\bbR^{d},
\end{equation}
where $0\equiv\beta_{0}<\beta_{1}<\cdots<\beta_{M}\equiv 1$ represents a sequence of inverse temperatures. 
Here $\beta_{0}=0$ corresponds to the prediction distribution, while $\beta_{M}=1$ corresponds to the target posterior. When $\beta_{m}$ is small, the intermediate distribution $\pi_{k+1}^{m}(\cdot|\vy_{[k+1]})$ is predominantly influenced by the prediction distribution, enabling efficient sampling via LMC initialized from the prediction distribution. As $\beta_{m}$ approaches 1, the intermediate distribution $\pi_{k+1}^{m}(\cdot|\vy_{[k+1]})$ converges to the target posterior distribution. Through this gradual increase in inverse temperatures from $\beta_{0}=0$ to $\beta_{M}=1$, the easily sampleable prediction distribution $q_{k+1}(\cdot|\vy_{[k]})$ smoothly transitions toward the target posterior distribution $\pi_{k+1}(\cdot|\vy_{[k+1]})$. 

For implementation, we introduce a general annealing scheme: \( \beta_m = (m/M)^\rho\) for $0 \leq m \leq M$, where the hyper-parameter $\rho$ controls the temperature distribution.
\begin{itemize}
  \item When $\rho=0$, $\beta_m \equiv 1$, this corresponds to the no annealing version.
  \item When $\rho=1$, this corresponds to the uniform annealing (linearly increase from 0 to 1).
  \item When $\rho>1$, inverse temperatures concentrate near 0.
  \item When $\rho \in (0, 1)$, inverse temperatures concentrate near 1.
\end{itemize}
Appendix \ref{appendix:experiment:sensitivity} presents the empirical influence of different annealing schedule $\rho$.  

\par At each inverse temperature $\beta_{m}$, we sample from the intermediate distribution $\pi_{k+1}^{m}(\cdot|\vy_{[k+1]})$ using the Langevin diffusion 
\begin{equation}\label{eq:method:update:LD:anneal}
\d\mZ_{t}^{m}=\nabla_{\vx}\log\pi_{k+1}^{m}(\mZ_{t}^{m}|\vy_{[k+1]})\dt+\sqrt{2}\d\mB_{t},  \quad \mZ_{0}^{m}\sim\pi_{k+1}^{m-1}(\cdot|\vy_{[k+1]}),
\end{equation}
where the score of the intermediate distribution is given by:
\begin{align*}
\nabla_{\vx}\log\pi_{k+1}^{m}(\vx|\vy_{[k+1]})
&=\beta_{m}\nabla_{\vx}\log\pi_{k+1}(\vx|\vy_{[k+1]})+(1-\beta_{m})\nabla_{\vx}\log q_{k+1}(\vx|\vy_{[k]}) \\
&=\beta_{m}\nabla_{\vx}\log g_{k+1}(\vy_{k+1}|\vx)+\nabla_{\vx}\log q_{k+1}(\vx|\vy_{[k]}), \quad \vx\in\bbR^{d}.
\end{align*}
Based on the construction of the intermediate distributions~\eqref{eq:interpolation}, when consecutive temperatures are sufficiently close, the target distribution $\pi_{k+1}^{m}(\cdot|\vy_{[k+1]})$ and the initial distribution $\pi_{k+1}^{m-1}(\cdot|\vy_{[k+1]})$ exhibit minimal discrepancy, facilitating rapid convergence of the Langevin diffusion~\eqref{eq:method:update:LD:anneal}. Through this annealing strategy, we effectively decompose the challenging posterior sampling problem~\eqref{eq:method:update:LD} into a sequence of more tractable posterior sampling steps~\eqref{eq:method:update:LD:anneal}. Figure~\ref{fig:almc} presents a comparison between the original Langevin algorithm and its annealed variant. Besides, this annealing procedure has been shown to enable effective sampling from multi-modal distributions~\cite{wu2024annealing}.

\begin{figure}[htbp]
\vspace*{15pt}
\centering 
\begin{tikzpicture}[
NodeState/.style={rectangle, thick, minimum width=8mm, minimum height=5mm},
NodeStateBlank/.style={rectangle, draw=white, fill=white, thick, minimum width=5mm, minimum height=5mm},
NodeBlank/.style={rectangle, draw=white, fill=white, thick, minimum width=5mm, minimum height=6mm},
]
\node[NodeState,draw=darkgray!90,fill=lightgray!90](upper1){$q_{k+1}(\cdot|\vy_{[k]})$};
\node[NodeState,draw=darkgray!60,fill=lightgray!60](upper2)[right=27pt of upper1]{$\pi_{k+1}^{1}(\cdot|\vy_{[k+1]})$};
\node[NodeStateBlank,draw=white,fill=white](upper3)[right=10pt of upper2]{$\cdots$};
\node[NodeState,draw=darkgray!35,fill=lightgray!35](upper4)[right=10pt of upper3]{$\pi_{k+1}^{M-1}(\cdot|\vy_{[k+1]})$};
\node[NodeState,draw=darkgray!25,fill=lightgray!25](upper5)[right=27pt of upper4]{$\pi_{k+1}(\cdot|\vy_{[k+1]})$};
\node[NodeBlank](lower2)[below=3pt of upper2]{$\beta_{1}$};
\node[NodeBlank](lower3)[below=3pt of upper3]{$\cdots$};
\node[NodeBlank](lower4)[below=3pt of upper4]{$\beta_{M-1}$};
\node[NodeBlank](lower11)[below=15pt of upper1]{inverse temperatures};
\node[NodeBlank](lower55)[below=15pt of upper5]{};
\draw [->,line width=1.0pt](upper1)--node[above]{LMC}(upper2);
\draw [->,line width=1.0pt](upper2)--(upper3);
\draw [->,line width=1.0pt](upper3)--(upper4);
\draw [->,line width=1.0pt](upper4)--node[above]{LMC}(upper5);
\draw [->,line width=1pt,color=darkgray](lower11)--(lower55);
\node[NodeState,draw=darkgray!90,fill=lightgray!90](langevin1)[above=20pt of upper1]{\tikzmarknode{prior}{$q_{k+1}(\cdot|\vy_{[k]})$}};
\node[NodeState,draw=darkgray!25,fill=lightgray!25](langevin2)[above=20pt of upper5]{\tikzmarknode{posterior}{$\pi_{k+1}(\cdot|\vy_{[k+1]})$}};
\draw [->,line width=1.0pt](langevin1)--node[above]{LMC}(langevin2);
\node[NodeBlank](lmc)[left=3pt of langevin1]{\textbf{LMC}};
\node[NodeBlank](lmc)[left=3pt of upper1]{\textbf{ALMC}};
\end{tikzpicture}
\begin{tikzpicture}[overlay,remember picture,nodes={align=left,inner ysep=1pt},<-]
\path ([yshift=5mm,xshift=-2mm]prior.north)node[anchor=south west,color=black!75](commentprior){\itshape prediction distribution};
\draw [<-,line width=0.5pt,black!75]([xshift=-5mm,yshift=1mm]prior.north)|-(commentprior.south east);
\path ([yshift=5mm,xshift=-10mm]posterior.north)node[anchor=south east,color=black!75](commentposterior){\itshape target posterior};
\draw [<-,line width=0.5pt,black!75]([xshift=-6mm,yshift=1mm]posterior.north)|-(commentposterior.south west);
\end{tikzpicture}
\vspace*{-5pt}
\caption{Schematic comparison of vanilla and annealed Langevin algorithms. (Top) The vanilla Langevin algorithm samples from the target posterior distribution, using the prediction distribution as initialization. (Bottom) The annealed Langevin algorithm employs a sequence of interpolations that smoothly transition from the prediction distribution to the target posterior distribution.}
\label{fig:almc}
\end{figure}

\par Since the Langevin diffusion~\eqref{eq:method:update:LD:anneal} cannot be simulated analytically, we employ the Euler-Maruyama discretization to approximate it, yielding the following sampling scheme:
\begin{equation}\label{eq:method:update:ALMC}
\begin{aligned}
\what{\mZ}_{(\ell+1)h}^{m}&=\what{\mZ}_{\ell h}^{m}+h\what{\vb}_{k+1}^{m}(\what{\mZ}_{\ell h},\vy_{[k+1]})+\sqrt{2h}\vxi_{\ell}^{m}, \quad 0\leq\ell\leq K-1, \\
\what{\mZ}_{0}^{1}&\sim \what{q}_{k+1}(\cdot|\vy_{[k]}), \quad \what{\mZ}_{0}^{m}=\what{\mZ}_{Kh}^{m-1}, \quad 2\leq m\leq M, 
\end{aligned}
\end{equation}
where $h>0$ is the step size, $(\vxi_{\ell}^{m})_{m,\ell}$ is a sequence of independent standard Gaussian variables, and the drift term is a weighted sum of the gradient of log-likelihood and the estimated prediction score function~\eqref{eq:dsm}
\begin{equation}\label{eq:method:update:ALMC:score}
\what{\vb}_{k+1}^{m}(\vx,\vy_{[k+1]})=\beta_{m}\nabla_{\vx}\log g_{k+1}(\vy_{k+1}|\vx)+\what{\vs}_{k+1}(\vx,\vy_{[k]})\approx\nabla_{\vx}\log\pi_{k+1}^{m}(\vx|\vy_{[k+1]}), \quad \vx\in\bbR^{d}.
\end{equation}
The complete procedure for the $(k+1)$-th update step is presented in Algorithm~\ref{alg:ALMC}.

\begin{algorithm}[htbp]
\caption{Update by Annealed Langevin Monte Carlo (ALMC).}
\label{alg:ALMC}
\begin{algorithmic}[1]
\REQUIRE{Predicted samples $\underline{\mX}_{k+1}^{1},\ldots,\underline{\mX}_{k+1}^{n}$, a prediction score estimator $\what{\vs}_{k+1}(\cdot,\vy_{[k]})$, the measurement likelihood $g_{k+1}(\vy_{k+1}|\cdot)$.} \\
\ENSURE{A particle approximation $\what{\mX}_{k+1}^{1},\ldots,\what{\mX}_{k+1}^{n}$ to the posterior $\pi_{k+1}(\cdot|\vy_{[k+1]})$.} \\
\STATE{Set inverse temperatures $0\equiv\beta_{0}<\beta_{1}<\cdots<\beta_{M}\equiv 1$, and a step size $h>0$.}
\STATE{Initialize the particles $\what{\mZ}_{0}^{1,i}\leftarrow\underline{\mX}_{k+1}^{i}$ for each $1\leq i\leq n$.}
\FOR {$m=1,\ldots,M$}
\FOR {$\ell=0,\ldots,K-1$}
\STATE{Sample independent Gaussian noises $\vxi_{\ell}^{m,1},\ldots,\vxi_{\ell}^{m,n}\sim^{\iid}\mathcal{N}(\bm{0},\mI_{d})$.}
\STATE{Compute the estimated posterior score \\ $\what{\vb}_{k+1}^{m}(\cdot,\vy_{[k+1]})\leftarrow\beta_{m}\nabla_{\vx}\log g_{k+1}(\vy_{k+1}|\cdot)+\what{\vs}_{k+1}(\cdot,\vy_{[k]})$.}
\STATE{Update by the LMC, for $1\leq i\leq n$, \\
$\what{\mZ}_{(\ell+1)h}^{m,i}\leftarrow\what{\mZ}_{\ell h}^{m,i}+h\what{\vb}_{k+1}^{m}(\what{\mZ}_{\ell h}^{m,i},\vy_{[k+1]})+\sqrt{2h}\vxi_{\ell}^{m,i}$.}
\ENDFOR
\STATE{Initialize the particles for the next temperature $\what{\mZ}_{0}^{m+1,i}\leftarrow\what{\mZ}_{Kh}^{m,i}$ for $1\leq i\leq n$.}
\ENDFOR
\RETURN{$\what{\mX}_{k+1}^{i}\leftarrow\what{\mZ}_{Kh}^{M,i}$ for $1\leq i\leq n$.}
\end{algorithmic}
\end{algorithm}

\begin{remark}[An alternative annealing strategy]The strategy defined in \eqref{eq:interpolation} can be interpreted as annealing with respect to the likelihood. Specifically, it gradually incorporates observational information to transition from the initial prior density to the posterior distribution. In practice, annealing can also be performed on the posterior as a whole:$$  \pi^m_{k+1}(\mathbf{x}|\mathbf{y}_{[k+1]}) \propto \pi_{k+1}(\mathbf{x}|\mathbf{y}_{[k+1]})^{\beta_m} q(\mathbf{x})^{1-\beta_m},$$where $q(\mathbf{x})$ is a reference distribution, such as the standard Gaussian. Empirical evidence suggests that the choice of strategy depends on the fidelity of the measurement model. If the likelihood provides more accurate information than the prior, annealing on the posterior may be preferable. This approach typically yields faster convergence toward the target density, whereas annealing on the likelihood might inadvertently weaken the guidance provided by accurate observations. Conversely, if the prior is more reliable, annealing on the likelihood remains an ideal option. We present a detailed empirical comparison of these two strategies in Section \ref{section:experiment:doublewell:multi-modal}.\end{remark}

\subsection{Summary of the procedure}
\label{section:method:procedure}

\par Building upon the methods described in Sections~\ref{section:method:predict} and~\ref{section:method:update}, we can sample from the current posterior distribution $\pi_{k+1}(\cdot|\vy_{[k+1]})$ given a particle approximation to the previous posterior distribution $\pi_{k}(\cdot|\vy_{[k]})$.

\par To obtain a particle approximation to the initial posterior distribution $\pi_{1}(\cdot|\vy_{1})$, we apply Bayes' rule, yielding:
\begin{equation}\label{eq:bayes:score:initial}
\nabla_{\vx}\log\pi_{1}(\vx|\vy_{1})=\nabla_{\vx}\log g_{1}(\vy_{1}|\vx)+\nabla_{\vx}\log p_{\mX_{1}}(\vx), \quad \vx\in\bbR^{d},
\end{equation}
where $q_{1}:=p_{\mX_{1}}$ denotes the initial prior distribution. Thus, sampling from the initial posterior distribution $\pi_{1}(\cdot|\vy_{1})$ requires only an estimate of the score of the initial prior distribution $\nabla_{\vx}\log q_{1}$. In practice, one typically has access to a set of samples drawn independently from $q_{1}$. Using these samples, we can estimate the initial prior score $\nabla_{\vx}\log q_{1}$ through Gaussian smoothing and denoising score matching as shown in~\eqref{eq:dsm}, which we denote as $\what{\vs}_{1}$. The complete procedure for score-based sequential Langevin sampling is presented in Algorithm~\ref{alg:recursive:DA:ALMC}.

\begin{algorithm}[htbp]
\caption{Score-based sequential Langevin sampling for data assimilation.}
\label{alg:recursive:DA:ALMC}
\begin{algorithmic}[1]
\REQUIRE{The observations $(\vy_{k})_{k\in\bbN}$, the dynamics model $(\calF_{k})_{k\in\bbN}$, the measurement likelihood $\{g_{k}(\vy_{k}|\cdot)\}_{k\in\bbN}$.} 
\ENSURE{A particle approximation $\what{\mX}_{k+1}^{1},\ldots,\what{\mX}_{k+1}^{n}$ to the distribution $\pi_{k+1}(\cdot|\vy_{[k+1]})$.}
\STATE{\texttt{\# Initial posterior sampling.}}
\STATE{Draw i.i.d. samples from the initial prior distribution: $\underline{\mX}_{1}^{1},\ldots,\underline{\mX}_{1}^{n}\sim^{\iid}q_{1}$.}
\STATE{Estimate the score from $\{\underline{\mX}_{1}^{i}\}_{i=1}^{n}$ by score matching $\what{\vs}_{1}$.}
\STATE{Sample from the posterior distribution $\what{\pi}_{1}(\cdot|\vy_{1})$ by the ALMC (Algorithm~\ref{alg:ALMC}): \\
$\what{\mX}_{1}^{1},\ldots,\what{\mX}_{1}^{n}\leftarrow\almc(\underline{\mX}_{1}^{1},\ldots,\underline{\mX}_{1}^{n},\what{\vs}_{1},g_{1}(\vy_{1}|\cdot))$.}
\STATE{\texttt{\# Recursive posterior sampling.}}
\FOR{$k=1,2,\ldots$}
\STATE{\texttt{\# Prediction step.}}
\STATE{Run the dynamics model: $\underline{\mX}_{k+1}^{i}\leftarrow\calF_{k}(\what{\mX}_{k}^{i},\mV_{k}^{i})$ with $\mV_{k}^{i}\sim p_{\mV}$ for $1\leq i\leq n$.}
\STATE{Estimate the prediction score from $\{\underline{\mX}_{k+1}^{i}\}_{i=1}^{n}$ by score matching $\what{\vs}_{k+1}(\cdot,\vy_{[k]})$.}
\STATE{\texttt{\# Update step.}}
\STATE{Sample from the posterior distribution $\pi_{k+1}(\cdot|\vy_{[k+1]})$ by the ALMC (Algorithm~\ref{alg:ALMC}): \\
$\what{\mX}_{k+1}^{1},\ldots,\what{\mX}_{k+1}^{n}\leftarrow\almc(\underline{\mX}_{k+1}^{1},\ldots,\underline{\mX}_{k+1}^{n},\what{\vs}_{k+1}(\cdot,\vy_{[k]}),g_{k+1}(\vy_{k+1}|\cdot))$.}
\ENDFOR
\end{algorithmic}
\end{algorithm}

\begin{remark}[Computational cost reduction]
The implementation of Algorithm~\ref{alg:recursive:DA:ALMC} requires learning a score network from predicted states at each time step, which introduces substantial computational overhead. However, we demonstrate that in practical applications, one can effectively fine-tune the score network using the current predicted states while initializing it with parameters obtained from the previous time step. This approach eliminates the need for complete network retraining with random initialization, thereby achieving significant computational efficiency.

As an alternative method, one can train a single score network with a shared time parameter to handle new observation trajectories. For example, using massive historical datasets (such as reanalysis data), one could pre-train a global "climate" model. This model would serve as a time- and season-aware prior, ready to be integrated with newly incoming observations. Following this line of thought, recent work \cite{yang2025generative} has successfully explored training a background score network on extensive historical data. In future work, we plan to adopt this shared-parameter paradigm. By simply fine-tuning a globally pre-trained network on real-time measurements, we expect to drastically reduce the training time required for new trajectories.
\end{remark}

\section{Non-asymptotic Convergence Guarantees}
\label{section:convergence}

\par In this section, we present a convergence analysis for the score-based sequential Langevin sampling (SSLS). Our theoretical analysis focuses on the core algorithm of the SSLS:
\begin{equation}\label{eq:section:convergence:SLMC}
\begin{aligned}
\what{\mZ}_{(\ell+1)h}&=\what{\mZ}_{\ell h}+h\what{\vb}_{k+1}(\what{\mZ}_{\ell h},\vy_{[k+1]})+\sqrt{2h}\vxi_{\ell}, \quad 1\leq\ell\leq K-1, \\
\what{\mZ}_{0}&\sim\pi_{k+1}^{0}(\cdot|\vy_{[k+1]}),
\end{aligned}
\end{equation}
where $h>0$ represents the time step, $(\vxi_{\ell})_{\ell}$ is a sequence of independent standard Gaussian variables, and $\pi_{k+1}^{0}(\cdot|\vy_{[k+1]})$ denotes the initial distribution at the $(k+1)$-th update step. While this initial distribution is typically selected as $\what{q}_{k+1}(\cdot|\vy_{[k]})$, our analysis can generalize to any choice of initial distribution. The estimated posterior score takes the form:
\begin{equation}\label{eq:method:update:posterior:score}
\what{\vb}_{k+1}(\vx,\vy_{[k+1]})=\nabla_{\vx}\log g_{k+1}(\vy_{k+1}|\vx)+\what{\vs}_{k+1}(\vx,\vy_{[k]}), \quad \vx\in\bbR^{d},
\end{equation}
where the prediction score $\what{\vs}_{k+1}$ is estimated using score matching~\eqref{eq:dsm}. 

\par Through this analysis, we establish rigorous theoretical guarantees for data assimilation using SSLS and provide theoretical understandings for the benefits of the annealing strategy employed in Section~\ref{section:method:update}.

\subsection{Notations and assumptions}
\label{section:convergence:assumptions}

\par Before proceeding with our analysis, we introduce some notations and assumptions. 

\begin{definition}[Total variation distance]
The total variation (TV) distance between two distributions $\mu$ and $\pi$ is defined as
\begin{equation*}
\|\mu-\pi\|_{\tv}=\frac{1}{2}\int|\mu(\vx)-\pi(\vx)|\d\vx.
\end{equation*}
\end{definition}

\begin{definition}[Chi-squared divergence]
The $\chi^{2}$-divergence between two distributions $\mu$ and $\pi$ is defined as
\begin{equation*}
\chi^{2}(\mu\|\pi)=\int\Big(\frac{\mu(\vx)}{\pi(\vx)}\Big)^{2}\pi(\vx)\d\vx-1=\int\Big(\frac{\mu(\vx)}{\pi(\vx)}-1\Big)^{2}\pi(\vx)\d\vx.
\end{equation*}
\end{definition}

\par Let $\what{\pi}_{k+1}$ denote the law of $\what{\mZ}_{T}$ with $T=Kh$, representing the SSLS estimate of the target posterior distribution. We denote $\varepsilon_{\tv}^{k}$ as the total variation distance between the target posterior distribution and its SSLS estimate:
\begin{equation}\label{eq:section:convergence:TV}
\varepsilon_{\tv}^{k}:=\|\pi_{k}(\cdot|\vy_{[k]})-\what{\pi}_{k}(\cdot|\vy_{[k]})\|_{\tv}, \quad k\in\bbN.
\end{equation}
For a comprehensive list of notations used throughout this section, we refer readers to Appendix~\ref{appendix:notations}.

\par Our analysis relies on the following assumptions on the posterior distributions.

\begin{assumption}[Lipschitz score]
\label{assumption:posterior:smooth}
For each $k\in\bbN$, the posterior score is $\lambda$-Lipschitz on $\bbR^{d}$, that is, for each $\vx_{1},\vx_{2}\in\bbR^{d}$,
\begin{align*}
\|\nabla_{\vx}\log\pi_{k+1}(\vx_{1}|\vy_{[k+1]})-\nabla_{\vx}\log\pi_{k+1}(\vx_{2}|\vy_{[k+1]})\|_{2} &\leq\lambda\|\vx_{1}-\vx_{2}\|_{2}, \\
\|\nabla_{\vx}\log\pi_{1}(\vx_{1})-\nabla_{\vx}\log\pi_{1}(\vx_{2})\|_{2} &\leq\lambda\|\vx_{1}-\vx_{2}\|_{2}
\end{align*}
\end{assumption}

\begin{assumption}[Log-Sobolev inequality]
\label{assumption:LSI:posterior}
For all $k\in\bbN$, the posterior distribution $\pi_{k+1}(\cdot|\vy_{[k+1]})$ satisfies a log-Sobolev inequality with constant $C_{\LSI}\geq 1$, i.e., for each function $f\in C_{0}^{\infty}(\bbR^{d})$,
\begin{equation*}
\ent(f^{2})\leq 2C_{\LSI}\bbE\big[\|\nabla f\|_{2}^{2}\big],
\end{equation*}
where the entropy is defined as $\ent(g)\coloneqq\bbE[g\log g]-\bbE[g]\log\bbE[g]$, and the expectation is taken with respect to the posterior distribution $\pi_{k+1}(\cdot|\vy_{[k+1]})$ and $\pi_{1}$. Further, assume the posterior distributions are centered.
\end{assumption}

\par Assumption~\ref{assumption:posterior:smooth} guarantees the existence and uniqueness of a strong solution to the Langevin diffusion in~\eqref{eq:method:update:LD}. Together, Assumptions~\ref{assumption:posterior:smooth} and~\ref{assumption:LSI:posterior} form the fundamental requirements for the convergence analysis of Langevin-type sampling methods~\cite{Chewi2024Analysis,Chewi2024log,Lee2022Convergence,Tang2024Adaptivity}. LSI is known to hold for a broad and practically relevant class of distributions. Under the Bakry-{\'E}mery theorem~\cite{Bakry1985Diffusions}, any $\beta$-strongly log-concave density satisfies the log-Sobolev inequality with constant $C_{\LSI}=\beta^{-1}$. Thus, Assumption~\ref{assumption:LSI:posterior} is sufficiently broad to accommodate Gaussian distributions, general log-concave distributions, and even certain multi-modal distributions~\cite{Chen2011Dimension}, such as Gaussian mixtures. Consequently, this framework naturally encompasses the setting of Kalman filtering, which operates under the Gaussian assumption. The LSI constant $C_{\LSI}$ captures the intrinsic geometric difficulty of sampling from the target posterior. The LSI constant of a multi-modal distribution can be large, and verifying LSI for a given distribution is itself a non-trivial task.

\begin{assumption}[Boundedness and regularity]\label{assumption:bounded}
There exist a universal constant $B\geq 1$ such that for all time steps $k \in \mathbb{N}$:
\begin{enumerate}[(i)]
\item The transition density and its gradient are uniformly bounded, i.e., for any $\vx\in\mathbb{R}^d$,
\begin{equation*}
\rho_{k}(\vx|\vx_{k}) \leq B \quad \text{and} \quad \|\nabla_{\vx}\rho_{k}(\vx|\vx_{k})\|_{\infty} \leq B.
\end{equation*}
\item The score function of the prediction distribution exhibits at most linear growth, i.e., for any $\vx\in\mathbb{R}^d$,
\begin{equation*}
\|\nabla_{\vx}\log q_{1}(\vx)\|_{2},\|\nabla_{\vx}\log q_{k+1}(\vx|\vy_{[k]})\|_{2} \leq B(1+\|\vx\|_{2}).
\end{equation*}
Further, this linear growth condition applies analogously to $\nabla_{\vx}\log\widehat{q}_{1}(\vx)$, $\widehat{\vs}_{1}(\vx)$, $\nabla_{\vx}\log\widehat{q}_{k+1}(\vx|\vy_{[k]})$, and $\widehat{\vs}_{k+1}(\vx,\vy_{[k]})$.
\end{enumerate}
\end{assumption}

\par Assumption~\ref{assumption:bounded} establishes essential regularity conditions for the state transition, prediction density, and measurement likelihood. Assumption~\ref{assumption:bounded} (i) ensures the transition kernel is non-singular, effectively excluding purely deterministic transitions. Assumption~\ref{assumption:bounded} (ii) constrains the tail behavior of the prediction distribution, ensuring the log-density does not decay faster than a quadratic. This condition is satisfied by a wide class of distributions:

\begin{example}[Gaussian distribution]\label{example:gaussian}
If the prediction density is a Gaussian distribution, i.e., $q_{k+1}(\vx|\vy_{[k]})\coloneqq\gamma_{d}(\vx;\bm{\mu},\bm{\Sigma})$ for some $\bm{\mu}\in\mathbb{R}^{d}$ and $\bm{\Sigma}\succ\bm{0}$, then the score is $\nabla_{\vx}\log q_{k+1}(\vx|\vy_{[k]})=-\bm{\Sigma}^{-1}(\vx - \bm{\mu})$. The linear growth condition holds directly.
\end{example}

\begin{example}[Gaussian mixture]\label{example:gaussian:mixture}
If the prediction density is a Gaussian mixture, i.e., 
\begin{equation*}
q_{k+1}(\vx|\vy_{[k]})\coloneqq\sum_{i=1}^{m}w_{i}\gamma_{d}(\vx;\bm{\mu}_{i},\bm{\Sigma}_{i}),
\end{equation*}
the score function reads 
\begin{equation*}
\nabla\log q_{k+1}(\vx|\vy_{[k]}) = \sum_{i=1}^{m}\frac{w_{i}\gamma_{d}(\vx;\bm{\mu}_{i},\bm{\Sigma}_{i})}{q_{k+1}(\vx|\vy_{[k]})}\nabla\log\gamma_{d}(\vx;\bm{\mu}_{i},\bm{\Sigma}_{i}) \eqqcolon \sum_{i=1}^{m}w_{i}^{\prime}(\vx)\nabla\log\gamma_{d}(\vx;\bm{\mu}_{i},\bm{\Sigma}_{i}),
\end{equation*}
which is a weighted combination of component scores with $w_{i}^{\prime}(\vx)\in(0,1)$. Since each component score is linear, the mixture score retains linear growth, as the tail behavior is dominated by the component with the largest variance in a given direction. See~\cite[Appendix C.1]{chang2025provable} for detailed derivations.
\end{example}

\begin{example}[Gaussian convolution]\label{example:gaussian:convolution}
Let $\nu$ be a probability density with a compact support. If the prediction density is a Gaussian convolution of $\nu$, i.e., there exists a constant $\sigma>0$ such that 
\begin{equation*}
q_{k+1}(\vx|\vy_{[k]})\coloneqq \int\gamma_{d}(\vx;\vx^{\prime},\sigma^{2}\mI_{d})\nu(\vx^{\prime})\d\vx.
\end{equation*}
As demonstrated in~\cite[Proposition 3.2]{ding2024characteristic}, the score of such a density exhibits at most linear growth, regardless of the complexity or singularity of the underlying measure $\nu$. 
\end{example}

\par In essence, Assumption~\ref{assumption:bounded} (ii) constrains the relative decay of the prediction density $q_{k+1}(\vx|\vy_{[k]})$, ensuring it does not vanish arbitrarily quickly as $\|\vx\|_{2}\to\infty$. This condition implies that the predictive density is bounded below by a Gaussian-like tail. This property is formalised in the following proposition.

\begin{proposition}\label{proposition:lower:bound}
Suppose Assumption~\ref{assumption:bounded} (ii) holds. Let $\vx_{*}\in\mathbb{R}^{d}$ be the reference point satisfying $q_{k+1}(\vx_{*}|\vy_{[k]})>0$. Then for any $\vx\in\mathbb{R}^{d}$,
\begin{equation*}
q_{k+1}(\vx|\vy_{[k]})\geq H\exp\Big(-\frac{\|\vx\|_{2}^{2}}{V^{2}}\Big),
\end{equation*}
where 
\begin{equation*}
H \coloneqq \frac{q_{k+1}(\vx_{*}|\vy_{[k]})}{\exp(B(1+3\|\vx_{*}\|_{2}^{2}))}, \quad\text{and}\quad V^{2} \coloneqq \frac{1}{2B}.
\end{equation*}
\end{proposition}

\par The proof of Proposition~\ref{proposition:lower:bound} is provided in Appendix~\ref{appendix:convergence}.

\par Without loss of generality, we assume the approximate prediction density $\widehat{q}_{k+1}(\vx|\vy_{[k]})$ also satisfies this lower bound. In numerical implementations, this can be ensured through techniques such as Gaussian perturbation.

\begin{remark}[Condition number]\label{remark:cn}
As noted by \cite{purohit2024posterior}, the quantity 
\begin{equation}\label{eq:condition:number}
\frac{\sup_{\vx}g_{k+1}(\vy_{k+1}|\vx)}{\int_{\mathbb{R}^d} g_{k+1}(\vy_{k+1}|\vx)q_{k+1}(\vx|\vy_{[k]})\mathrm{d}\vx} \leq \kappa
\end{equation}
quantifies the inherent difficulty of posterior sampling. To illustrate this concept at the $(k+1)$-th step: when the likelihood function $g_{k+1}(\vy_{k+1}|\cdot)$ concentrates on the high-probability support of the predictive distribution $q_{k+1}(\cdot|\vy_{[k]})$, the denominator in \eqref{eq:condition:number} remains bounded away from zero, resulting in a moderate condition number $\kappa$. Conversely, if the likelihood concentrates in a region where the predictive probability is negligible, the denominator approaches zero, leading to a large condition number and indicating numerical ill-posedness. As demonstrated in Theorem~\ref{theorem:section:convergence}, the error of the posterior sampling grows with this condition number. Proposition~\ref{proposition:lower:bound} establishes that the linear growth condition in Assumption~\ref{assumption:bounded} (ii) ensures the density retains sufficient mass across the state space to prevent the denominator from vanishing, thereby providing a safeguard for the condition number.
\end{remark}

\par Finally, to provide some intuations, we provide some concrete examples that simultaneously satisfy Assumptions~\ref{assumption:posterior:smooth},~\ref{assumption:LSI:posterior} and~\ref{assumption:bounded}. All these examples consider linear Gaussian measurement model. Verifying the log-Sobolev inequality for the posterior distribution with a nonlinear Gaussian measurement model remains open.

\begin{example}\label{example:gaussian:posterior}
If the prediction density is a Gaussian distribution as Example~\ref{example:gaussian}, and the measurement likelihood function $g_{k+1}(\vy_{k+1}|\vx_{k+1})$ is Gaussian, then the posterior density $\pi_{k+1}(\vx_{k+1}|\vy_{[k+1]})$ is Gaussian, thus satisfying Assumptions~\ref{assumption:posterior:smooth} and~\ref{assumption:LSI:posterior}.
\end{example}

\begin{example}\label{example:gaussian:mixture:posterior}
If the prediction density is a Gaussian mixture as Example~\ref{example:gaussian:mixture}, and the measurement likelihood function $g_{k+1}(\vy_{k+1}|\vx_{k+1})$ is Gaussian, then the posterior density $\pi_{k+1}(\vx_{k+1}|\vy_{[k+1]})$ is also a Gaussian mixture, thus satisfying Assumptions~\ref{assumption:LSI:posterior} directly. According to~\cite[Appendix C1]{chang2025provable}, the score of the Gaussian mixture exhibits a uniform Lipschitz constant, thus Assumption~\ref{assumption:posterior:smooth} holds.
\end{example}

\begin{example}\label{example:gaussian:convolution:posterior}
\par Consider the case where the predictive density is a Gaussian convolution as in Example~\ref{example:gaussian:convolution}, and the measurement likelihood $g_{k+1}(\vy_{k+1}|\vx_{k+1})$ is Gaussian. As demonstrated in \cite[Proposition 3.5]{ding2024characteristic}, the score of the predictive density admits a uniform Lipschitz constant. Furthermore, since $\nabla\log g_{k+1}(\vy_{k+1}|\cdot)$ is also uniformly Lipschitz, the resulting posterior satisfies the regularity requirements of Assumption~\ref{assumption:posterior:smooth}. To establish the log-Sobolev inequality, we examine the curvature of the posterior potential. Let the latent measure $\nu$ be compactly supported such that $\mathrm{supp}(\nu) \subseteq \{\vx \in \mathbb{R}^d : \|\vx\|_2 \leq R\}$. Utilizing a variant of Tweedie's formula \cite{Efron2011Tweedie}, it has been shown in \cite{Grenioux2024Stochastic} that the Hessian of the log-predictive density satisfies:
\begin{equation}\label{eq:tweedie:lower:bound}
-\nabla^2 \log q_{k+1}(\vx|\vy_{[k]}) \succeq \frac{\sigma^2 - dR^2}{\sigma^4} \mathbf{I}_d.
\end{equation}
Note that while the right-hand side of \eqref{eq:tweedie:lower:bound} may be negative, the predictive density is at most $1/\sigma^2$-semi-concave. Now, consider a Gaussian likelihood $g_{k+1}(\vy_{k+1}|\vx_{k+1}) \coloneqq \gamma_{d}(\vy_{k+1}; \vx_{k+1}, V^2 \mathbf{I}_d)$. Then
\begin{align*}
-\nabla^2 \log \pi_{k+1}(\vx_{k+1}|\vy_{[k+1]}) &= -\nabla^2 \log q_{k+1}(\vx_{k+1}|\vy_{[k]}) - \nabla^2 \log g_{k+1}(\vy_{k+1}|\vx_{k+1}) \\
&\succeq \Big( \frac{\sigma^2 - dR^2}{\sigma^4} + \frac{1}{V^2} \Big) \mathbf{I}_d.
\end{align*}
Consequently, if the measurement noise $V^2$ is small enough, the posterior becomes strongly log-concave. By the Bakry-{\'E}mery theorem~\cite{Bakry1985Diffusions}, such a posterior satisfies the log-Sobolev inequality.
\end{example}

\par We next introduce a ``black-box'' assumption on score matching~\eqref{eq:dsm}. 
\begin{assumption}[Error of score matching]
\label{assumption:sm}
There exists a score matching tolerance $\Delta\in(0,1)$ such that
\begin{align*}
\bbE_{\underline{\mX}_{1}}\big[\|\nabla_{\vx}\log\what{q}_{1}(\underline{\mX}_{1})-\what{\vs}_{1}(\underline{\mX}_{1})\|_{2}^{2}\big]&\leq\Delta^{2}, \\
\bbE_{\underline{\mX}_{k+1}}\big[\|\nabla_{\vx}\log\what{q}_{k+1}(\underline{\mX}_{k+1}|\vy_{[k]})-\what{\vs}_{k+1}(\underline{\mX}_{k+1},\vy_{[k]})\|_{2}^{2}\big]&\leq\Delta^{2}, 
\end{align*}
for each $k\in\bbN$. Here the expectation $\bbE_{\underline{\mX}_{1}}[\cdot]$ is taken with respect to $\underline{\mX}_{1}\sim\what{q}_{1}$, and the expectation $\bbE_{\underline{\mX}_{k+1}}[\cdot]$ is taken with respect to $\underline{\mX}_{k+1}\sim\what{q}_{k+1}(\cdot|\vy_{[k]})$.
\end{assumption}

\par Assumption~\ref{assumption:sm} requires the $L^{2}$-error of prediction score estimator $\what{\vs}_{k+1}$~\eqref{eq:dsm} to be sufficiently small, where the error is measured with respect to the approximated prediction distribution $\what{q}_{k+1}(\cdot|\vy_{[k]})$. While this assumption could be substituted with explicit score matching bounds, we maintain this formulation for clarity of presentation. Specifically, some standard techniques of non-parametric regression using deep neural networks~\cite{schmidt2020nonparametric,kohler2021rate,Jiao2023deep} demonstrates that $\bbE_{\underline{\mX}_{k+1}}\big[\|\nabla_{\vx}\log\what{q}_{k+1}(\underline{\mX}_{k+1}|\vy_{[k]})-\what{\vs}_{k+1}(\underline{\mX}_{k+1},\vy_{[k]})\|_{2}^{2}\big]$ can be sufficiently small with high probability, as the number of samples $n$ approaches infinity and the smoothing level $\sigma>0$ in~\eqref{eq:smooth:density} converges to zero. A complete proof of this convergence result can be found in~\cite[Theorem 1]{Tang2024Adaptivity}.

\par The rest of this section is organized as Figure~\ref{figure:structure:error}.

\begin{figure}[htbp]
\centering

\definecolor{myteal}{RGB}{32, 133, 114}
\definecolor{teallight}{RGB}{230, 245, 240}

\definecolor{mypurple}{RGB}{112, 89, 171}
\definecolor{purplelight}{RGB}{238, 235, 248}

\definecolor{myamber}{RGB}{204, 122, 0}
\definecolor{amberlight}{RGB}{253, 243, 230}

\definecolor{mycoral}{RGB}{214, 86, 61}
\definecolor{corallight}{RGB}{253, 234, 229}

\definecolor{myblue}{RGB}{41, 105, 176}
\definecolor{bluelight}{RGB}{235, 243, 250}

\definecolor{mygray}{RGB}{100, 100, 100}
\definecolor{graylight}{RGB}{245, 245, 245}

\tikzset{
basebox/.style={
rounded corners=6pt, 
align=center, 
inner sep=8pt, 
thick,
font=\sffamily\footnotesize},
termbox/.style={
basebox,
minimum width=3.2cm,
text width=2.8cm,
minimum height=1.1cm},
largebox/.style={
basebox,
minimum width=15.0cm,
text width=13.4cm},
myarrow/.style={
-{Stealth[length=6pt, width=5pt]}, 
thick, 
mygray},
arrowlabel/.style={
right, 
font=\sffamily\footnotesize\itshape, 
text=mygray}
}

\begin{tikzpicture}[node distance=0.6cm and 0.4cm]

\node[termbox, draw=myteal, fill=teallight!10, text=myteal] (lan) {Convergence of Langevin dynamics};
\node[termbox, draw=mypurple, fill=purplelight!10, text=mypurple, right=of lan] (disc) {Discretization error};
\node[termbox, draw=myamber, fill=amberlight!10, text=myamber, right=of disc] (prior) {Prior error};
\node[termbox, draw=mycoral, fill=corallight!10, text=mycoral, right=of prior] (score) {Score estimation error};

\node[largebox, draw=mygray, fill=graylight!10, text=black, below=0.7cm of $(lan.south)!0.5!(score.south)$] (total) {%
\textbf{Error of Posterior Sampling} {(Theorem~\ref{theorem:section:convergence} and Corollary~\ref{corollary:section:convergence})}
\begin{equation*}
(\varepsilon_{\tv}^{k+1})^{2}\lesssim C(C_{\LSI}\eta_{\chi}+T)C_{\LSI}^{\frac{1}{4}}(\varepsilon_{\tv}^{k})^{2\gamma}, \quad \gamma\in(0,1).
\end{equation*}
\vspace{-10pt}
};

\draw[myarrow] (lan.south)   -- (lan.south   |- total.north);
\draw[myarrow] (disc.south)  -- (disc.south  |- total.north);
\draw[myarrow] (prior.south) -- (prior.south |- total.north);
\draw[myarrow] (score.south) -- (score.south |- total.north);

\node[largebox, draw=myblue, fill=bluelight!10, text=black, below=0.6cm of total] (assim) {%
\textbf{Error of Assimilation} \upshape{(Theorem~\ref{theorem:section:convergence:assimilation})}
\begin{equation*}
\varepsilon_{\tv}^{k+1} \leq \big\{C(C_{\LSI}\eta_{\chi}+T)C_{\LSI}^{\frac{1}{4}}\big\}^{\frac{1-\gamma^{k}}{2(1-\gamma)}} \varepsilon^{\gamma^{k}}, \quad \gamma\in(0,1).
\end{equation*}
\vspace{-10pt}
};

\draw[myarrow] (total.south) -- node[arrowlabel]{unroll $k$ steps} (assim.north);

\node[largebox, draw=myteal, fill=teallight!10, text=black, below=0.6cm of assim] (floor) {%
\textbf{Long-Time Behavior and Asymptotic Stability}\\
\vspace{-0.2cm}
\begin{equation*}
\limsup_{k\to\infty}\varepsilon_{\tv}^{k+1} \leq \big\{C(C_{\LSI}\eta_{\chi}+T)C_{\LSI}^{\frac{1}{4}}\big\}^{\frac{1}{2(1-\gamma)}}, \quad \gamma\in(0,1).
\end{equation*}
\vspace{-10pt}
};

\draw[myarrow] (assim.south) -- node[arrowlabel]{$k\to\infty$} (floor.north);

\end{tikzpicture}
\caption{The organization of theoretical results. The definition of constants are given in Theorem~\ref{theorem:section:convergence}.}
\label{figure:structure:error}
\end{figure}

\subsection{Convergence analysis for posterior sampling}
\label{section:convergence:posterior}

\par Our main theoretical result for posterior sampling is stated as the following theorem.

\begin{theorem}[Error of posterior sampling]\label{theorem:section:convergence}
Suppose Assumptions~\ref{assumption:posterior:smooth},~\ref{assumption:LSI:posterior},~\ref{assumption:bounded}, and~\ref{assumption:sm} hold. Then for each $k\in\bbN$ and each terminal time $T=Kh$, 
\begin{align*}
(\varepsilon_{\tv}^{k+1})^{2}
&\lesssim\underbrace{\exp\Big(-\frac{T}{5C_{\LSI}}\Big)\eta_{\chi}^{2}}_{\text{convergence of Langevin diffusion}}+\underbrace{dC_{\LSI}\lambda^{2}h}_{\text{discretization error}}+\underbrace{C(C_{\LSI}\eta_{\chi}+T)C_{\LSI}^{\frac{1}{4}}(\varepsilon_{\tv}^{k})^{2\gamma}}_{\text{prior error}} \\
&\quad +\underbrace{C^{\prime}(C_{\LSI}\eta_{\chi}+T)C_{\LSI}(\kappa\Delta)^{\alpha}\log^{d+2}\Big(\frac{C_{\LSI}}{\kappa\Delta}\Big)}_{\text{score estimation error}},
\end{align*}
where $C$ and $C^{\prime}$ are constants only depending on $d$ and $B$, and
\begin{equation*}
\alpha \coloneqq \frac{1}{2+16BC_{\LSI}}, \quad \gamma \coloneqq \frac{1+96BC_{\LSI}}{1+128BC_{\LSI}}.
\end{equation*}
Here the step size $h$ and the initial discrepancy $\eta_{\chi}$ satisfies
\begin{equation*}
h\lesssim\frac{1}{dC_{\LSI}\lambda^{2}}, \quad 
\chi^{2}\big(\pi_{k+1}^{0}(\cdot|\vy_{[k+1]})\|\pi_{k+1}(\cdot|\vy_{[k+1]})\big)\leq\eta_{\chi}^{2}.
\end{equation*}
\end{theorem}

\par The proof of Theorem~\ref{theorem:section:convergence} is provided in Appendix~\ref{appendix:convergence}.

\par\textbf{Error decomposition.} 
Theorem~\ref{theorem:section:convergence} decomposes the total variation error into four fundamental components: the convergence of Langevin diffusion, the discretization error, the prior error, and the score estimation error.
\begin{enumerate}[(i)]
\item The error of the Langevin diffusion~\eqref{eq:method:update:LD} exhibits exponential convergence to zero as the terminal time $T$ increases~\cite{Vempala2019Rapid,Chewi2024Analysis}.
\item The discretization error, arising from the Euler-Maruyama approximation~\eqref{eq:section:convergence:SLMC}, converges linearly with respect to the step size $h$.
\item The prior error stems from the prior distribution approximation in~\eqref{eq:prediction:distribution:hat} and is governed by the error $\varepsilon_{\tv}^{k}$ of the previous posterior distribution estimation.
\item The score estimation error decreases as the score matching tolerance $\Delta$ in Assumption~\ref{assumption:sm} tends to zero.
\end{enumerate}

To complement the theoretical analysis above, we conduct a sensitivity study in Appendix~\ref{appendix:experiment:sensitivity} examining the effect of key
hyperparameters on posterior accuracy. The results demonstrate that increasing
the number of Langevin steps $K$ and the ensemble size, while reducing the step size $h$, consistently leads to smaller posterior errors. These empirical findings align with the theoretical bounds in Theorem~\ref{theorem:section:convergence}: a larger number of Langevin steps
corresponds to a longer terminal time $T$ and thus a smaller Langevin
diffusion error; a finer step size $h$ reduces the discretization error; and
a larger ensemble size yields a more accurate score approximation, reducing
the score estimation error.

\par\textbf{Early-stopping.}
The first term in Theorem~\ref{theorem:section:convergence} diminishes with increasing terminal time $T$. However, both the score estimation error and the prior error grow with $T$, establishing a fundamental trade-off in the sampling error decomposition. This trade-off necessitates early-stopping in score-based Langevin sampling, as noted by~\cite{Lee2022Convergence}. Corollary~\ref{corollary:section:convergence} derives the optimal number of iterations based on this trade-off.

\par The following corollary provides a theoretical guidance for selecting hyper-parameters in the score-based Langevin sampling, as well as outlines its computational complexity.

\begin{corollary}\label{corollary:section:convergence}
Suppose Assumptions~\ref{assumption:posterior:smooth},~\ref{assumption:LSI:posterior},~\ref{assumption:bounded}, and~\ref{assumption:sm} hold. Then for all $k\in\bbN$,
\begin{equation*}
(\varepsilon_{\tv}^{k+1})^{2}\lesssim C(C_{\LSI}\eta_{\chi}+T)C_{\LSI}^{\frac{1}{4}}(\varepsilon_{\tv}^{k})^{2\gamma},
\end{equation*}
where the step size $h$, the number of the Langevin iterations $K$, and the score matching error $\Delta$ satisfy
\begin{align*}
&h \asymp \frac{C}{d\lambda^{2}}(C_{\LSI}\eta_{\chi}+T)C_{\LSI}^{-\frac{3}{4}}(\varepsilon_{\tv}^{k})^{2\gamma}, \quad
K \asymp \frac{d\lambda^{2}}{C(C_{\LSI}\eta_{\chi}+T)}C_{\LSI}^{\frac{7}{4}}(\varepsilon_{\tv}^{k})^{-2\gamma}, \\
&\Delta \lesssim \Big(\frac{C}{C^{\prime}}\Big)^{\frac{1}{\alpha}}C_{\LSI}^{-\frac{3}{4\alpha}}\kappa^{-1}(\varepsilon_{\tv}^{k})^{\frac{2\gamma}{\alpha}}.
\end{align*}
Here, logarithmic factors are omitted, and the constants $\alpha \in (0,1)$ and $\gamma \in (0,1)$ are defined in Theorem~\ref{theorem:section:convergence}, and $C$ and $C^{\prime}$ are constants only depending on $d$ and $B$.
\end{corollary}

\par The proof of Corollary~\ref{corollary:section:convergence} is provided in Appendix~\ref{appendix:convergence}.

\par\textbf{Warm-start.}
In Theorem~\ref{theorem:section:convergence} and Corollary~\ref{corollary:section:convergence}, a warm-start condition in terms of $\chi^{2}$-divergence is essential, requiring that the initial distribution remains sufficiently close to the target posterior distribution:
\begin{equation}\label{eq:warm:start}
\chi^{2}\big(\pi_{k+1}^{0}(\cdot|\vy_{[k+1]})\|\pi_{k+1}(\cdot|\vy_{[k+1]})\big)\leq\eta_{\chi}^{2}.
\end{equation}
This warm-start condition~\eqref{eq:warm:start} influences the convergence of Langevin diffusion, the prior error, and the score estimation error in Theorem~\ref{theorem:section:convergence}. Regarding the score estimation, a critical observation is that the score estimator $\what{\vs}_{k+1}(\cdot,\vy_{[k+1]})$~\eqref{eq:dsm} approximates the prediction score in the $L^{2}$-norm (Assumption~\ref{assumption:sm}), where the $L^{2}$-error is measured with respect to the approximated prediction distribution $\what{q}_{k+1}(\cdot|\vy_{[k]})$~\eqref{eq:prediction:distribution:hat}. However, the Girsanov theorem~\cite{chen2023sampling} indicates that the score estimation error is bounded by the $L^{2}$-error of the score with respect to the law of Langevin sampling using the exact score. This discrepancy necessitates the warm-start condition in $\chi^{2}$-divergence, as it ensures the out-of-distribution generalization of the score estimator from the approximated prediction distribution to the law of Langevin sampling using the exact score~\cite{Lee2022Convergence,Tang2024Adaptivity}. 

\par There are two commonly-used initialization strategies:
\begin{enumerate}
\item \emph{Purly prior-based initialization.} Setting $\pi_{k+1}^{0}(\cdot|\vy_{[k+1]}) \coloneqq q_{k+1}(\cdot|\vy_{[k]})$, i.e., solely using the prior information, yields $\eta_{\chi}^{2}\leq\kappa-1$ via~\eqref{eq:condition:number}. Indeed, 
\begin{align}
\chi^{2}\big(\pi_{k+1}^{0}(\cdot|\vy_{[k+1]})\|\pi_{k+1}(\cdot|\vy_{[k+1]})\big)
&=\int\Big(\frac{\pi_{k+1}^{0}(\vx|\vy_{[k+1]})}{\pi_{k+1}(\vx|\vy_{[k+1]})}\Big)^{2}\pi_{k+1}(\vx|\vy_{[k+1]})\d\vx-1 \nonumber \\
&\leq \kappa\int\Big(\frac{\pi_{k+1}^{0}(\vx|\vy_{[k+1]})}{\pi_{k+1}(\vx|\vy_{[k+1]})}\Big)\pi_{k+1}(\vx|\vy_{[k+1]})\d\vx-1 = \kappa-1. \label{eq:init:discrepancy}
\end{align}
\item \emph{Purly likelihood-based initialization.} Alternatively, one may construct an initial distribution using maximum likelihood from the measurement likelihood alone, as described in~\cite[Section 5]{purohit2024posterior}.
\end{enumerate}

\par By similar arguments, we can also provide a convergence rate for posterior sampling at the initial time step.

\begin{theorem}\label{theorem:convergence:init}
Suppose Assumptions~\ref{assumption:posterior:smooth},~\ref{assumption:LSI:posterior},~\ref{assumption:bounded}, and~\ref{assumption:sm} hold. Then for any $\varepsilon\in(0,1)$, 
\begin{equation*}
(\varepsilon_{\tv}^{1})^{2}\lesssim \varepsilon^{2},
\end{equation*}
provided that the step size $h$, the number of the Langevin iterations $K$, and the score matching error $\Delta$ satisfy
\begin{equation*}
h \asymp \frac{\varepsilon^{2}}{dC_{\LSI}\lambda^{2}}, \quad
K \asymp \frac{dC_{\LSI}^{2}\lambda^{2}}{\varepsilon^{2}}, \quad \Delta \lesssim \frac{1}{\kappa}\Big(\frac{\varepsilon^{2}}{C^{\prime}C_{\LSI}^{2}(\eta_{\chi}+1)}\Big)^{\frac{1}{\alpha}}.
\end{equation*}
Here, logarithmic factors are omitted, and the constant $\alpha \in (0,1)$ is defined in Theorem~\ref{theorem:section:convergence}, and $C^{\prime}$ is a constant only depending on $d$ and $B$.
\end{theorem}

\par The proof of Theorem~\ref{theorem:convergence:init} is provided in Appendix~\ref{appendix:convergence}.

\subsection{Discussions on annealing}

\par In this subsection, we demonstrate theoretical advantages of annealing strategy introduced in Section~\ref{section:method:update}.

\par Theorem~\ref{theorem:section:convergence} and Corollary~\ref{corollary:section:convergence} reveal two critical factors affecting the posterior sampling error: the condition number $\kappa$ in~\eqref{eq:condition:number}, and the initial discrepancy $\eta_{\chi}^{2}$ in~\eqref{eq:warm:start}. Specifically, the score matching tolerance $\Delta$ decreases polynomially with both the condition number and the initial discrepancy. Consequently, addressing the challenge posed by large condition numbers and substantial initial discrepancy necessitates the use of larger deep neural networks and increased ensemble size, substantially raising computational costs.

\par We next demonstrate how the annealing strategy introduced in Section~\ref{section:method:update} effectively mitigates the fundamental theoretical bottlenecks of score-based sampling by reducing both the effective condition number and the initial discrepancy.

\par\textbf{Reduction of condition number.}
According to Theorem~\ref{theorem:section:convergence}, the convergence rate and the sampling complexity of the Langevin dynamics depend critically on the quantity $\kappa$ defined in~\eqref{eq:condition:number}. In highly ill-posed inverse problems, this factor can become arbitrarily large; see Remark~\ref{remark:cn}. However, by employing a sufficiently dense annealing schedule, the effective condition number at each temperature can be maintained near its optimal value of $1$, regardless of the underlying ill-posedness of the original posterior sampling. Specifically, from~\eqref{eq:bayes} and~\eqref{eq:interpolation}, for a temperature $\beta_{m}$, the posterior density $\pi_{k+1}^{m}(\cdot|\vy_{[k+1]})$ and the prior density $\pi_{k+1}^{m-1}(\cdot|\vy_{[k+1]})$ are given, respectively, as 
\begin{align*}
\pi_{k+1}^{m}(\vx|\vy_{[k+1]}) &\propto g_{k+1}^{\beta_{m}}(\vy_{k+1}|\vx)q_{k+1}(\vx|\vy_{[k]}), \\
\pi_{k+1}^{m-1}(\vx|\vy_{[k+1]}) &\propto g_{k+1}^{\beta_{m-1}}(\vy_{k+1}|\vx)q_{k+1}(\vx|\vy_{[k]}).
\end{align*}
Then the effective condition number $\kappa_m$ for the transition from $m-1$ to $m$ is given by:
\begin{align*}
\kappa_m 
&\coloneqq \frac{\sup_{\vx} g_{k+1}^{\beta_{m}-\beta_{m-1}}(\vy_{k+1}|\vx)}{\int g_{k+1}^{\beta_{m}-\beta_{m-1}}(\vy_{k+1}|\vx)\pi_{k+1}^{m-1}(\vx|\vy_{[k+1]})\mathrm{d}\vx} \\
&=\sup_{\vx}g_{k+1}^{\beta_{m}-\beta_{m-1}}(\vy_{k+1}|\vx)\frac{\int g_{k+1}^{\beta_{m-1}}(\vy_{k+1}|\vx)q_{k+1}(\vx|\vy_{[k]})\d\vx}{\int g_{k+1}^{\beta_{m}}(\vy_{k+1}|\vx)q_{k+1}(\vx|\vy_{[k]})\d\vx}.
\end{align*}
It is apparent that $\kappa=\prod_{m=1}^{M}\kappa_{m}$. Therefore, if the noise schedule is proper chosen, we have $\kappa_{m}\simeq\kappa^{1/M}$. This result demonstrates that annealing significantly reduces the condition number of the sampling problem.

\par\textbf{Reduction of initial discrepancy.}
The second primary advantage of the annealing strategy lies in its control over the initial discrepancy. In the vanilla Langevin sampling, as shown in~\eqref{eq:init:discrepancy}, the initial discrepancy is $\eta_{\chi}^{2}\leq\kappa-1$. In contrast, at the $m$-th stage of annealed Langevin sampling, we have
\begin{equation*}
\eta_{\chi,m}^{2} \coloneqq \chi^{2}(\pi_{k+1}^{m}(\cdot|\vy_{[k+1]})\| \pi_{k+1}^{m-1}(\cdot|\vy_{[k+1]})) \leq \kappa_{m}-1 \simeq \kappa^{1/M}-1.
\end{equation*}
This significantly smaller discrepancy ensures that, at each annealing stage, the Langevin sampler is initialized within a high-probability region of the target distribution, thereby accelerating convergence.

\begin{remark}
We provide some theoretical insights of the theoretical benefits of annealing. However, a rigorous non-asymptotic convergence analysis for the full annealing schedule presents additional challenges. In particular, it requires additional assumptions on the annealing geometry, controlling the error propagation across the intermediate annealing stages, and bounding the LSI constants of the intermediate distributions. These difficulties make a complete formal analysis considerably more involved than the single-stage result in Theorem~\ref{theorem:section:convergence}, and we leave a full theoretical treatment of annealing as an important direction for future work.
\end{remark}

\subsection{Convergence analysis for assimilation}
\label{section:convergence:assimilation}

Building upon the posterior estimation convergence results for score-based Langevin sampling established in the preceding subsection, we now analyze the convergence properties of score-based sequential Langevin sampling (SSLS). Corollary~\ref{corollary:section:convergence} establishes an error recursion of the form
\begin{equation}\label{eq:contraction}
\varepsilon_{\tv}^{k+1} \lesssim A(\varepsilon_{\tv}^{k})^{\gamma}.
\end{equation}
Combining this recursion with the initial error bound $\varepsilon_{\tv}^{1} \lesssim \varepsilon$ from Theorem~\ref{theorem:convergence:init}, a straightforward induction yields
\begin{equation*}
\varepsilon_{\tv}^{k+1} \lesssim A^{\sum_{i=0}^{k-1}\gamma^{i}}\varepsilon^{\gamma^{k}}=A^{\frac{1-\gamma^{k}}{1-\gamma}}\varepsilon^{\gamma^{k}}.
\end{equation*}
The following theorem characterizes the convergence behavior of the assimilation process under SSLS.

\begin{theorem}[Error of assimilation in TV distance]\label{theorem:section:convergence:assimilation}
Suppose Assumptions~\ref{assumption:posterior:smooth}, \ref{assumption:LSI:posterior}, \ref{assumption:bounded}, and \ref{assumption:sm} hold. Then, for all time steps $k \in \bbN$ and any error tolerance $\varepsilon \in (0,1)$, the total variation error satisfies
\begin{equation*}
\varepsilon_{\tv}^{k+1} \leq \left\{C(C_{\LSI}\eta_{\chi}+T)C_{\LSI}^{\frac{1}{4}}\right\}^{\frac{1-\gamma^{k}}{2(1-\gamma)}} \varepsilon^{\gamma^{k}},
\end{equation*}
where the step size $h$, the number of Langevin iterations $K$, and the score matching error $\Delta$ for the initial time step and $k$-th time step are given in Theorem~\ref{theorem:convergence:init} and Corollary~\ref{corollary:section:convergence}, respectively. Here the constant $\gamma \in (0,1)$ are as defined in Theorem~\ref{theorem:section:convergence}, and $C$ is a constant depending only on $d$ and $B$.
\end{theorem}

\par The proof of Theorem~\ref{theorem:section:convergence:assimilation} is provided in Appendix~\ref{appendix:convergence}.

\par\textbf{Consistency over finite time horizons.}
Theorem~\ref{theorem:section:convergence:assimilation} establishes that SSLS yields a consistent estimator of the posterior distribution. As noted in Remark~\ref{remark:cn}, the score matching error $\Delta$ vanishes as the ensemble size grows to infinity. Consequently, Theorem~\ref{theorem:section:convergence:assimilation} implies that the SSLS error vanishes asymptotically over any finite time horizon when the ensemble size and the number of Langevin iterations tend to infinity simultaneously. This is a notable advantage over the Kalman filter and the standard diffusion-based methods discussed in Section~\ref{section:related:diffusion}, which exhibit inherent inconsistency in assimilation scenarios where linearity and Gaussianity assumptions are violated.

\par\textbf{Long-time behavior and error stability.}
We examine the long-time behavior of the error bound in Theorem~\ref{theorem:section:convergence:assimilation} by separately analyzing the prefactor and the dependence on the initial error. This long-time behavior is governed entirely by the contraction rate $\gamma\in(0,1)$.
\begin{enumerate}
\item \emph{Forgetting of initial error.}
As $k\to\infty$, we have $\gamma^{k}\to 0$ (since $\gamma\in(0,1)$), which implies $\varepsilon^{\gamma^{k}}\to 1$. Consequently, the bound asymptotically loses its dependence on the magnitude of the initial error $\varepsilon$. This demonstrates a crucial long-time stability property: the long-run performance of the algorithm is insensitive to initialization.
\item \emph{Convergence to the error floor.}
Taking the limit $k\to\infty$ for the prefactor $A$ in the recursion yields the time-uniform asymptotic error bound
\begin{equation}\label{eq:error:floor}
\limsup_{k\to\infty}\varepsilon_{\tv}^{k+1} \leq \left\{C(C_{\LSI}\eta_{\chi}+T)C_{\LSI}^{\frac{1}{4}}\right\}^{\frac{1}{2(1-\gamma)}},
\end{equation}
effectively confining the total variation error to a stable neighborhood around the true posterior.
\end{enumerate}

\par Intuitively, the error should not be expected to decay to zero without imposing further structural assumptions on the sequential algorithm, since at each assimilation step $k$, fresh approximation errors are inevitably introduced by the time discretization $h$, the finite number of Langevin iterations $K$, and the score matching error $\Delta$. The presence of such a persistent error floor in~\eqref{eq:error:floor} is theoretically sound and standard in the sequential filtering literature; see, e.g.,~\cite{Crisan2002survey} and~\cite[Chapter~7]{Moral2004Feynman}. Furthermore, this bound is valuable because it guarantees that the assimilation error does not accumulate indefinitely from local errors, a property intimately connected to the ergodicity and stability of nonlinear filters~\cite{Bain2009Fundamentals,Kelly2014Well,Tong2016Nonlinear}.

\par\textbf{Convergence in Wasserstein distance.}
Theorem~\ref{theorem:section:convergence:assimilation} provides a convergence rate for assimilation in TV distance. Now we aim to establish convergence rate in Wasserstein distance. 

\begin{definition}[Wasserstein-1 distance]
Let $\mu$ and $\pi$ be two probability measures. The Wasserstein-1 distance between $\mu$ and $\pi$ is defined by 
\begin{equation*}
W_{1}(\mu,\pi) \coloneqq \inf\left\{\mathbb{E}\left[\|X-Y\|_{2}\right]:\,\mathrm{Law}(X)=\mu,\,\mathrm{Law}(Y)=\pi\right\}.
\end{equation*}
\end{definition}

Before proceeding, we introduce a truncation operator with a radius $R_{0}\geq 1$,
\begin{equation*}
\Pi_{R_{0}}:\mathbb{R}^{d}\to\mathbb{R}^{d}, \quad \vx\mapsto \vx\mathbbm{1}\{\|\vx\|_{2}\leq R_{0}\}.
\end{equation*}
The convergence rate of the truncated estimated distribution $(\Pi_{R_{0}})_{\sharp}\what{\pi}_{k+1}(\cdot|\vy_{[k+1]})$ in Wasserstein distance is stated in the following corollary.

\begin{corollary}[Error of assimilation in Wasserstein distance]\label{corollary:section:convergence:wasserstein}
Under the same assumptions as Theorem~\ref{theorem:section:convergence:assimilation}. For all time steps $k \in \bbN$ and any error tolerance $\varepsilon \in (0,1)$, the Wasserstein error satisfies
\begin{equation*}
W_{1}^{2}((\Pi_{R_{0}})_{\sharp}\what{\pi}_{k+1}(\cdot|\vy_{[k+1]}),\pi_{k+1}(\cdot|\vy_{[k+1]})) \lesssim C_{\LSI}\left\{C(C_{\LSI}\eta_{\chi}+T)C_{\LSI}^{\frac{1}{4}}\right\}^{\frac{1-\gamma^{k}}{1-\gamma}}\varepsilon^{2\gamma^{k}},
\end{equation*}
where logarithmic factors are omitted, $C$ is a constant depending only on $d$ and $B$, and the truncation radius $R_{0}$ is given as 
\begin{equation*}
R_{0}^{2}=C_{\LSI}\log\Big(e^{\frac{d}{2}}C_{\LSI}\{C(C_{\LSI}\eta_{\chi}+T)C_{\LSI}^{\frac{1}{4}}\}^{-\frac{1-\gamma^{k}}{1-\gamma}} \varepsilon^{-2\gamma^{k}}\Big).
\end{equation*}
\end{corollary}

\par The proof of Corollary~\ref{corollary:section:convergence:wasserstein} is provided in Appendix~\ref{appendix:convergence}.


\section{Numerical Experiments}
\label{section:experiments}

\par In this section, we demonstrate the effectiveness of score-based sequential Langevin sampling (SSLS) through numerical experiments.
\begin{enumerate}[(1)]
\item In Section~\ref{section:experiment:doublewell}, we examine the assimilation of Langevin diffusion with a double-well potential. This investigation compares SSLS against the auxiliary particle filter (APF) and ensemble Kalman filter (EnKF) in scenarios featuring state mutations and model nonlinearity.
\item In Section~\ref{section:experiment:Kolmogorov}, we apply SSLS to the assimilation of Kolmogorov flow under sparse and partial observations. We highlight the crucial role of the prior score through comparisons with standard maximum likelihood estimation (MLE) and demonstrate methods for quantifying the uncertainty of states estimated by SSLS.
\end{enumerate}

\subsection{The double-well potential}
\label{section:experiment:doublewell}

\par A classic problem in molecular dynamics involves a one-dimensional Langevin diffusion with a double-well potential, governed by the nonlinear SDE:
\begin{equation}\label{eq:dw:sde}
\d X_{t}=-\nabla U(X_{t})\dt+\beta\d B_{t},
\end{equation}
where $U(x):=x^{4} -2x^{2}$ is a double-well potential, $\beta>0$ is the temperature parameter, and $(B_{t})_{t\geq 0}$ denotes the standard Brownian motion. In this experiment, we focus on the dynamics model defined as the Euler-Maruyama discretization of~\eqref{eq:dw:sde}
\begin{equation}\label{eq:dw:dynamics}
X_{k+1}=X_{k}-\delta t\nabla U(X_k)+\beta\sqrt{\delta t}V_{k},
\end{equation}
where $\delta t>0$ and $V_{k}\sim N(0,1)$. As we will demonstrate, this dynamic model exhibits state mutations that introduce ill-posedness to the assimilation process. This experiment is designed to showcase SSLS's effectiveness in handling such mutations.

\par\textbf{State mutations in dynamics model.}
Since the dynamics model~\eqref{eq:dw:dynamics} has similar behaviors as its time-continuous counterpart~\eqref{eq:dw:sde} when $\delta t$ is sufficiently small, we illustrate state mutations in~\eqref{eq:dw:sde} for the sake of simplicity. The dynamics model~\eqref{eq:dw:sde} exhibits two local stable states at $x=-1$ and $x=1$. A particle initialized at any position is drawn toward one of these stable states by the drift term in~\eqref{eq:dw:sde}, while the diffusion term models thermal collisions with the environment, introducing stochasticity into the particle's trajectory. At low temperatures (small $\beta$), the particle typically remains confined near one stable state, making only rare transitions to the other potential well. At higher temperatures (large $\beta$), the particle transitions more frequently between the two potential wells, leading to state mutations.

\par When state mutations occur, the true state can deviate significantly from the support of the prediction distribution. Such deviation results in an extremely large condition number~\eqref{eq:condition:number}, indicating severe ill-posedness in the posterior sampling, as discussed by Remark~\ref{remark:cn}.

\par\textbf{Reference states generation.}
To evaluate the assimilation methods, we simulate the dynamics model~\eqref{eq:dw:dynamics} with temperature $\beta=0.3$ and time step $\delta t=0.1$ to generate true states. To effectively demonstrate how different assimilation methods respond to state mutations, we manually induce mutations by switching states from $X_{k}$ to $-X_{k}$ every 20 time steps.

\par\textbf{Baseline.}
For comparison, we evaluate SSLS alongside two widely-used assimilation techniques: the auxiliary particle filter (APF)~\cite{pitt1999filtering} and the ensemble Kalman filter~\cite{houtekamer1998data}. In SSLS, prediction scores at each assimilation step are learned from 1000 particles. For fair comparison, we maintain the same ensemble size of 1000 for both APF and EnKF. 

\subsubsection{Linear measurement model}

\par We first consider a linear measurement model with Gaussian additive noise
\begin{equation}\label{eq:dw:m1}
Y_{k}=X_{k}+\sigma_{\obs}W_{k},
\end{equation}
where $k\in\bbN$ and $W_{k}\sim N(0,1)$. Figure~\ref{fig:DW_0.1} plots the results of assimilation for the state-space model~\eqref{eq:dw:dynamics} and~\eqref{eq:dw:m1}. The observation noise level is set as $\sigma_{\obs}=0.1$.

\begin{figure}[htbp]
\centering
\includegraphics[width=0.75\linewidth]{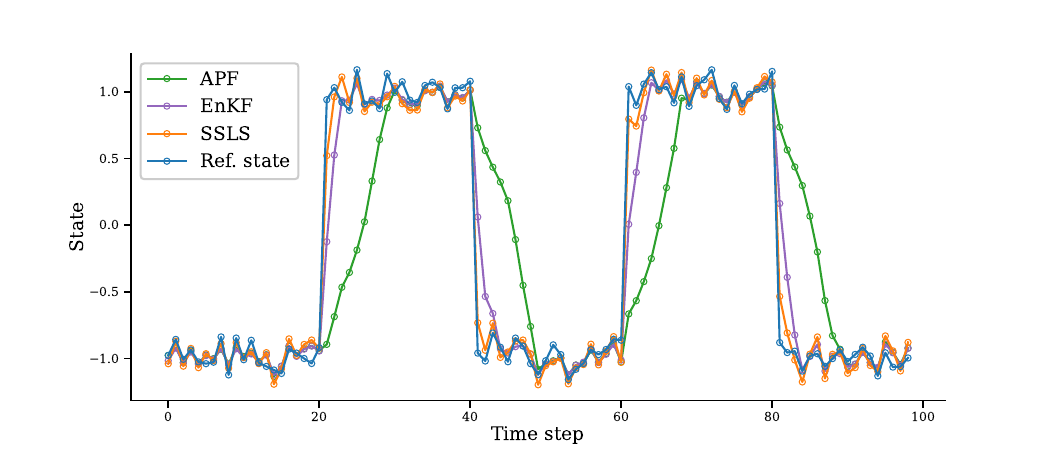}
\caption{Results of assimilation for Langevin diffusion with a double-well potential~\eqref{eq:dw:dynamics} with a linear measurement model~\eqref{eq:dw:m1}. The ensemble mean of SSLS, APF, and EnKF at each time steps are shown in the figure.}
\label{fig:DW_0.1}
\end{figure}

\par As shown in Figure~\ref{fig:DW_0.1}, all assimilation methods provide estimates that closely track the reference states before state mutations occur. However, APF exhibits a notable delay following state mutations. This delay arises because APF approximates the prediction distribution using a weighted particle set. Due to state mutations, particles sampled from the prediction distribution may be far from the observation, resulting in small assigned weights. These particles thus contribute minimally to the posterior distribution approximation~\cite[Chapter 11.6]{Sarkka20232023Bayesian}. This phenomenon, known as particle degeneracy, persists even though APF offers some improvement over the standard PF.

\par In contrast, SSLS and EnKF avoid reliance on particle approximations. SSLS employs a score network, while EnKF uses a Gaussian distribution to approximate the prediction distribution. Both methods can generalize to regions near the observation where predicted particles may be absent, enabling them to respond more rapidly to state mutations.

\subsubsection{Nonlinear measurement model}
As illustrated in Figure~\ref{fig:DW_0.1}, SSLS outperforms EnKF due to the nonlinearity of dynamics model~\eqref{eq:dw:dynamics}. This advantage arises because EnKF only captures linear components while disregarding higher-order structures of the dynamics model. To further demonstrate the advantages of SSLS over EnKF in nonlinear settings, we consider a nonlinear measurement model:
\begin{equation}\label{eq:dw:m2}
Y_{k}=\exp(X_{k}-\gamma_{k})+\sigma_{\obs}W_{k},
\end{equation}
where $\gamma_{k}=0.6$, $\sigma_{\obs}=0.2$, and $W_{k}\sim N(0,1)$.
\begin{figure}[htbp]
\centering
\includegraphics[width=0.75\linewidth]{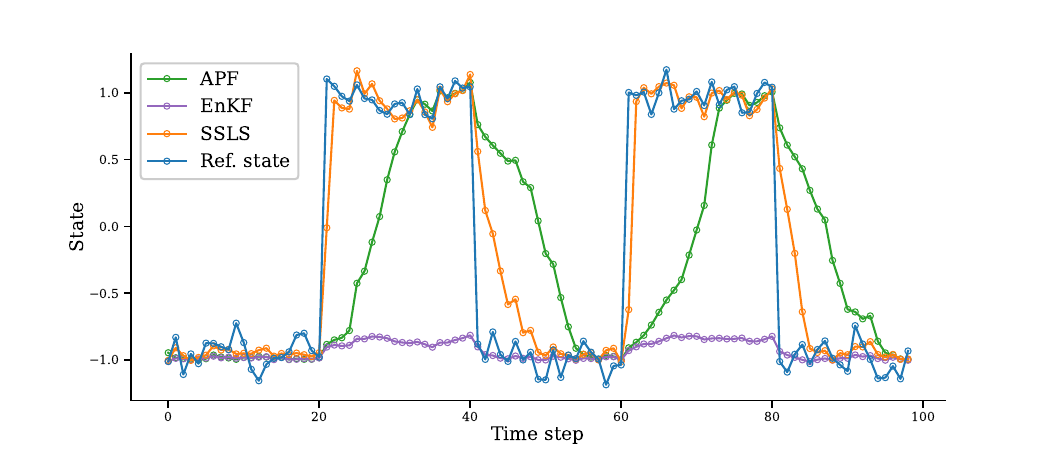}
\caption{Results of assimilation for Langevin diffusion with a double-well potential~\eqref{eq:dw:dynamics} with a nonlinear measurement model~\eqref{eq:dw:m2}. The ensemble mean of SSLS, APF, and EnKF at each time steps are shown in the figure.}
\label{fig:DW_nonlinear_0.2}
\end{figure}

\par Figure~\ref{fig:DW_nonlinear_0.2} presents the numerical results, revealing that EnKF fails to effectively assimilate observation data in the fully nonlinear state-space model~\eqref{eq:dw:dynamics} and~\eqref{eq:dw:m2}. In contrast, SSLS and APF maintain their performance despite the nonlinear measurement model. These results establish SSLS as a robust method for handling nonlinear scenarios, even under full nonlinearity.

\subsubsection{Multi-modal measurement model}
\label{section:experiment:doublewell:multi-modal}
To demonstrate the effectiveness of different annealing strategies, we consider a new measurement model characterized by a multi-modal noise distribution:$$  Y_k = X_k + W_k,$$where$$  W_k \sim
  \begin{cases}
    0.1 \mathcal{N}(-0.2, 1) + 0.1 \mathcal{N}(0.5, 1)  & \text{if } k \in \{20, 40\}, \\
    0.2 \mathcal{N}(0, 1) & \text{otherwise}.
  \end{cases}$$
This setup introduces multi-modality into the likelihood and posterior during state transitions. It mimics scenarios with reduced observational confidence --- a situation common when dealing with inconsistent measurements, particularly when states undergo complex and drastic spatial-temporal changes. Under this configuration, we evaluate three variants of SSLS, specifically the versions without annealing, with annealing on the likelihood, and with annealing on the posterior. We also compare these variants against APF and EnKF.

As illustrated in Figure \ref{fig:annealing-ablation}, SSLS with either annealing strategy generally detects and adapts to drastic state transitions more rapidly than the version without annealing. Notably, this superior performance is maintained even when the number of Langevin iterations is reduced by half, which is particularly evident following the second jump. Furthermore, SSLS with annealing on the posterior slightly outperforms the likelihood-based approach, likely due to the relatively low noise level in the measurement model. Regardless of the annealing setting, SSLS consistently demonstrates superior performance over both APF and EnKF.

\begin{figure}
  \centering
  \includegraphics[width=0.8\textwidth]{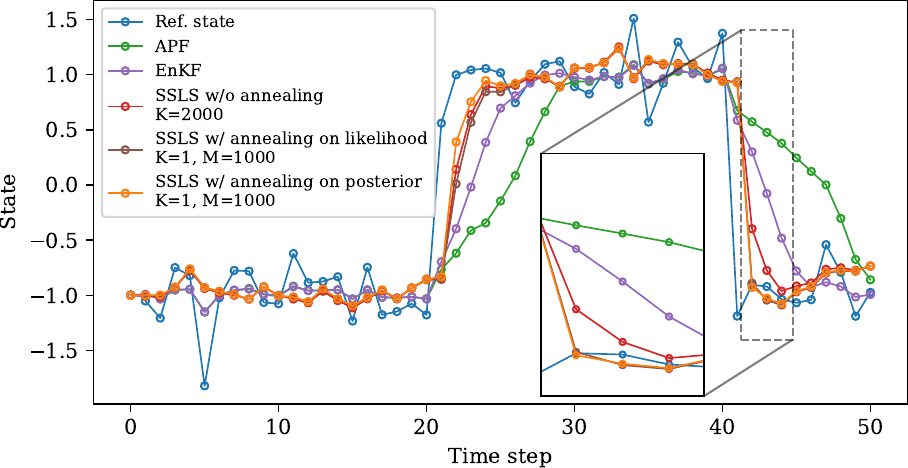}%
  \caption{Performance of SSLS (with different annealing strategies), APF, and EnKF for the double-well system \eqref{eq:dw:dynamics} under multi-modal measurement noise. The curves represent the ensemble mean of each method over time.}
  \label{fig:annealing-ablation}
\end{figure}

\subsection{Kolmogorov flow}
\label{section:experiment:Kolmogorov}

\par In this example, we consider the assimilation of a Kolmogorov flow, which arises in atmospherical sciences and fluid dynamics~\cite{Cotter2009Bayesian,Rozet2023score}. Kolmogorov flow is a viscous and incompressible fluid flow governed by the Navier-Stokes (NS) equation on the two-dimensional torus $[0,2\pi]^{2}$,
\begin{equation}\label{P3:kolmogorov}
\left\{
\begin{aligned}
\partial_{t}\vu&=-(\vu\cdot\nabla)\vu+\frac{1}{\mathrm{Re}} \nabla^2\vu-\frac{1}{\rho}\nabla p+\mF, \\
0&=\nabla \cdot \vu,
\end{aligned}\right.
\end{equation}
where $\vu$ represents the velocity field, $\mathrm{Re}$ is the Reynolds number, $\rho$ denotes the fluid density, $p$ is the pressure field, and $\mF$ is the external forcing. This example uses the periodic boundary conditions, a large Reynolds number $\mathrm{Re}=10^{3}$, a constant density $\rho\equiv 1$ and an external forcing $\mF$ corresponding to Kolmogorov forcing with linear damping. Our objective is to track a velocity field describing the solutions to the NS equation~\eqref{P3:kolmogorov} with unknown initial condition.

\par\textbf{Reference states generation.} 
Let $\vu_{0}$ be a initial random state sampled from a Gaussian random field. The NS equation~\eqref{P3:kolmogorov} is evolved from this initial state $\vu_{0}$. Following a warm-up period of $T_{0}=10$, reference states are downsampled from the trajectory between $T_{0}=10$ and $T=20$ with a spatial resolution of $128\times 128$ and a temporal resolution of $\Delta t=0.2$. Here the NS equation~\eqref{P3:kolmogorov} is solved using the \texttt{jax-cfd} package~\cite{Kochkov2021Machine}. In Appendix \ref{section:experiments:appendix:Kolmogorov}, we further compare our method against a score-based posterior sampling baseline, the Ensemble Kalman Diffusion Guidance (EnKG) \cite{zheng2025ensemble}. We conduct this comparison on the Navier-Stokes (NS) equation under partial observations. Additional experimental details are provided therein.

\par\textbf{Initial prior distribution shift.} 
The SSLS assimilation process begins with a set of independent random realizations of $\vu_{0}$. A crucial consideration is that the initial reference state $\vu_{T_{0}}$ follows a distribution that differs from the initial distribution employed in SSLS. Such initial distributional shifts are prevalent in practical applications, where the true distribution of the initial state often remains unknown. Therefore, successful assimilation under these distributional shifts represents both a critical requirement and a significant challenge in the field.

\subsubsection{Assimilation with sparse or partial observation data}

\par This subsection demonstrates the effectiveness of SSLS under various sparse and partial observation scenarios through three distinct tasks: (i) super-resolution, (ii) sparse reconstruction, and (iii) box reconstruction. For the super-resolution task, observations are generated by applying average pooling to the reference states, followed by the addition of pointwise Gaussian noise. The sparse reconstruction task employs a measurement model combining uniform-stride downsampling with Gaussian perturbation. In the box reconstruction task, measurements within a specified domain are unavailable, while measurements outside this domain are perturbed by Gaussian noise.

\begin{figure}[htbp]
\centering
\includegraphics[width=0.75\linewidth]{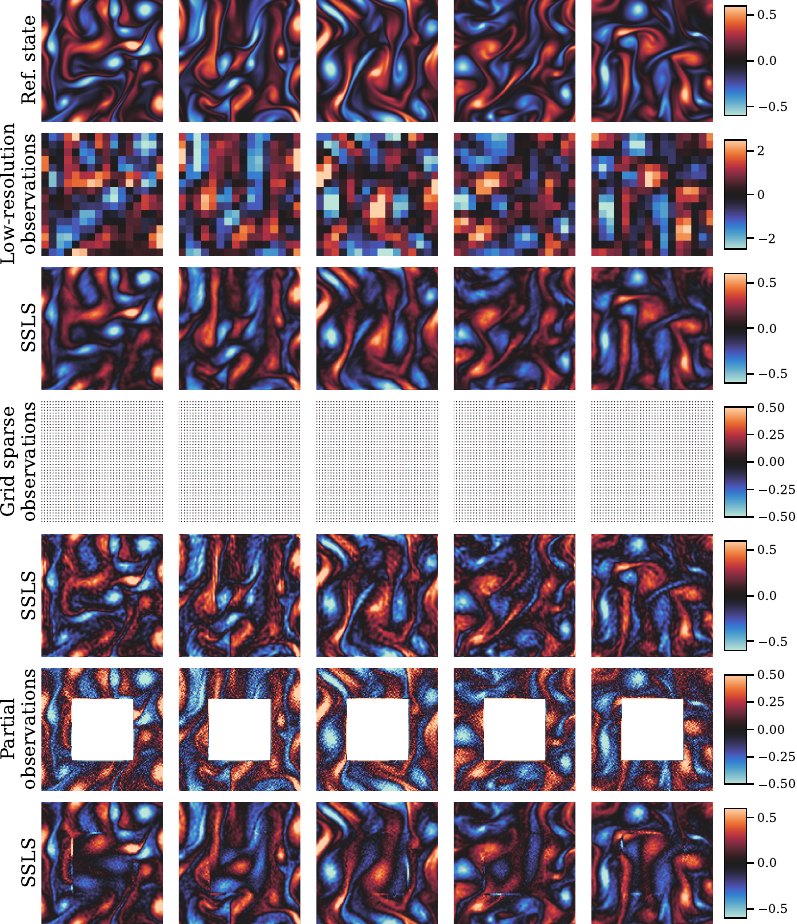}
\caption{Results of assimilation for Kolmogorov flow~\eqref{P3:kolmogorov} with different measurement models. Each column corresponds to distinct time steps (states are plotted for every 10 time steps). The first row displays the reference state. (i) Super-resolution: The 2nd and 3rd rows display the noisy observations with 8x average pooling and the corresponding SSLS estimations, respectively. (ii) Sparse reconstruction: The 4th and 5th rows show the noisy observations with a uniform mask and the corresponding SSLS estimations, respectively. (iii) Box reconstruction: The 6th and 7th rows demonstrate the noisy observations with a centering square mask and the corresponding SSLS estimations, respectively.}
\label{fig:KolmogorovEvolution}
\end{figure}

\par Figure~\ref{fig:KolmogorovEvolution} presents the experimental results of SSLS under these measurement models, visualized through the vorticity field $\omega=\nabla\times\vu$. The results demonstrate the exceptional performance of SSLS across all three tasks:

\begin{enumerate}[(i)]
\item In the super-resolution task, while the observation data lacks most micro-structures present in the reference vorticity field, SSLS successfully reconstructs the majority of these intricate details.
\item The sparse reconstruction results showcase SSLS's remarkable capability to reconstruct the field even when 88.72\% of observation points are masked. This robustness to sparse observations is particularly valuable in applications where measurement acquisition is costly or limited. 
\item In the box reconstruction task, SSLS demonstrates impressive performance by accurately reconstructing the field even within the completely masked region where no observations are available.
\end{enumerate}
Notably, across aforementioned tasks, the estimated states exhibit significant deviations from the observation data, indicating the substantial influence of prediction information on the assimilation process. The effectiveness of this prediction information will be further examined through ablation studies in subsequent experiments. Furthermore, a mathematical analysis of this efficiency under sparse or partial measurements is presented in Section~\ref{section:experiment:Kolmogorov:ablation}.

\subsubsection{Ablation study: influence of the prediction score}
\label{section:experiment:Kolmogorov:ablation}

\par Previous experiments demonstrate that the estimated states closely align with reference states, even under significant observational noise and occlusion. This remarkable performance underscores the fundamental importance of prior information encoded in the prediction score. We now investigate the specific contribution of the prediction score in SSLS through comprehensive comparisons with score-free methods.

\begin{figure}[htbp]
\centering
\includegraphics[width=0.75\linewidth]{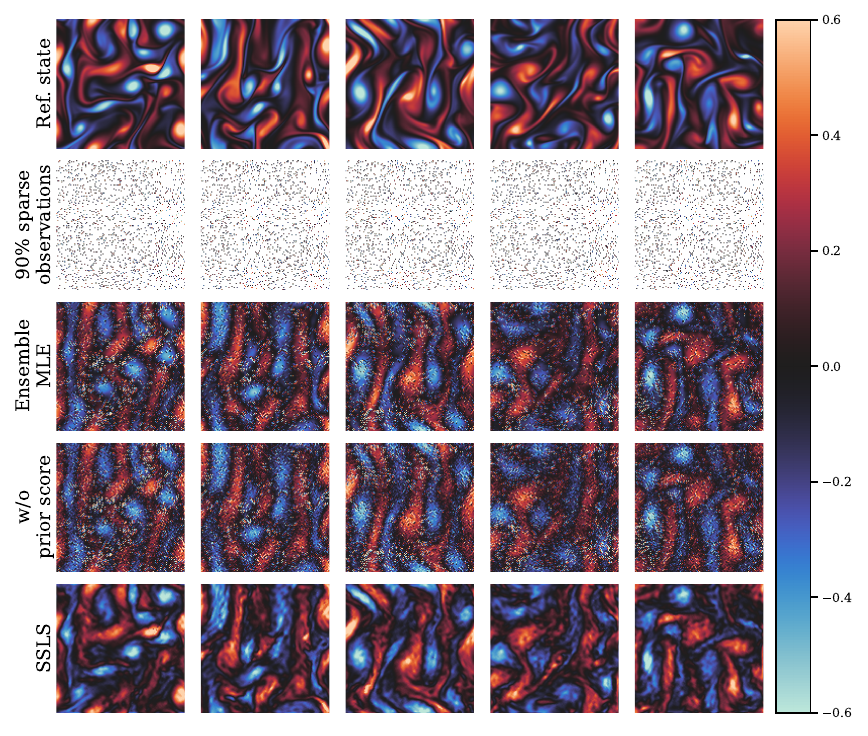}
\caption{Comparison of results of SSLS and methods without prediction score. From top to bottom: the reference state, observations, estimations of the ensemble MLE, estimations of Langevin sampling without the prediction score, and estimations of SSLS. Here the noise level is set as $\sigma_{\rm obs}=0.3$ and let 90\% points be randomly masked.}
\label{fig:KolmogorovLikelihood}
\end{figure}

\par\textbf{Baseline.}
To evaluate the impact of the prediction score in SSLS, we employ two comparative methods:
\begin{enumerate}[(i)]
\item Langevin sampling without prediction score, and
\item ensemble maximum likelihood estimation (MLE).
\end{enumerate}
Method (i) employs Langevin sampling where the drift term consists solely of the log-likelihood gradient, with the prediction score component removed. This formulation can be interpreted as a noise-augmented gradient method for maximum likelihood estimation. Method (ii) represents a pure gradient-based MLE approach, obtained by additionally removing the Gaussian noise term from the Langevin sampling formulation.

\par To ensure fair comparison, both methods are initialized using the approximated prediction distribution~\eqref{eq:prediction:distribution:hat}, consistent with the SSLS framework. Consequently, although these methods do not explicitly incorporate the prediction score, their state estimates inherently reflect the influence of historical observations through their initialization.

\par\textbf{Influence of the prediction score.}
As illustrated in Figure~\ref{fig:KolmogorovLikelihood}, both Langevin sampling without the prediction score (method (i)) and ensemble MLE (method (ii)) produce notably non-smooth state estimations. This phenomenon has a clear mathematical explanation: the log-likelihood gradient exists only at locations with available measurements and vanishes elsewhere. This gradient absence presents a fundamental challenge in assimilation with sparse or partial measurements. In such scenarios, methods (i) and (ii) can only update in measured locations, while unmeasured locations either experience random perturbations (method (i)) or remain unchanged (method (ii)). This spatially inconsistent updating mechanism leads to discontinuous and non-smooth state estimations in both methods.

\par In contrast, SSLS produces estimations that exhibit strong alignment with the reference states, attributed to the prediction score's incorporation of dynamics-based spatiotemporal correlations. This score term enables meaningful updates even at unmeasured locations by leveraging the physical constraints and dynamics embedded within it. The score effectively bridges information gaps between measured and unmeasured regions, ensuring appropriate smoothness and physical consistency throughout the domain. This capacity to maintain physical coherence while assimilating sparse measurements allows SSLS to achieve superior reconstruction quality, particularly in regions where observational data is limited or absent.

\subsubsection{Uncertainty quantification in assimilation}

\par The preceding experiments demonstrate that the states estimated by the SSLS closely align with reference states, even when significant occlusion and noise are present. However, point estimation alone proves insufficient, particularly in contexts where estimation reliability is paramount. In such high-stakes scenarios, the quantification of estimation uncertainties becomes critical for informed decision-making~\cite{Sullivan2015Introduction,Adler2025Deep}.

\par In this experiment, we will demonstrate how to quantify the uncertainties associated with the states estimated by SSLS. To this end, we consider the random reconstruction task as an example, where the observation data are obtained by masking 95\% of grid points randomly and perturbating with Gaussian noise. The experimental results are shown in Figure~\ref{fig:KolmogorovUQ}.
\begin{figure}[htbp]
\centering
\includegraphics[width=0.75\linewidth]{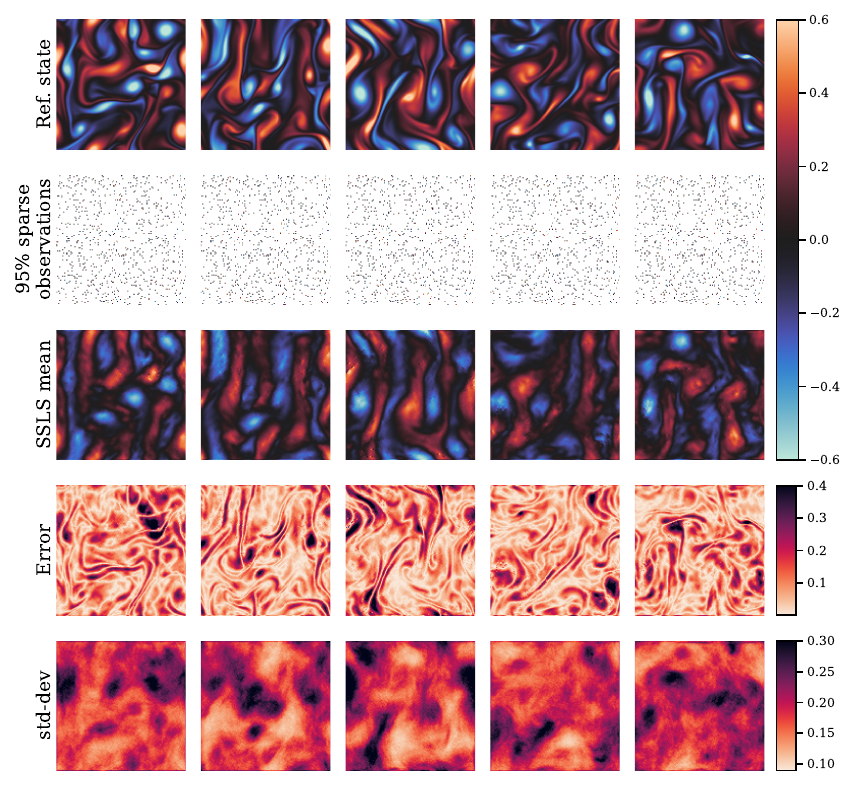}
\caption{Quantify the uncertainty associated with states estimated by SSLS. From top to bottom: the reference states, observations (95\% random mask), the SSLS assimilated states, point-wise error (in absolute value) and standard deviation. The noise level is set as $\sigma_{\rm obs}=0.4$.}
\label{fig:KolmogorovUQ}
\end{figure}

\par\textbf{Standard deviation and uncertainty.}
A notable advantage of SSLS is its ability to generate multiple ensemble samples from the posterior distribution, enabling the computation of standard deviations that illuminate the quality of the estimated states~\cite{Patel2021GAN,Adler2025Deep}. The estimated pointwise standard deviations are presented in the last row of Figure~\ref{fig:KolmogorovUQ}.

\par The first and last rows of Figure~\ref{fig:KolmogorovUQ} reveal that the standard deviation concentrates in regions of high reference vorticity magnitude. This empirical observation aligns with physical intuition. Vorticity, defined as the curl of the velocity field, quantifies the local fluid rotation. Regions of high vorticity indicate intense rotational and swirling motion in the fluid. The uniformly distributed measurement positions prove insufficient to capture high-frequency information in these high-vorticity regions, resulting in greater uncertainty in the estimated states for these areas.

\par\textbf{Well-calibrated uncertainty estimation.}
One critical measure of uncertainty estimation quality is calibration. A well-calibrated uncertainty estimation ensures that the estimated standard deviation aligns with the pointwise error~\cite{Raad2022Probabilistic,antoran2023uncertainty,Raad2024Conditional}.
\begin{figure}[htbp]
\centering
\includegraphics[width=0.75\linewidth]{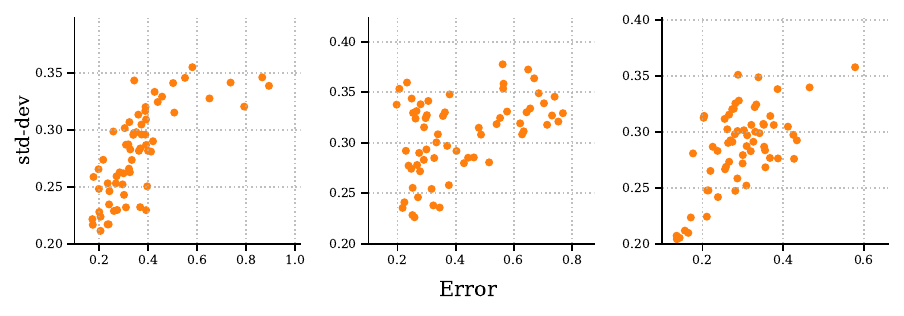}
\caption{The correlation between the standard deviation and estimation error of the SSLS. The standard deviation and error  are down-sampled by max pooling for clearer visualization. From left to right: the results at three equally separated time points of the assimilation process.}
\label{fig:KolmogorovBiasVariance}
\end{figure}

\par The last two rows of Figure~\ref{fig:KolmogorovUQ} demonstrate that the standard deviation estimated by SSLS correlates with the pointwise error. To analyze this correlation quantitatively, Figure~\ref{fig:KolmogorovBiasVariance} plots the standard deviation against the error for pixels down-sampled via max-pooling with a kernel size of 16. The correlation is examined at three equally spaced time points, revealing a consistent positive correlation between standard deviation and error across all temporal snapshots.

\par Both Figures~\ref{fig:KolmogorovUQ} and~\ref{fig:KolmogorovBiasVariance} establish that SSLS provides well-calibrated uncertainty estimations, with the computed standard deviation serving as a reliable indicator of state estimation error. This robust uncertainty quantification proves crucial for assessing the reliability of estimated states and offers valuable insights for optimizing observational positions and refining the model.


\section{Related work}
\label{section:related}

\subsection{Diffusion models for Bayesian inverse problems}
\label{section:related:diffusion}

\par Bayesian inverse problems are closely linked to data assimilation, as demonstrated in Section~\ref{section:method:update}. In recent years, diffusion methods have gained prominence as an effective technique for posterior sampling in Bayesian inverse problems, such as~\cite{chung2023diffusion,song2023pseudo,Song2023Loss,purohit2024posterior}. This section provides a comprehensive review of the existing literature on diffusion-based approaches for posterior sampling.

\par We first introduce the setup of Bayesian inverse problems. Suppose the forward model is defined as
\begin{equation*}
\mU=\calG(\mZ,\mW), \quad \mZ\sim p_{\mZ},
\end{equation*}
where $\calG$ represents a given measurement model, $p_{\mZ}$ is a known prior distribution, and $\mW$ is a random variable with a known distribution. As a result, the measurement likelihood $p_{\mU|\mZ}(\vu|\vz)$ can be obtained from the measurement model $\calG$ and the distribution of $\mW$. 

\par Bayesian inverse problems aim to estimate the posterior distribution $p_{\mZ|\mU}(\vz|\vu)$ given the observation $\vu$, the prior $p_{\mZ}$, and the measurement likelihood $p_{\mU|\mZ}$. It is evident that the update step (Section~\ref{section:method:update}) within the recursive Bayesian filtering framework exactly corresponds to solving a Bayesian inverse problem.

\par\textbf{Diffusion model with guidance.}
The mainstream technique in diffusion models for conditional or posterior sampling is guidance~\cite{guo2024gradient}. As an illustrative example, we consider the diffusion model with the following forward process
\begin{equation*}
\d\mZ_{s}=-\mZ_{s}\ds+\sqrt{2}\d\mB_{s}, \quad \mZ_{0}\sim p_{\mZ},~ s\in(0,T),
\end{equation*}
where $(\mB_{s})_{s\geq 0}$ is a standard Brownian motion. Denote by $p_{\mZ_{s}}$ the law of $\mZ_{s}$ for each $s\in(0,T)$. Following Bayes' rule, the time-reversal process for sampling from the posterior distribution $p_{\mZ|\mU}(\cdot|\vu)$ reads~\cite{chung2023diffusion}
\begin{equation}\label{eq:related:diffusion:reversal}
\begin{aligned}
\d\bar{\mZ}_{s}&=\big\{\bar{\mZ}_{s}+2\overbrace{\nabla_{\vz}\log p_{\mU|\mZ_{T-s}}(\vu|\bar{\mZ}_{s})}^{\text{intractable}}+2\overbrace{\nabla_{\vz}\log p_{\mZ_{T-s}}(\bar{\mZ}_{s})}^{\text{score}}\big\}\ds+\sqrt{2}\d\mB_{s}, \\
\bar{\mZ}_{0}&\sim N(\mathbf{0},\mI_{d}), ~ s\in(0,T),
\end{aligned} 
\end{equation}
where $p_{\mU|\mZ_{s}}(\vu|\vz)$ is the time-dependent likelihood, and $p_{\mZ_{s}}(\vz)$ is the time-dependent prior. It worth noting that the score in~\eqref{eq:related:diffusion:reversal} can be estimated by denoising score matching~\cite{Vincent2011Connection,song2021score}, while the gradient of log-likelihood is typically intractable.

\par\textbf{Linear Gaussian inverse problems.}
When the measurement model $\calG$ is linear, and both the prior and the likelihood are Gaussian, that is,
\begin{equation*}
\calG(\mZ,\mW):=\mG\mZ+\sigma_{\obs}\mW, \quad \mZ\sim p_{\mZ}=N(\bm{\mu},\mathbf{\Sigma}),~\mW\sim N(\mathbf{0},\mI_{d}),
\end{equation*}
the gradient of time-dependent log-likelihood in~\eqref{eq:related:diffusion:reversal} can be estimated without bias. For a detailed derivation, please refer to~\cite[Lemma 1]{guo2024gradient}. However, this linear scenario corresponds to the assimilation with linear state-space model. The solution to the linear assimilation can be obtained using the ensemble Kalman filter~\cite{Sarkka20232023Bayesian}. Therefore, there is no necessity to utilize diffusion models for the linear assimilation.

\par\textbf{Nonlinear inverse problems.}
Researchers' interests lie in nonlinear inverse problems, which corresponds to the nonlinear assimilation. One of the most widely-used diffusion-based approaches for nonlinear inverse problem is the diffusion posterior sampling (DPS)~\cite{chung2023diffusion}, which estimates the gradient of time-dependent log-likelihood in~\eqref{eq:related:diffusion:reversal} by exchanging the expectation with likelihood, that is,
\begin{align}
p_{\mU|\mZ_{s}}(\vu|\vz)
&=\int p_{\mU|\mZ_{0}}(\vu|\vz_{0})p_{\mZ_{0}|\mZ_{s}}(\vz_{0}|\vz)\d\vz_{0} \nonumber \\
&=\bbE\big[p_{\mU|\mZ_{0}}(\vu|\mZ_{0})|\mZ_{s}=\vz\big]\approx p_{\mU|\mZ_{0}}(\vu|\bbE[\mZ_{0}|\mZ_{s}=\vz]). \label{eq:related:diffusion:DPS}
\end{align}
However, it is important to note that this approximation introduces a bias term known as Jensen's gap, as pointed out by~\cite{chung2023diffusion}. Therefore, DPS can not yield a consistent estimation in nonlinear Bayesian inverse problems. In contrast, the score-based LMC proposed by this work is consistent in nonlinear scenarios, as shown in Theorem~\ref{theorem:section:convergence}.

\subsection{Diffusion-based methods for assimilation}

\par In the preceding subsection, we discussed diffusion-based methods for Bayesian inverse problems. Building upon these methods, a variety of diffusion-based approaches have been developed for data assimilation, which can be broadly classified into data-driven methods and filtering methods. A comparison of these techniques, along with commonly used classical methods, is presented in Table~\ref{table:comparison}.

\begin{table}[htbp]
\caption{A comparison of methods for nonlinear assimilation.}
\centering
\begin{tabular}{lcl}
\toprule 
Method & Consistent & Approximation or assumption \\ 
\midrule 
3D-Var/4D-Var & \ding{56} & ~(i) Linearization \\ 
& & (ii) Gaussian prior and likelihood \\ 
\midrule
SDA~\cite{Rozet2023score} & \ding{56} & Jensen's gap~\cite{chung2023diffusion} \\ 
SOAD~\cite{li2024state} & \ding{52} & Gaussian prior and likelihood \\  
\midrule
EnKF & \ding{56} & ~(i) Linearization \\ 
& & (ii) Gaussian prior and likelihood \\ 
PF & \ding{52} & None \\ 
SF~\cite{Bao2024score} & \ding{56} & Damping function \\ 
SSLS (ours) & \ding{52} & None \\  
\bottomrule
\end{tabular}
\label{table:comparison}
\end{table}

\par\textbf{Data-driven methods for assimilation.} 
Data-driven methods focus on estimating the posterior distribution of all latent states $\mZ:=\mX_{[k]}$ given all available observation data $\mU:=\mY_{[k]}$. By substituting $\mZ$ and $\mU$ into the time-reversal process~\eqref{eq:related:diffusion:reversal}, this line of methods reformulate the assimialtion as a single Bayesian inverse problem~\cite{Rozet2023score,rozet2023scoreqg,li2024state}. 

\par The prior score in~\eqref{eq:related:diffusion:reversal} can be approximated by utilizing random copies of $\mZ=\mX_{[k+1]}$ through denoising score matching, as discussed in~\cite{Rozet2023score,li2024state}. This characteristic makes these methods data-driven, as they rely on empirical data rather than explicit knowledge of the underlying physical models, i.e., the states transition dynamics. In these data-driven approaches, the physical mechanism is implicitly incorporated into the prior score.

\par Then it remains to estimate the gradient of time-dependent log-likelihood in~\eqref{eq:related:diffusion:reversal}. For example, the score-based data assimilation (SDA)~\cite{Rozet2023score,rozet2023scoreqg} estimates the gradient of time-dependent log-likelihood in a similar manner as DPS~\eqref{eq:related:diffusion:DPS}. In~\cite{li2024state}, the authors propose the state-observation augmented diffusion (SOAD) method, which involves converting a nonlinear state-space model into a linear one through variable augmentation. The gradient of time-dependent log-likelihood of augmented linear state-space model can be estimated without bias, as discussed in the previous section.

\par Computing the joint distribution of the states across all time steps is computationally inefficient due to the increasing dimensionality of the diffusion model as the number of time steps increases. This limitation of the data-driven methods hinders the application of these methods to high-dimensional data assimilation tasks. Additionally, data-driven methods are unable to fully leverage the underlying physical mechanisms. In contrast, the dimensionality of SSLS proposed in this study remains constant, regardless of the assimilation time. Moreover, SSLS effectively integrates physical principles with observation data.

\par\textbf{Filtering methods for assimilation.}
Another category of diffusion-based assimilation methods are developed within the Bayesian filtering framework (Section~\ref{section:method:Bayes}), with examples including the score-based filter (SF)~\cite{Bao2024score}. In the update at the $(k+1)$-th time step, SF sets $\mU:=\mY_{k+1}$, $\vu:=\vy_{k+1}$, and $\mZ_{0}\sim\what{q}_{k+1}(\cdot|\vy{[k]})$. SF estimates the prior score in\eqref{eq:related:diffusion:reversal} using sliced score matching~\cite{song2020Sliced}, while approximating the gradient of the time-dependent log-likelihood through a separation of variables:
\begin{equation*}
\nabla_{\vz}\log p_{\mU|\mZ_{s}}(\vu|\vz)=h(s)\nabla_{\vz}\log p_{\mU|\mZ_{0}}(\vu|\vz), \quad s\in(0,T),
\end{equation*}
Here, the damping function $h$ is a monotonically decreasing function in the interval $[0,T]$, with $h(0)=1$ and $h(T)=0$. However, this separation of variable approximation is often inconsistent, and the optimal choice of the damping function remains unresolved. In contrast, our method is consistent and does not rely on such heuristic approximations.


\section{Conclusions}
\label{section:conclusion}

\par This paper presents score-based sequential Langevin sampling, a novel approach for nonlinear assimilation within a recursive Bayesian filtering framework. The theoretical analysis establishes SSLS convergence in TV-distance under mild conditions, providing insights into error behavior with respect to hyper-parameters. Extensive numerical experiments demonstrate SSLS's exceptional performance in high-dimensional and nonlinear scenarios, particularly with sparse or partial measurements. Furthermore, SSLS effectively quantifies state estimation uncertainty, enabling error calibration.

\par Several promising directions exist for future research and methodological enhancement. While the current framework assumes a known state-space model, practical applications often involve uncertain parameters in both dynamics and measurement models. Future work will extend SSLS to enable simultaneous estimation of latent states and model parameters. Additionally, we aim to address the computational burden of score network training at each time step through techniques such as in-context learning. 

\bibliographystyle{abbrv}
\bibliography{reference}

\begin{thebibliography}{10}

\bibitem{Adler2025Deep}
J.~Adler and O.~{\"O}ktem.
\newblock Deep {B}ayesian inversion.
\newblock In T.~A. Bubba, editor, {\em Data-driven Models in Inverse Problems}, volume~31 of {\em Radon Series on Computational and Applied Mathematics}, pages 359--412. Berlin, Boston: De Gruyter, 2025.

\bibitem{Anderson2009Spatially}
J.~L. Anderson.
\newblock Spatially and temporally varying adaptive covariance inflation for ensemble filters.
\newblock {\em Tellus A}, 61(1):72--83, 2009.

\bibitem{Anderson1999Monte}
J.~L. Anderson and S.~L. Anderson.
\newblock A {M}onte {C}arlo implementation of the nonlinear filtering problem to produce ensemble assimilations and forecasts.
\newblock {\em Monthly Weather Review}, 121:2741--2758, 1999.

\bibitem{antoran2023uncertainty}
J.~Antoran, R.~Barbano, J.~Leuschner, J.~M. Hern{\'a}ndez-Lobato, and B.~Jin.
\newblock Uncertainty estimation for computed tomography with a linearised deep image prior.
\newblock {\em Transactions on Machine Learning Research}, 2023.

\bibitem{Bain2009Fundamentals}
A.~Bain and D.~Crisan.
\newblock {\em Fundamentals of Stochastic Filtering}, volume~60 of {\em Stochastic Modelling and Applied Probability (SMAP)}.
\newblock Springer New York, NY, first edition, 2009.

\bibitem{Bakr2014Analysis}
D.~Bakr, I.~Gentil, and M.~Ledoux.
\newblock {\em Analysis and Geometry of Markov Diffusion Operators}, volume 348 of {\em Grundlehren der mathematischen Wissenschaften (GL)}.
\newblock Springer Cham, first edition, 2014.

\bibitem{Bakry1985Diffusions}
D.~Bakry and M.~{\'E}mery.
\newblock Diffusions hypercontractives.
\newblock In J.~Az{\'e}ma and M.~Yor, editors, {\em S{\'e}minaire de Probabilit{\'e}s XIX 1983/84}, pages 177--206. Springer Berlin Heidelberg, 1985.

\bibitem{Bao2024score}
F.~Bao, Z.~Zhang, and G.~Zhang.
\newblock A score-based filter for nonlinear data assimilation.
\newblock {\em Journal of Computational Physics}, 514:113207, 2024.

\bibitem{Bengtsson2008Curse}
T.~Bengtsson, P.~Bickel, and B.~Li.
\newblock Curse-of-dimensionality revisited: {C}ollapse of the particle filter in very large scale systems.
\newblock In D.~Nolan and T.~Speed, editors, {\em Institute of Mathematical Statistics Collections}, Probability and Statistics: Essays in Honor of David A. Freedman, pages 316--334. Institute of Mathematical Statistics, 2008.

\bibitem{Beskos2015Sequential}
A.~Beskos, A.~Jasra, E.~A. Muzaffer, and A.~M. Stuart.
\newblock Sequential {M}onte {C}arlo methods for {B}ayesian elliptic inverse problems.
\newblock {\em Statistics and Computing}, 25:727--737, 2015.

\bibitem{Bhar2010Stochastic}
R.~Bhar.
\newblock {\em Stochastic Filtering with Applications in Finance}.
\newblock World Scientific, 2010.

\bibitem{Bickel2008Sharp}
P.~Bickel, B.~Li, and T.~Bengtsson.
\newblock Sharp failure rates for the bootstrap particle filter in high dimensions.
\newblock In B.~Clarke and S.~Ghosal, editors, {\em Institute of Mathematical Statistics Collections}, Pushing the Limits of Contemporary Statistics: Contributions in Honor of Jayanta K. Ghosh, pages 318--329. Institute of Mathematical Statistics, 2008.

\bibitem{Brocker2012Evaluating}
J.~Br{\"o}cker.
\newblock Evaluating raw ensembles with the continuous ranked probability score.
\newblock {\em Quarterly Journal of the Royal Meteorological Society}, 138(667):1611--1617, 2012.

\bibitem{Brosse2018Normalizing}
N.~Brosse, A.~Durmus, and {\'E}.~Moulines.
\newblock Normalizing constants of log-concave densities.
\newblock {\em Electronic Journal of Statistics}, 12(1):851 -- 889, 2018.

\bibitem{chang2025provable}
J.~Chang, C.~Duan, Y.~Jiao, R.~Li, J.~Z. Yang, and C.~Yuan.
\newblock Provable diffusion posterior sampling for {B}ayesian inversion, 2025.
\newblock arXiv:2512.08022.

\bibitem{Chen2011Dimension}
H.-B. Chen, S.~Chewi, and J.~Niles-Weed.
\newblock Dimension-free log-{S}obolev inequalities for mixture distributions.
\newblock {\em Journal of Functional Analysis}, 281(11):109236, 2021.

\bibitem{chen2023sampling}
S.~Chen, S.~Chewi, J.~Li, Y.~Li, A.~Salim, and A.~Zhang.
\newblock Sampling is as easy as learning the score: theory for diffusion models with minimal data assumptions.
\newblock In {\em The Eleventh International Conference on Learning Representations}, 2023.

\bibitem{Chewi2024log}
S.~Chewi.
\newblock Log-concave sampling, 2024.
\newblock unfinished draft.

\bibitem{Chewi2024Analysis}
S.~Chewi, M.~A. Erdogdu, M.~Li, R.~Shen, and M.~S. Zhang.
\newblock Analysis of {L}angevin {M}onte {C}arlo from {P}oincar{\'e} to log-{S}obolev.
\newblock {\em Foundations of Computational Mathematics}, 2024.

\bibitem{chung2023diffusion}
H.~Chung, J.~Kim, M.~T. Mccann, M.~L. Klasky, and J.~C. Ye.
\newblock Diffusion posterior sampling for general noisy inverse problems.
\newblock In {\em The Eleventh International Conference on Learning Representations}, 2023.

\bibitem{Cotter2009Bayesian}
S.~L. Cotter, M.~Dashti, J.~C. Robinson, and A.~M. Stuart.
\newblock Bayesian inverse problems for functions and applications to fluid mechanics.
\newblock {\em Inverse Problems}, 25(11):115008, 2009.

\bibitem{Crisan2002survey}
D.~Crisan and A.~Doucet.
\newblock A survey of convergence results on particle filtering methods for practitioners.
\newblock {\em IEEE Transactions on Signal Processing}, 50(3):736--746, 2002.

\bibitem{DelMoral2006Sequential}
P.~Del~Moral, A.~Doucet, and A.~Jasra.
\newblock Sequential {M}onte {C}arlo samplers.
\newblock {\em Journal of the Royal Statistical Society Series B: Statistical Methodology}, 68(3):411--436, 05 2006.

\bibitem{ding2024characteristic}
Z.~Ding, C.~Duan, Y.~Jiao, R.~Li, J.~Z. Yang, and P.~Zhang.
\newblock Characteristic learning for provable one step generation, 2024.
\newblock arXiv:2405.05512.

\bibitem{Doucet2001Sequential}
A.~Doucet, N.~Freitas, and N.~Gordon, editors.
\newblock {\em Sequential Monte Carlo Methods in Practice}.
\newblock Information Science and Statistics (ISS). Springer New York, NY, first edition, 2001.

\bibitem{Duchi2024Information}
J.~C. Duchi.
\newblock Information theory and statistics, 2024.
\newblock unfinished draft.

\bibitem{Efron2011Tweedie}
B.~Efron.
\newblock {T}weedie's formula and selection bias.
\newblock {\em Journal of the American Statistical Association}, 106(496):1602--1614, 2011.

\bibitem{Elliott2013option}
R.~J. Elliott and T.~K. Siu.
\newblock Option pricing and filtering with hidden {M}arkov-modulated pure-jump processes.
\newblock {\em Applied Mathematical Finance}, 20(1):1--25, 2013.

\bibitem{Evensen2022Data}
G.~Evensen, F.~C. Vossepoel, and P.~J. van Leeuwen.
\newblock {\em Data Assimilation Fundamentals: A Unified Formulation of the State and Parameter Estimation Problem}.
\newblock Springer Textbooks in Earth Sciences, Geography and Environment (STEGE). Springer Cham, first edition, 2022.

\bibitem{Rudiger2012Pricing}
R.~Frey and T.~Schmidt.
\newblock Pricing and hedging of credit derivatives via the innovations approach to nonlinear filtering.
\newblock {\em Finance and Stochastics}, 16:105--133, 2012.

\bibitem{Ge2020Estimating}
R.~Ge, H.~Lee, and J.~Lu.
\newblock Estimating normalizing constants for log-concave distributions: algorithms and lower bounds.
\newblock In {\em Proceedings of the 52nd Annual ACM SIGACT Symposium on Theory of Computing}, STOC 2020, pages 579--586. Association for Computing Machinery, 2020.

\bibitem{Gneiting2007Probabilistic}
T.~Gneiting, F.~Balabdaoui, and A.~E. Raftery.
\newblock Probabilistic forecasts, calibration and sharpness.
\newblock {\em Journal of the Royal Statistical Society: Series B (Statistical Methodology)}, 69(2):243--268, 2007.

\bibitem{Gordon1993Novel}
N.~Gordon, D.~Salmond, and A.~Smith.
\newblock Novel approach to nonlinear/non-{G}aussian {B}ayesian state estimation.
\newblock {\em IEE Proceedings on Radar and Signal Processing}, 140:107--113, 1993.

\bibitem{Grenioux2024Stochastic}
L.~Grenioux, M.~Noble, M.~Gabri\'{e}, and A.~Oliviero~Durmus.
\newblock Stochastic localization via iterative posterior sampling.
\newblock In {\em Proceedings of the 41st International Conference on Machine Learning}, volume 235 of {\em Proceedings of Machine Learning Research}, pages 16337--16376. PMLR, 2024.

\bibitem{guo2024gradient}
Y.~Guo, H.~Yuan, Y.~Yang, M.~Chen, and M.~Wang.
\newblock Gradient guidance for diffusion models: {A}n optimization perspective, 2024.
\newblock arXiv:2404.14743.

\bibitem{Ho2020Denoising}
J.~Ho, A.~Jain, and P.~Abbeel.
\newblock Denoising diffusion probabilistic models.
\newblock In H.~Larochelle, M.~Ranzato, R.~Hadsell, M.~Balcan, and H.~Lin, editors, {\em Advances in Neural Information Processing Systems}, volume~33, pages 6840--6851. Curran Associates, Inc., 2020.

\bibitem{houtekamer1998data}
P.~L. Houtekamer and H.~L. Mitchell.
\newblock Data assimilation using an ensemble {K}alman filter technique.
\newblock {\em Monthly weather review}, 126(3):796--811, 1998.

\bibitem{hyvarinen2005Estimation}
A.~Hyv{\"a}rinen.
\newblock Estimation of non-normalized statistical models by score matching.
\newblock {\em Journal of Machine Learning Research}, 6(24):695--709, 2005.

\bibitem{Jalal2021Robust}
A.~Jalal, M.~Arvinte, G.~Daras, E.~Price, A.~G. Dimakis, and J.~Tamir.
\newblock Robust compressed sensing {MRI} with deep generative priors.
\newblock In M.~Ranzato, A.~Beygelzimer, Y.~Dauphin, P.~Liang, and J.~W. Vaughan, editors, {\em Advances in Neural Information Processing Systems}, volume~34, pages 14938--14954. Curran Associates, Inc., 2021.

\bibitem{Jiao2023deep}
Y.~Jiao, G.~Shen, Y.~Lin, and J.~Huang.
\newblock Deep nonparametric regression on approximate manifolds: {N}onasymptotic error bounds with polynomial prefactors.
\newblock {\em The Annals of Statistics}, 51(2):691 -- 716, 2023.

\bibitem{Kantas2014Sequential}
N.~Kantas, A.~Beskos, and A.~Jasra.
\newblock Sequential {M}onte {C}arlo methods for high-dimensional inverse problems: {A} case study for the {N}avier-{S}tokes equations.
\newblock {\em SIAM/ASA Journal on Uncertainty Quantification}, 2(1):464--489, 2014.

\bibitem{Katsafados2020Numerical}
P.~Katsafados, E.~Mavromatidis, and C.~Spyrou.
\newblock {\em Numerical Weather Prediction and Data Assimilation}.
\newblock John Wiley \& Sons, Ltd, 2020.

\bibitem{Kelly2014Well}
D.~T.~B. Kelly, K.~J.~H. Law, and A.~M. Stuart.
\newblock Well-posedness and accuracy of the ensemble {K}alman filter in discrete and continuous time.
\newblock {\em Nonlinearity}, 27(10):2579, sep 2014.

\bibitem{Kitagawa1996Monte}
G.~Kitagawa.
\newblock Monte {C}arlo filter and smoother for non-{G}aussian nonlinear state space models.
\newblock {\em Journal of Computational and Graphical Statistics}, 5(1):1--25, 1996.

\bibitem{Kochkov2021Machine}
D.~Kochkov, J.~A. Smith, A.~Alieva, Q.~Wang, M.~P. Brenner, and S.~Hoyer.
\newblock Machine learning-accelerated computational fluid dynamics.
\newblock {\em Proceedings of the National Academy of Sciences}, 118(21):e2101784118, 2021.

\bibitem{kohler2021rate}
M.~Kohler and S.~Langer.
\newblock On the rate of convergence of fully connected deep neural network regression estimates.
\newblock {\em The Annals of Statistics}, 49(4):2231--2249, 2021.

\bibitem{Law2015Data}
K.~Law, A.~Stuart, and K.~Zygalakis.
\newblock {\em Data Assimilation: A Mathematical Introduction}, volume~62 of {\em Texts in Applied Mathematics (TAM)}.
\newblock Springer Cham, first edition, 2015.

\bibitem{LeDimet1986Variational}
F.-X. Le~Dimet and O.~Talagrand.
\newblock Variational algorithms for analysis and assimilation of meteorological observations: theoretical aspects.
\newblock {\em Tellus A}, 38A(2):97--110, 1986.

\bibitem{Lee2022Convergence}
H.~Lee, J.~Lu, and Y.~Tan.
\newblock Convergence for score-based generative modeling with polynomial complexity.
\newblock In S.~Koyejo, S.~Mohamed, A.~Agarwal, D.~Belgrave, K.~Cho, and A.~Oh, editors, {\em Advances in Neural Information Processing Systems}, volume~35, pages 22870--22882. Curran Associates, Inc., 2022.

\bibitem{li2024state}
Z.~Li, B.~Dong, and P.~Zhang.
\newblock State-observation augmented diffusion model for nonlinear assimilation with unknown dynamics, 2025.

\bibitem{Majda2012Filtering}
A.~J. Majda and J.~Harlim.
\newblock {\em Filtering Complex Turbulent Systems}.
\newblock Cambridge University Press, 2012.

\bibitem{Mandel2012convergence}
J.~Mandel, L.~Cobb, and J.~D. Beezley.
\newblock On the convergence of the ensemble {K}alman filter.
\newblock {\em Applications of Mathematics}, 56:533--541, 2012.

\bibitem{Moral2004Feynman}
P.~Moral.
\newblock {\em Feynman-Kac Formulae: Genealogical and Interacting Particle Systems with Applications}.
\newblock Probability and Its Applications (PIA). Springer New York, NY, first edition, 2004.

\bibitem{Patel2021GAN}
D.~V. Patel and A.~A. Oberai.
\newblock {GAN}-based priors for quantifying uncertainty in supervised learning.
\newblock {\em SIAM/ASA Journal on Uncertainty Quantification}, 9(3):1314--1343, 2021.

\bibitem{Pavliotis2014Stochastic}
G.~A. Pavliotis.
\newblock {\em Stochastic Processes and Applications: Diffusion Processes, the Fokker-Planck and Langevin Equations}, volume~60 of {\em Texts in Applied Mathematics (TAM)}.
\newblock Springer New York, NY, 2014.

\bibitem{pitt1999filtering}
M.~K. Pitt and N.~Shephard.
\newblock Filtering via simulation: {A}uxiliary particle filters.
\newblock {\em Journal of the American statistical association}, 94(446):590--599, 1999.

\bibitem{purohit2024posterior}
V.~Purohit, M.~Repasky, J.~Lu, Q.~Qiu, Y.~Xie, and X.~Cheng.
\newblock Posterior sampling via {L}angevin dynamics based on generative priors, 2024.
\newblock arXiv:2410.02078.

\bibitem{Raad2022Probabilistic}
R.~Raad, D.~Patel, C.-C. Hsu, V.~Kothapalli, D.~Ray, B.~Varghese, D.~Hwang, I.~Gill, V.~Duddalwar, and A.~A. Oberai.
\newblock Probabilistic medical image imputation via deep adversarial learning.
\newblock {\em Engineering with Computers}, 38:3975--3986, 2022.

\bibitem{Raad2024Conditional}
R.~Raad, D.~Ray, B.~Varghese, D.~Hwang, I.~Gill, V.~Duddalwar, and A.~A. Oberai.
\newblock Conditional generative learning for medical image imputation.
\newblock {\em Scientific Reports}, 14(171), 2024.

\bibitem{Rassoul2015Course}
F.~Rassoul-Agha and T.~Sepp{\"a}l{\"a}inen.
\newblock {\em A Course on Large Deviations with an Introduction to Gibbs Measures}, volume 162 of {\em Graduate Studies in Mathematics}.
\newblock American Mathematical Society (AMS), 2015.

\bibitem{Reich2019Data}
S.~Reich.
\newblock Data assimilation: {T}he {S}chr{\"o}dinger perspective.
\newblock {\em Acta Numerica}, 28:635--711, 2019.

\bibitem{Reich2015Probabilistic}
S.~Reich and C.~Cotter.
\newblock {\em Probabilistic Forecasting and Bayesian Data Assimilation}.
\newblock Cambridge University Press, 2015.

\bibitem{Rozet2023score}
F.~Rozet and G.~Louppe.
\newblock Score-based data assimilation.
\newblock In A.~Oh, T.~Naumann, A.~Globerson, K.~Saenko, M.~Hardt, and S.~Levine, editors, {\em Advances in Neural Information Processing Systems}, volume~36, pages 40521--40541. Curran Associates, Inc., 2023.

\bibitem{rozet2023scoreqg}
F.~Rozet and G.~Louppe.
\newblock Score-based data assimilation for a two-layer quasi-geostrophic model, 2023.
\newblock arXiv:2310.01853.

\bibitem{Sacher2008Sampling}
W.~Sacher and P.~Bartello.
\newblock Sampling errors in ensemble {K}alman filtering. {P}art {I}: {T}heory.
\newblock {\em Monthly Weather Review}, 136:3035--3049, 2008.

\bibitem{Sarkka20232023Bayesian}
S.~S{\"a}rkk{\"a} and L.~Svensson.
\newblock {\em Bayesian filtering and smoothing}.
\newblock Institute of Mathematical Statistics Textbooks. Cambridge University Press, second edition, 2023.

\bibitem{schmidt2020nonparametric}
J.~Schmidt-Hieber.
\newblock Nonparametric regression using deep neural networks with {ReLU} activation function.
\newblock {\em The Annals of Statistics}, 48(4):1875--1897, 2020.

\bibitem{si2024latent}
P.~Si and P.~Chen.
\newblock {Latent-EnSF}: {A} latent ensemble score filter for high-dimensional data assimilation with sparse observation data, 2024.
\newblock arXiv:2409.00127.

\bibitem{Snyder2008Obstacles}
C.~Snyder, T.~Bengtsson, P.~Bickel, and J.~Anderson.
\newblock Obstacles to high-dimensional particle filtering.
\newblock {\em Monthly weather review}, 136:4629--4640, 2008.

\bibitem{song2023pseudo}
J.~Song, A.~Vahdat, M.~Mardani, and J.~Kautz.
\newblock Pseudoinverse-guided diffusion models for inverse problems.
\newblock In {\em International Conference on Learning Representations}, 2023.

\bibitem{Song2023Loss}
J.~Song, Q.~Zhang, H.~Yin, M.~Mardani, M.-Y. Liu, J.~Kautz, Y.~Chen, and A.~Vahdat.
\newblock Loss-guided diffusion models for plug-and-play controllable generation.
\newblock In A.~Krause, E.~Brunskill, K.~Cho, B.~Engelhardt, S.~Sabato, and J.~Scarlett, editors, {\em Proceedings of the 40th International Conference on Machine Learning}, volume 202 of {\em Proceedings of Machine Learning Research}, pages 32483--32498. PMLR, 23--29 Jul 2023.

\bibitem{Song2019Generative}
Y.~Song and S.~Ermon.
\newblock Generative modeling by estimating gradients of the data distribution.
\newblock In H.~Wallach, H.~Larochelle, A.~Beygelzimer, F.~d\textquotesingle Alch\'{e}-Buc, E.~Fox, and R.~Garnett, editors, {\em Advances in Neural Information Processing Systems}, volume~32. Curran Associates, Inc., 2019.

\bibitem{song2020Sliced}
Y.~Song, S.~Garg, J.~Shi, and S.~Ermon.
\newblock Sliced score matching: {A} scalable approach to density and score estimation.
\newblock In R.~P. Adams and V.~Gogate, editors, {\em Proceedings of The 35th Uncertainty in Artificial Intelligence Conference}, volume 115 of {\em Proceedings of Machine Learning Research}, pages 574--584. PMLR, 22--25 Jul 2020.

\bibitem{song2021score}
Y.~Song, J.~Sohl-Dickstein, D.~P. Kingma, A.~Kumar, S.~Ermon, and B.~Poole.
\newblock Score-based generative modeling through stochastic differential equations.
\newblock In {\em International Conference on Learning Representations}, 2021.

\bibitem{Spantini2022Coupling}
A.~Spantini, R.~Baptista, and Y.~Marzouk.
\newblock Coupling techniques for nonlinear ensemble filtering.
\newblock {\em SIAM Review}, 64(4):921--953, 2022.

\bibitem{Sullivan2015Introduction}
T.~Sullivan.
\newblock {\em Introduction to Uncertainty Quantification}, volume~63 of {\em Texts in Applied Mathematics (TAM)}.
\newblock Springer Cham, first edition, 2015.

\bibitem{Tang2024Adaptivity}
R.~Tang and Y.~Yang.
\newblock Adaptivity of diffusion models to manifold structures.
\newblock In S.~Dasgupta, S.~Mandt, and Y.~Li, editors, {\em Proceedings of The 27th International Conference on Artificial Intelligence and Statistics}, volume 238 of {\em Proceedings of Machine Learning Research}, pages 1648--1656. PMLR, 02--04 May 2024.

\bibitem{Thelen2022comprehensive1}
A.~Thelen, X.~Zhang, O.~Fink, Y.~Lu, S.~Ghosh, D.~Byeng~Youn, M.~D. Todd, S.~Mahadevan, C.~Hu, and Z.~Hu.
\newblock A comprehensive review of digital twin -- part 1: modeling and twinning enabling technologies.
\newblock {\em Structural and Multidisciplinary Optimization}, 65(354), 2022.

\bibitem{Thelen2022comprehensive2}
A.~Thelen, X.~Zhang, O.~Fink, Y.~Lu, S.~Ghosh, D.~Byeng~Youn, M.~D. Todd, S.~Mahadevan, C.~Hu, and Z.~Hu.
\newblock A comprehensive review of digital twin -- part 2: roles of uncertainty quantification and optimization, a battery digital twin, and perspectives.
\newblock {\em Structural and Multidisciplinary Optimization}, 66(1), 2023.

\bibitem{Tong2016Nonlinear}
X.~T. Tong, A.~J. Majda, and D.~Kelly.
\newblock Nonlinear stability and ergodicity of ensemble based {K}alman filters.
\newblock {\em Nonlinearity}, 29(2):657, jan 2016.

\bibitem{Tsybakov2009Introduction}
A.~B. Tsybakov.
\newblock {\em Introduction to Nonparametric Estimation}.
\newblock Springer Series in Statistics (SSS). Springer New York, NY, first edition, 2009.

\bibitem{Vempala2019Rapid}
S.~Vempala and A.~Wibisono.
\newblock Rapid convergence of the unadjusted {L}angevin algorithm: {I}soperimetry suffices.
\newblock In H.~Wallach, H.~Larochelle, A.~Beygelzimer, F.~d\textquotesingle Alch\'{e}-Buc, E.~Fox, and R.~Garnett, editors, {\em Advances in Neural Information Processing Systems}, volume~32. Curran Associates, Inc., 2019.

\bibitem{Villani2009Optimal}
C.~Villani.
\newblock {\em Optimal Transport: Old and New}, volume 338 of {\em Grundlehren der mathematischen Wissenschaften (GL)}.
\newblock Springer Berlin, Heidelberg, first edition, 2009.

\bibitem{Vincent2011Connection}
P.~Vincent.
\newblock A connection between score matching and denoising autoencoders.
\newblock {\em Neural Computation}, 23(7):1661--1674, 2011.

\bibitem{Wainwright2019high}
M.~J. Wainwright.
\newblock {\em High-Dimensional Statistics: A Non-Asymptotic Viewpoint}.
\newblock Cambridge Series in Statistical and Probabilistic Mathematics. Cambridge University Press, 2019.

\bibitem{wu2024annealing}
D.~Wu and Y.~Xie.
\newblock Annealing flow generative models towards sampling high-dimensional and multi-modal distributions, 2025.

\bibitem{yang2025generative}
S.~Yang, C.~Nai, X.~Liu, W.~Li, J.~Chao, J.~Wang, L.~Wang, X.~Li, X.~Chen, B.~Lu, et~al.
\newblock Generative assimilation and prediction for weather and climate.
\newblock {\em arXiv preprint arXiv:2503.03038}, 2025.

\bibitem{zheng2025ensemble}
H.~Zheng, W.~Chu, A.~Wang, N.~B. Kovachki, R.~Baptista, and Y.~Yue.
\newblock Ensemble {K}alman diffusion guidance: {A} derivative-free method for inverse problems.
\newblock {\em Transactions on Machine Learning Research}, 2025.

\end{thebibliography}

\appendix


\section*{Outline of Appendices}
\par The supplementary material comprises several appendices containing notation summaries, additional derivations, theoretical proofs, experimental results, and implementation details:
\begin{enumerate}[(I)]
\item Appendix~\ref{appendix:notations} provides a summary of notation used throughout the paper.
\item Appendix~\ref{appendix:method} details the recursive Bayesian filtering framework underlying our approach and presents derivations of denoising score matching.
\item Appendices~\ref{appendix:convergence} to~\ref{appendix:auxiliary} collect the proofs of all theoretical results.
\begin{enumerate}
\item Appendix~\ref{appendix:convergence} proves the main results of Section~\ref{section:convergence}; 
supporting results are established in the subsequent appendices.
\item Appendix~\ref{appendix:convergence:decomposition} decomposes the posterior sampling error 
into three components: the Langevin Monte Carlo convergence error, the prior error, 
and the score estimation error.
\item Appendix~\ref{appendix:LMC} establishes a convergence rate for Langevin Monte Carlo.
\item Appendix~\ref{appendix:prior} bounds the prior error.
\item Appendix~\ref{appendix:score} derives a bound on the score estimation error.
\item Appendix~\ref{appendix:init} provides an error bound for posterior sampling 
at the initial time step.
\item Appendix~\ref{appendix:wasserstein} establishes relation between TV error and Wasserstein error.
\item Appendix~\ref{appendix:auxiliary} collects auxiliary lemmas used throughout the proofs.
\end{enumerate}
\item Appendix~\ref{section:experiments:appendix} presents additional numerical experiments:
\begin{enumerate}
    \item Section~\ref{section:experiment:appendix:linearGaussian} evaluates SSLS on a linear Gaussian state-space model, where analytical solution is given by Kalman filter. 
    \item Section~\ref{section:experiments:appendix:Lorenz} evaluates SSLS on the benchmark Lorenz-96 problem, and compares it with APF.
    \item Section~\ref{section:experiments:appendix:Kolmogorov} compares SSLS with the baseline EnKG method \cite{zheng2025ensemble}.
    \item Section~\ref{section:experiments:appendix:budget} reports the detailed computational time and memory load of different assimilation methods.
    \item Section~\ref{appendix:experiment:sensitivity} reports the sensitivity analysis of the hyper-parameters of SSLS, consistent with the theoretical results.
\end{enumerate}
\item The implementation details are documented in Appendix~\ref{appendix:experiment:details}.
\end{enumerate}

\section{A Summary of Notations}\label{appendix:notations}

\par Table~\ref{table:notations} summarizes the notations used in Sections~\ref{section:method} and~\ref{section:convergence} for easy reference and cross-checking.
\begin{table}[htbp]
\caption{The list of notations defined in Sections~\ref{section:method} and~\ref{section:convergence}.}
\centering 
\begin{tabular}{cl}
\toprule 
Symbols & Description \\ 
\midrule 
$\calF_{k}$ & The dynamics model at the $k$-th time step, defined as~\eqref{eq:dynamic}. \\
$\calG_{k}$ & The measurement model at the $k$-th time step, defined as~\eqref{eq:measurement}. \\ 
\midrule 
$\mX_{k}$ & The state at the $k$-th time step. \\ 
$\mY_{k}$ & The observation of $\mX_{k}$. \\ 
$\what{\mX}_{k}$ & The estimated state at the $k$-th time step using SSLS. \\
$\underline{\mX}_{k}$ & \makecell[l]{The predicted state at the $k$-th time step using the dynamics model,\\ defined as~\eqref{eq:predict}.} \\  
\midrule
$\rho_{k}$ & The state transition density at the $k$-th time step, specified by~\eqref{eq:dynamic}. \\ 
$g_{k}$ & The measurement likelihood at the $k$-th time step, specified by~\eqref{eq:measurement}. \\
$\pi_{k}$ & The posterior at the $k$-th time step, defined as~\eqref{eq:assimilation}. \\
$\what{\pi}_{k}$ or $\what{\pi}_{k}^{T}$ & \makecell[l]{The estimated posterior at the $k$-th time step using SSLS \\ with terminal time $T$, which is the law of $\what{\mX}_{k}$.}. \\
$q_{k}$ & \makecell[l]{The prediction distribution  at the $k$-th time step, defined as~\eqref{eq:prediction:distribution}, \\ serving as the prior in the posterior $\pi_{k}$.} \\
$\what{q}_{k}$ & \makecell[l]{The approximated prediction distribution at the $k$-th time step, \\ defined as~\eqref{eq:prediction:distribution:hat}, which is the law of $\underline{\mX}_{k}$.} \\
\midrule
$\mZ_{t}$ & The stochastic process specified by the Langevin diffusion~\eqref{eq:method:update:LD}. \\
$\what{\mZ}_{t}$ & The stochastic process specified by the score-based Langevin sampling~\eqref{eq:section:convergence:SLMC}. \\
$\what{\vb}_{k}$ & The drift term in the score-based Langevin Monte Carlo~\eqref{eq:method:update:posterior:score}. \\
\midrule
$(\beta_{m})_{m=1}^{M}$ & A sequence of inverse temperatures for annealing. \\
$\what{\mZ}_{t}^{m}$ & \makecell[l]{The stochastic process specified by the annealed Langevin \\ Monte Carlo~\eqref{eq:method:update:ALMC} with a inverse temperature $\beta_{m}$.}  \\
$\what{\vb}_{k}^{m}$ & \makecell[l]{The drift term in the score-based annealed Langevin \\ Monte Carlo~\eqref{eq:method:update:ALMC:score} with a inverse temperature $\beta_{m}$.} \\
\bottomrule
\end{tabular}
\label{table:notations}
\end{table}

\section{Proofs in Section~\ref{section:method}}
\label{appendix:method}

\par In this section, we provide proofs in Section~\ref{section:method}. The derivation of the recursive Bayesian framework is shown in Appendix~\ref{appendix:method:recursive}, and the proof of the denoising score matching is demonstrated in Appendix~\ref{appendix:method:sm}. 

\subsection{Recursive Bayesian framework}
\label{appendix:method:recursive}

\par This section verifies the recursion~\eqref{eq:prediction:recursion}. Indeed,
\begin{align*}
&\pi_{k+1}(\vx_{k+1}|\vy_{[k+1]}) \\
&=\frac{p_{\mY_{k+1}|\mX_{k+1},\mY_{[k]}}(\vy_{k+1}|\vx_{k+1},\vy_{[k]})}{p_{\mY_{k+1}|\mY_{[k]}}(\vy_{k+1}|\vy_{[k]})}p_{\mX_{k+1}|\mY_{[k]}}(\vx_{k+1}|\vy_{[k]}) \\
&=\frac{g_{k+1}(\vy_{k+1}|\vx_{k+1})}{p_{\mY_{k+1}|\mY_{[k]}}(\vy_{k+1}|\vy_{[k]})}\int p_{\mX_{k+1}|\mX_{k},\mY_{[k]}}(\vx_{k+1}|\vx_{k},\vy_{[k]})\pi_{k}(\vx_{k}|\vy_{[k]})\d\vx_{k} \\
&=\frac{g_{k+1}(\vy_{k+1}|\vx_{k+1})}{p_{\mY_{k+1}|\mY_{[k]}}(\vy_{k+1}|\vy_{[k]})}\int\rho_{k}(\vx_{k+1}|\vx_{k})\pi_{k}(\vx_{k}|\vy_{[k]})\d\vx_{k},
\end{align*}
where the first equality follows from Bayes's rule. The second equality invokes Chapman-Kolmogorov identity and the fact that $\mY_{k+1}$ is independent of $\mY_{[k]}$ given $\mX_{k+1}$. The last equality is owing to the fact that $\mX_{k+1}$ is independent of $\mY_{[k]}$ given $\mX_{k}$.

\subsection{Denoising score matching}
\label{appendix:method:sm}

\par In this section, we provide the derivations of the denoising score matching, which has been proven by~\cite{Vincent2011Connection,Ho2020Denoising,Song2019Generative}.

\par Let $\underline{\mX}_{k+1}$ be a random variable drawn from the prediction distribution $\what{q}_{k+1}(\cdot|\vy_{[k]})$, and let $\vepsilon$ be a standard Gaussian noise independent of $\underline{\mX}_{k+1}$. For each fixed noise level $\sigma>0$, define 
\begin{equation}\label{eq:appendix:dsm:1}
\underline{\mX}_{k+1}^{\sigma}=\underline{\mX}_{k+1}+\sigma\vepsilon.
\end{equation}
It is evident that $\underline{\mX}_{k+1}^{\sigma}$ obeys the Gaussian smoothed prediction distribution, that is,
\begin{align*}
q_{k+1}^{\sigma}(\vx^{\sigma}|\vy_{[k]})
&:=\int p_{\underline{\mX}_{k+1}^{\sigma}|\underline{\mX}_{k+1}}(\vx^{\sigma}|\vx)\what{q}_{k+1}(\vx|\vy_{[k]})\d\vx \\
&=\int\underbrace{\gamma_{d}(\vx^{\sigma};\vx,\sigma^{2}\mI_{d})}_{\text{Gaussian kernel}}\what{q}_{k+1}(\vx|\vy_{[k]})\d\vx.
\end{align*}
Observe that the score of the Gaussian kernel is given as 
\begin{equation}\label{eq:appendix:dsm:2}
\nabla_{\vx^{\sigma}}\log p_{\underline{\mX}_{k+1}^{\sigma}|\underline{\mX}_{k+1}}(\vx^{\sigma}|\vx)=-\frac{\vx^{\sigma}-\vx}{\sigma^{2}}.
\end{equation}

\par\textbf{\bfseries Step (I).} We first show that for each function $\vs$,
\begin{equation}\label{eq:appendix:dsm:3}
\begin{aligned}
&\bbE_{\underline{\mX}_{k+1}^{\sigma}\sim p_{\underline{\mX}_{k+1}^{\sigma}|\mY_{[k]}}(\cdot|\vy_{[k]})}\big[\|\vs(\underline{\mX}_{k+1}^{\sigma},\vy_{[k]})-\nabla_{\vx^{\sigma}}\log p_{\underline{\mX}_{k+1}^{\sigma}|\mY_{[k]}}(\underline{\mX}_{k+1}^{\sigma}|\vy_{[k]})\|_{2}^{2}\big] \\
&=\bbE_{\underline{\mX}_{k+1}\sim\what{q}_{k+1}(\cdot|\vy_{[k]})}\bbE_{\underline{\mX}_{k+1}^{\sigma}\sim p_{\underline{\mX}_{k+1}^{\sigma}|\underline{\mX}_{k+1}}}\big[\|\vs(\underline{\mX}_{k+1}^{\sigma},\vy_{[k]}) \\
&\quad-\nabla_{\vx_{\sigma}}\log p_{\underline{\mX}_{k+1}^{\sigma}|\underline{\mX}_{k+1}}(\underline{\mX}_{k+1}^{\sigma}|\underline{\mX}_{k+1})\|_{2}^{2}\big]+c,
\end{aligned} 
\end{equation}
where $c$ is a constant independent of $\vs$. Indeed, 
\begin{equation}\label{eq:appendix:dsm:4}
\begin{aligned}
&\bbE_{\underline{\mX}_{k+1}^{\sigma}\sim p_{\underline{\mX}_{k+1}^{\sigma}|\mY_{[k]}}(\cdot|\vy_{[k]})}\big[\|\vs(\underline{\mX}_{k+1}^{\sigma},\vy_{[k]})-\nabla_{\vx_{\sigma}}\log p_{\underline{\mX}_{k+1}^{\sigma}|\mY_{[k]}}(\underline{\mX}_{k+1}^{\sigma}|\vy_{[k]})\|_{2}^{2}\big] \\
&=\underbrace{\bbE_{\underline{\mX}_{k+1}^{\sigma}\sim p_{\underline{\mX}_{k+1}^{\sigma}|\mY_{[k]}}(\cdot|\vy_{[k]})}\big[\|\vs(\underline{\mX}_{k+1}^{\sigma},\vy_{[k]})\|_{2}^{2}\big]}_{\text{(i)}} \\
&\quad-2\underbrace{\bbE_{\underline{\mX}_{k+1}^{\sigma}\sim p_{\underline{\mX}_{k+1}^{\sigma}|\mY_{[k]}}(\cdot|\vy_{[k]})}\big[\vs(\underline{\mX}_{k+1}^{\sigma},\vy_{[k]})\cdot\nabla_{\vx_{\sigma}}\log p_{\underline{\mX}_{k+1}^{\sigma}|\mY_{[k]}}(\underline{\mX}_{k+1}^{\sigma}|\vy_{[k]})\big]}_{\text{(i)}}+c_{1},
\end{aligned} 
\end{equation}
where $c_{1}$ is a constant independent of $\vs$. For the term (i) in~\eqref{eq:appendix:dsm:4}, we have 
\begin{align}
&\bbE_{\underline{\mX}_{k+1}^{\sigma}\sim p_{\underline{\mX}_{k+1}^{\sigma}|\mY_{[k]}}(\cdot|\vy_{[k]})}\big[\|\vs(\underline{\mX}_{k+1}^{\sigma},\vy_{[k]})\|_{2}^{2}\big] \nonumber \\
&=\int\|\vs(\vx_{\sigma},\vy_{[k]})\|_{2}^{2}p_{\underline{\mX}_{k+1}^{\sigma}|\mY_{[k]}}(\vx_{\sigma}|\vy_{[k]})\d\vx_{\sigma} \nonumber \\
&=\int\|\vs(\vx_{\sigma},\vy_{[k]})\|_{2}^{2}\Big(\int p_{\underline{\mX}_{k+1}^{\sigma}|\underline{\mX}_{k+1},\mY_{[k]}}(\vx_{\sigma}|\vx,\vy_{[k]})q_{k+1}(\vx|\vy_{[k]})\d\vx\Big)\d\vx_{\sigma} \nonumber \\
&=\int\|\vs(\vx_{\sigma},\vy_{[k]})\|_{2}^{2}\Big(\int p_{\underline{\mX}_{k+1}^{\sigma}|\underline{\mX}_{k+1}}(\vx_{\sigma}|\vx)q_{k+1}(\vx|\vy_{[k]})\d\vx\Big)\d\vx_{\sigma} \nonumber \\
&=\int\Big(\int\|\vs(\vx_{\sigma},\vy_{[k]})\|_{2}^{2}p_{\underline{\mX}_{k+1}^{\sigma}|\underline{\mX}_{k+1}}(\vx_{\sigma}|\vx)\d\vx_{\sigma}\Big)q_{k+1}(\vx|\vy_{[k]})\d\vx \nonumber \\
&=\bbE_{\underline{\mX}_{k+1}\sim\what{q}_{k+1}(\cdot|\vy_{[k]})}\bbE_{\underline{\mX}_{k+1}^{\sigma}\sim p_{\underline{\mX}_{k+1}^{\sigma}|\underline{\mX}_{k+1}}}\big[\|\vs(\underline{\mX}_{k+1}^{\sigma},\vy_{[k]})\|_{2}^{2}\big], \label{eq:appendix:dsm:5}
\end{align}
where the second equality follows from Chapman-Kolmogorov identity, the third equality is due to the fact that $\underline{\mX}_{k+1}^{\sigma}$ is conditionally independent of $\mY_{[k]}$ given $\underline{\mX}_{k+1}$. For the term (ii) in~\eqref{eq:appendix:dsm:4}, by the same argument, 
\begin{align}
&\bbE_{\underline{\mX}_{k+1}^{\sigma}\sim p_{\underline{\mX}_{k+1}^{\sigma}|\mY_{[k]}}(\cdot|\vy_{[k]})}\big[\vs(\underline{\mX}_{k+1}^{\sigma},\vy_{[k]})\cdot\nabla_{\vx_{\sigma}}\log p_{\underline{\mX}_{k+1}^{\sigma}|\mY_{[k]}}(\underline{\mX}_{k+1}^{\sigma}|\vy_{[k]})\big] \nonumber \\
&=\int\vs(\vx_{\sigma},\vy_{[k]})\cdot\nabla_{\vx_{\sigma}}p_{\underline{\mX}_{k+1}^{\sigma}|\mY_{[k]}}(\vx_{\sigma}|\vy_{[k]})\d\vx_{\sigma} \nonumber \\
&=\int\vs(\vx_{\sigma},\vy_{[k]})\cdot\nabla_{\vx_{\sigma}}\Big(\int p_{\underline{\mX}_{k+1}^{\sigma}|\underline{\mX}_{k+1}}(\vx_{\sigma}|\vx)q_{k+1}(\vx|\vy_{[k]})\d\vx\Big)\d\vx_{\sigma} \nonumber \\
&=\int\vs(\vx_{\sigma},\vy_{[k]})\cdot\Big(\int\nabla_{\vx_{\sigma}}p_{\underline{\mX}_{k+1}^{\sigma}|\underline{\mX}_{k+1}}(\vx_{\sigma}|\vx)q_{k+1}(\vx|\vy_{[k]})\d\vx\Big)\d\vx_{\sigma} \nonumber \\
&=\int\Big(\int\vs(\vx_{\sigma},\vy_{[k]})\cdot\nabla_{\vx_{\sigma}}p_{\underline{\mX}_{k+1}^{\sigma}|\underline{\mX}_{k+1}}(\vx_{\sigma}|\vx)\d\vx_{\sigma}\Big)q_{k+1}(\vx|\vy_{[k]})\d\vx \nonumber \\
&=\bbE_{\underline{\mX}_{k+1}\sim\what{q}_{k+1}(\cdot|\vy_{[k]})}\bbE_{\underline{\mX}_{k+1}^{\sigma}\sim p_{\underline{\mX}_{k+1}^{\sigma}|\underline{\mX}_{k+1}}}\big[\vs(\underline{\mX}_{k+1}^{\sigma},\vy_{[k]})\cdot\nabla\log p_{\underline{\mX}_{k+1}^{\sigma}|\underline{\mX}_{k+1}}(\underline{\mX}_{k+1}^{\sigma}|\underline{\mX}_{k+1})\big]. \label{eq:appendix:dsm:6}
\end{align}
Plugging~\eqref{eq:appendix:dsm:5} and~\eqref{eq:appendix:dsm:6} into~\eqref{eq:appendix:dsm:4} completes the proof of~\eqref{eq:appendix:dsm:3}. 

\par\textbf{\bfseries Step (II).} We next reformulate the right-hand side of~\eqref{eq:appendix:dsm:3} as 
\begin{align*}
&\bbE_{\underline{\mX}_{k+1}\sim\what{q}_{k+1}(\cdot|\vy_{[k]})}\bbE_{\underline{\mX}_{k+1}^{\sigma}\sim p_{\underline{\mX}_{k+1}^{\sigma}|\underline{\mX}_{k+1}}}\big[\|\vs(\underline{\mX}_{k+1}^{\sigma},\vy_{[k]})-\nabla_{\vx}\log p_{\underline{\mX}_{k+1}^{\sigma}|\underline{\mX}_{k+1}}(\underline{\mX}_{k+1}^{\sigma}|\underline{\mX}_{k+1})\|_{2}^{2}\big] \\
&=\bbE_{\underline{\mX}_{k+1}\sim\what{q}_{k+1}(\cdot|\vy_{[k]})}\bbE_{\underline{\mX}_{k+1}^{\sigma}\sim p_{\underline{\mX}_{k+1}^{\sigma}|\underline{\mX}_{k+1}}}\big[\big\|\vs(\underline{\mX}_{k+1}^{\sigma},\vy_{[k]})+\frac{1}{\sigma^{2}}(\underline{\mX}_{k+1}^{\sigma}-\underline{\mX}_{k+1})\big\|_{2}^{2}\big] \\
&=\frac{1}{\sigma^{2}}\bbE_{\underline{\mX}_{k+1}\sim\what{q}_{k+1}(\cdot|\vy_{[k]})}\bbE_{\vepsilon\sim\mathcal{N}(\mathbf{0},\mI_{d})}\big[\|\sigma\vs(\underline{\mX}_{k+1}+\sigma\vepsilon,\vy_{[k]})+\vepsilon\|_{2}^{2}\big],
\end{align*}
where the first equality is due to~\eqref{eq:appendix:dsm:2}, and the second equality used~\eqref{eq:appendix:dsm:1}. Combining this with~\eqref{eq:appendix:dsm:3} yields 
\begin{equation*}
\begin{aligned}
&\bbE_{\underline{\mX}_{k+1}^{\sigma}\sim p_{\underline{\mX}_{k+1}^{\sigma}|\mY_{[k]}}(\cdot|\vy_{[k]})}\Big[\|\vs(\underline{\mX}_{k+1}^{\sigma},\vy_{[k]})-\nabla\log p_{\underline{\mX}_{k+1}^{\sigma}|\mY_{[k]}}(\underline{\mX}_{k+1}^{\sigma}|\vy_{[k]})\|_{2}^{2}\Big] \\
&=\frac{1}{\sigma^{2}}\bbE_{\underline{\mX}_{k+1}\sim\what{q}_{k+1}(\cdot|\vy_{[k]})}\bbE_{\vepsilon\sim\mathcal{N}(\mathbf{0},\mI_{d})}\Big[\|\sigma\vs(\underline{\mX}_{k+1}+\sigma\vepsilon,\vy_{[k]})+\vepsilon\|_{2}^{2}\Big]+c,
\end{aligned} 
\end{equation*}
which achieves the population risk of the denoising score matching. Consequently, 
\begin{equation*}
\nabla_{\vx_{\sigma}}\log q_{k+1}^{\sigma}(\cdot|\vy_{[k+1]})=\argmin_{\vs}\bbE_{\underline{\mX}_{k+1}\sim\what{q}_{k+1}(\cdot|\vy_{[k]})}\bbE_{\vepsilon\sim\mathcal{N}(\mathbf{0},\mI_{d})}\big[\|\sigma\vs(\underline{\mX}_{k+1}+\sigma\vepsilon,\vy_{[k]})+\vepsilon\|_{2}^{2}\big].
\end{equation*}
Finally, approximate the above population risk by its empirical counterpart yields~\eqref{eq:dsm}.

\section{Proofs in Section~\ref{section:convergence}}
\label{appendix:convergence}

\par In this section, we provide proofs of results in Section~\ref{section:convergence}.

\begin{proof}[Proof of Proposition~\ref{proposition:lower:bound}]
For any $t\in(0,1)$, we have 
\begin{align*}
\frac{\d}{\dt}\log q_{k+1}(\vx_{*}+t(\vx-\vx_{*})|\vy_{[k]}) 
&= \langle \nabla\log q_{k+1}(\vx_{*}+t(\vx-\vx_{*})|\vy_{[k]}),\vx-\vx_{*} \rangle \\
&\geq -\|\nabla\log q_{k+1}(\vx_{*}+t(\vx-\vx_{*})|\vy_{[k]})\|_{2}\|\vx-\vx_{*}\|_{2} \\
&\geq -B(1+\|\vx_{*}+t(\vx-\vx_{*})\|_{2})\|\vx-\vx_{*}\|_{2} \\
&\geq -B(1+\|\vx_{*}\|_{2}+t\|\vx-\vx_{*}\|_{2})\|\vx-\vx_{*}\|_{2} \\
&= -B(1+\|\vx_{*}\|_{2})\|\vx-\vx_{*}\|_{2}-tB\|\vx-\vx_{*}\|_{2}^{2} \\
&\geq -B(1+\|\vx_{*}\|_{2}^{2})-\frac{B}{2}\|\vx-\vx_{*}\|_{2}^{2}-tB\|\vx-\vx_{*}\|_{2}^{2},
\end{align*}
where the first inequality is due to Cauchy-Schwarz inequality, the second inequality is owing to Assumption~\ref{assumption:bounded} (ii), and third inequality holds from the triangular inequality, and the last inequality invokes AM-GM inequality $ab\leq\frac{a^{2}}{2}+\frac{b^{2}}{2}$ and $(a+b)^{2}\leq 2a^{2}+2b^{2}$. Taking integral on both sides of the inequality yields
\begin{align*}
\log q_{k+1}(\vx|\vy_{[k]})-\log q_{k+1}(\vx_{*}|\vy_{[k]})
&= \int_{0}^{1}\frac{\d}{\dt}\log q_{k+1}(\vx_{*}+t(\vx-\vx_{*})|\vy_{[k]})\dt \\
&\geq -B(1+\|\vx_{*}\|_{2}^{2})-B\|\vx-\vx_{*}\|_{2}^{2} \\
&\geq -B(1+3\|\vx_{*}\|_{2}^{2})-2B\|\vx\|_{2}^{2},
\end{align*}
where the last inequality used the triangular inequality. Therefore, 
\begin{equation*}
\log q_{k+1}(\vx|\vy_{[k]})\geq\log q_{k+1}(\vx_{*}|\vy_{[k]})-B(1+3\|\vx_{*}\|_{2}^{2})-2B\|\vx\|_{2}^{2},
\end{equation*}
which implies 
\begin{equation*}
q_{k+1}(\vx|\vy_{[k]})\geq \frac{q_{k+1}(\vx_{*}|\vy_{[k]})}{\exp(B(1+3\|\vx_{*}\|_{2}^{2}))}\exp(-2B\|\vx\|_{2}^{2}).
\end{equation*}
This completes the proof.
\end{proof}

\begin{proof}[Proof of Theorem~\ref{theorem:section:convergence}]
Substituting Lemmas~\ref{lemma:appendix:discretization},~\ref{appendix:prior:3}, and~\ref{proposition:sm:error} into Lemma~\ref{lemma:appendix:convergence:decomposition}
completes the proof.
\end{proof}

\begin{proof}[Proof of Corollary~\ref{corollary:section:convergence}]
A direct conclusion of Theorem~\ref{theorem:section:convergence}. 
\end{proof}

\begin{proof}[Proof of Theorem~\ref{theorem:convergence:init}]
A direct conclusion of Lemma~\ref{lemma:appendix:init}.
\end{proof}

\begin{proof}[Proof of Theorem~\ref{theorem:section:convergence:assimilation}]
A direct conclusion of Corollary~\ref{corollary:section:convergence} and Theorem~\ref{theorem:convergence:init}.
\end{proof}

\begin{proof}[Proof of Corollary~\ref{corollary:section:convergence:wasserstein}]
Combining Theorem~\ref{theorem:section:convergence:assimilation} and Lemma~\ref{lemma:w1:tv} completes the proof.
\end{proof}

\section{Error Decomposition of Posterior Sampling}
\label{appendix:convergence:decomposition}

\par Recall the Langevin diffusion for the ($k+1$)-th time step of the data assimilation 
\begin{equation}\label{eq:LD}
\d\mZ_{t}=\nabla_{\vx}\log\pi_{k+1}(\mZ_{t}|\vy_{[k+1]})\dt+\sqrt{2}\d\mB_{t}, \quad \mZ_{0}\sim\pi_{k+1}^{0}(\cdot|\vy_{[k+1]}),~t\geq 0.
\end{equation}
Denote by $\pi_{k+1}^{t}$ the law of $\mZ_{t}$ for each $t\geq 0$. The Langevin Monte Carlo is defined as the Euler-Maruyama discretization of the Langevin diffusion. The interpolation of the Langevin Monte Carlo is given as, for each $0\leq\ell\leq K-1$,
\begin{equation}\label{eq:LMC}
\d\bar{\mZ}_{t}=\nabla_{\vx}\log\pi_{k+1}(\bar{\mZ}_{\ell h}|\vy_{[k+1]})\dt+\sqrt{2}\d\mB_{t}, \quad \ell h\leq t\leq(\ell+1)h,
\end{equation}
where $\bar{\mZ}_{0}\sim\pi_{k+1}^{0}(\cdot|\vy_{[k+1]})$. Denote by $\bar{\pi}_{k+1}^{t}$ the law of $\bar{\mZ}_{t}$ for each $0\leq t\leq Kh=T$. We next introduce the interpolation of the score-based Langevin Monte Carlo
\begin{equation}\label{eq:SLMC}
\d\what{\mZ}_{t}=\what{\vb}_{k+1}(\what{\mZ}_{\ell h}|\vy_{[k+1]})\dt+\sqrt{2}\d\mB_{t}, \quad \ell h\leq t\leq(\ell+1)h,
\end{equation}
where $\what{\mZ}_{0}\sim\pi_{k+1}^{0}(\cdot|\vy_{[k+1]})$, and the estimator of posterior score function is given as 
\begin{equation*}
\what{\vb}_{k+1}(\vx|\vy_{[k+1]})=\nabla_{\vx}\log g_{k+1}(\vy_{k+1}|\vx)+\what{\vs}_{k+1}(\vx,\vy_{[k]}).
\end{equation*}
Here the prediction score $\what{\vs}_{k+1}$ is defined as~\eqref{eq:dsm}. Denote by $\what{\pi}_{k+1}^{t}$ the law of $\what{\mZ}_{t}$ for each $0\leq t\leq Kh=T$. Recall the prediction distribution~\eqref{eq:prediction:distribution}
\begin{equation*}
q_{k+1}(\vx|\vy_{[k]})=\int\rho_{k}(\vx|\vx_{k})\pi_{k}(\vx_{k}|\vy_{[k]})\d\vx_{k},
\end{equation*}
which serves as the prior in the recursive Bayesian framework. Recall the approximated prediction distribution~\eqref{eq:prediction:distribution:hat}
\begin{equation*}
\what{q}_{k+1}(\vx|\vy_{[k]}):=\int\rho_{k}(\vx|\vx_{k})\what{\pi}_{k}^{T}(\vx_{k}|\vy_{[k]})\d\vx_{k}.
\end{equation*} 

\par The following lemma decomposes the TV distance between $\pi_{k+1}$ and $\what{\pi}_{k+1}^{T}$ into three parts: the convergence of the Langevin Monte Carlo, the prior error, and the error of score matching.

\begin{lemma}[Error decomposition]
\label{lemma:appendix:convergence:decomposition}
Suppose Assumptions~\ref{assumption:posterior:smooth} and~\ref{assumption:LSI:posterior} hold. Let $\pi_{k+1}$ be the stationary distribution of the Langevin diffusion~\eqref{eq:LD}, and let $\what{\pi}_{k+1}^{T}$ be the law of the score-based Langevin Monte Carlo~\eqref{eq:SLMC}. Assume the step size $h>0$ satisfies $400dC_{\LSI}\lambda^{2}h\leq 1$. Then for each $k\in\bbN$,
\begin{align*}
&\|\pi_{k+1}(\cdot|\vy_{[k+1]})-\what{\pi}_{k+1}^{T}(\cdot|\vy_{[k+1]})\|_{\tv}^{2} \\
&\lesssim \underbrace{\|\pi_{k+1}(\cdot|\vy_{[k+1]})-\bar{\pi}_{k+1}^{T}(\cdot|\vy_{[k+1]})\|_{\tv}^{2}}_{\text{convergence of Langevin Monte Carlo}} \\
&\quad +\underbrace{(C_{\LSI}\eta_{\chi}+T)\bbE^{\frac{1}{2}}\left[\|\nabla_{\vx}\log q_{k+1}(\mX_{k+1}|\vy_{[k]})-\nabla\log\what{q}_{k+1}(\mX_{k+1}|\vy_{[k]})\|_{2}^{4}\right]}_{\text{prior error}} \\
&\quad+\underbrace{(C_{\LSI}\eta_{\chi}+T)\bbE^{\frac{1}{2}}\left[\|\nabla_{\vx}\log\widehat{q}_{k+1}(\mX_{k+1}|\vy_{[k]})-\what{\vs}_{k+1}(\mX_{k+1},\vy_{[k]})\|_{2}^{4}\right]}_{\text{score estimation error}},
\end{align*}
where the expectation $\bbE[\cdot]$ is taken with respect to $\mX_{k+1}\sim\pi_{k+1}(\cdot|\vy_{[k+1]})$.
\end{lemma}

\begin{remark}
The convergence of the Langevin Monte Carlo will be analyzed in Appendix~\ref{appendix:LMC}. The prior error characterizes the error of the approximated prediction distribution $\what{q}_{k+1}$, which is induced by the error of the posterior distribution $\what{\pi}_{k}^{T}$ in the previous time step. The detailed analysis will be shown in Appendix~\ref{appendix:prior}. Finally, we will investigate the score estimation error in Appendix~\ref{appendix:score}.
\end{remark}

\begin{proof}[Proof of Lemma~\ref{lemma:appendix:convergence:decomposition}]
According to the triangular inequality of TV distance, we have 
\begin{align*}
&\|\pi_{k+1}(\cdot|\vy_{[k+1]})-\what{\pi}_{k+1}^{T}(\cdot|\vy_{[k+1]})\|_{\tv}^{2} \\
&\leq 2\|\pi_{k+1}(\cdot|\vy_{[k+1]})-\bar{\pi}_{k+1}^{T}(\cdot|\vy_{[k+1]})\|_{\tv}^{2}+2\|\bar{\pi}_{k+1}^{T}(\cdot|\vy_{[k+1]})-\what{\pi}_{k+1}^{T}(\cdot|\vy_{[k+1]})\|_{\tv}^{2}. 
\end{align*}
For the second term, we invoke Girsanov theorem~\cite{chen2023sampling} to show that 
\begin{align*}
&\|\bar{\pi}_{k+1}^{T}(\cdot|\vy_{[k+1]})-\what{\pi}_{k+1}^{T}(\cdot|\vy_{[k+1]})\|_{\tv}^{2} \\
&\leq 2\kl\big(\bar{\pi}_{k+1}^{T}(\cdot|\vy_{[k+1]}),\what{\pi}_{k+1}^{T}(\cdot|\vy_{[k+1]})\big) \\
&\leq\sum_{\ell=0}^{K-1}h\bbE_{\bar{\mZ}_{\ell h}}\big[\|\nabla_{\vx}\log\pi_{k+1}(\bar{\mZ}_{\ell h}|\vy_{[k+1]})-\what{\vb}_{k+1}(\bar{\mZ}_{\ell h},\vy_{[k+1]})\|_{2}^{2}\big] \\
&=\sum_{\ell=0}^{K-1}h\bbE_{\bar{\mZ}_{\ell h}}\big[\|\nabla_{\vx}\log q_{k+1}(\bar{\mZ}_{\ell h}|\vy_{[k]})-\what{\vs}_{k+1}(\bar{\mZ}_{\ell h},\vy_{[k]})\|_{2}^{2}\big],
\end{align*}
where the first inequality follows from Pinsker's inequality (Lemma~\ref{lemma:tv:kl}), and the second inequality invokes Girsanov theorem~\cite{chen2023sampling}. For each summand, we have 
\begin{align*}
&\bbE_{\bar{\mZ}_{\ell h}}\big[\|\nabla_{\vx}\log q_{k+1}(\bar{\mZ}_{\ell h}|\vy_{[k]})-\what{\vs}_{k+1}(\bar{\mZ}_{\ell h},\vy_{[k]})\|_{2}^{2}\big] \\
&=\int \|\nabla_{\vx}\log q_{k+1}(\vz|\vy_{[k]})-\what{\vs}_{k+1}(\vz,\vy_{[k]})\|_{2}^{2}\frac{\bar{\pi}_{k+1}^{\ell h}(\vz|\vy_{[k+1]})}{\pi_{k+1}(\vz|\vy_{[k+1]})}\pi_{k+1}(\vz|\vy_{[k+1]})\d\vz \\
&\leq \left(\int \|\nabla_{\vx}\log q_{k+1}(\vz|\vy_{[k]})-\what{\vs}_{k+1}(\vz,\vy_{[k]})\|_{2}^{4}\pi_{k+1}(\vz|\vy_{[k+1]})\d\vz\right)^{\frac{1}{2}} \\
&\quad \times\left(\int\Big(\frac{\bar{\pi}_{k+1}^{\ell h}(\vz|\vy_{[k+1]})}{\pi_{k+1}(\vz|\vy_{[k+1]})}\Big)^{2}\pi_{k+1}(\vz|\vy_{[k+1]})\d\vz\right)^{\frac{1}{2}} \\
&= \bbE^{\frac{1}{2}}\left[\|\nabla_{\vx}\log q_{k+1}(\mX_{k+1}|\vy_{[k]})-\what{\vs}_{k+1}(\mX_{k+1},\vy_{[k]})\|_{2}^{4}\right] \\
&\quad \times\left(\chi^{2}\big(\bar{\pi}_{k+1}^{\ell h}(\cdot|\vy_{[k+1]})\|\pi_{k+1}(\cdot|\vy_{[k+1]})\big)+1\right)^{\frac{1}{2}},
\end{align*}
where the inequality holds from Cauchy-Schwarz inequality. Further, using Lemma~\ref{lemma:appendix:discretization} implies that for step size $h>0$ satisfies $400dC_{\LSI}\lambda^{2}h\leq 1$,
\begin{align*}
&\sum_{\ell=0}^{K-1}h\bbE_{\bar{\mZ}_{\ell h}}\big[\|\nabla_{\vx}\log q_{k+1}(\bar{\mZ}_{\ell h}|\vy_{[k]})-\what{\vs}_{k+1}(\bar{\mZ}_{\ell h},\vy_{[k]})\|_{2}^{2}\big] \\
&\leq h\bbE^{\frac{1}{2}}\left[\|\nabla_{\vx}\log q_{k+1}(\mX_{k+1}|\vy_{[k]})-\what{\vs}_{k+1}(\mX_{k+1},\vy_{[k]})\|_{2}^{4}\right] \\
&\quad \times \sum_{\ell=0}^{K-1}\Big(\exp\Big(-\frac{\ell h}{5C_{\LSI}}\Big)\chi^{2}\big(\pi_{k+1}^{0}(\cdot|\vy_{[k+1]})\|\pi_{k+1}(\cdot|\vy_{[k+1]})\big)+2\Big)^{\frac{1}{2}} \\
&\leq h\bbE^{\frac{1}{2}}\left[\|\nabla_{\vx}\log q_{k+1}(\mX_{k+1}|\vy_{[k]})-\what{\vs}_{k+1}(\mX_{k+1},\vy_{[k]})\|_{2}^{4}\right] \\
&\quad \times\left(\sum_{\ell=0}^{K-1}\exp\Big(-\frac{\ell h}{10C_{\LSI}}\Big)\chi^{2}\big(\pi_{k+1}^{0}(\cdot|\vy_{[k+1]})\|\pi_{k+1}(\cdot|\vy_{[k+1]})\big)^{\frac{1}{2}}+2K\right) \\
&\leq \bbE^{\frac{1}{2}}\left[\|\nabla_{\vx}\log q_{k+1}(\mX_{k+1}|\vy_{[k]})-\what{\vs}_{k+1}(\mX_{k+1},\vy_{[k]})\|_{2}^{4}\right] \\
&\quad \times\left(\frac{20}{3}C_{\LSI}\chi^{2}\big(\pi_{k+1}^{0}(\cdot|\vy_{[k+1]})\|\pi_{k+1}(\cdot|\vy_{[k+1]})\big)^{\frac{1}{2}}+2T\right),
\end{align*}
where the last inequality is due to 
\begin{equation*}
\sum_{\ell=0}^{K-1}\exp\Big(-\frac{\ell h}{10C_{\LSI}}\Big)\leq\frac{20C_{\LSI}}{3h}.
\end{equation*}
Therefore, it follows from the triangular inequality that 
\begin{align*}
&\|\bar{\pi}_{k+1}^{T}(\cdot|\vy_{[k+1]})-\what{\pi}_{k+1}^{T}(\cdot|\vy_{[k+1]})\|_{\tv}^{2} \\
&\leq \left(\frac{20}{3}C_{\LSI}\eta_{\chi}+2T\right)\bbE^{\frac{1}{2}}\left[\|\nabla_{\vx}\log q_{k+1}(\mX_{k+1}|\vy_{[k]})-\what{\vs}_{k+1}(\mX_{k+1},\vy_{[k]})\|_{2}^{4}\right] \\
&\leq 28(C_{\LSI}\eta_{\chi}+T)\bbE^{\frac{1}{2}}\left[\|\nabla_{\vx}\log q_{k+1}(\mX_{k+1}|\vy_{[k]})-\nabla\log\what{q}_{k+1}(\mX_{k+1}|\vy_{[k]})\|_{2}^{4}\right] \\
&\quad +28(C_{\LSI}\eta_{\chi}+T)\bbE^{\frac{1}{2}}\left[\|\nabla_{\vx}\log\widehat{q}_{k+1}(\mX_{k+1}|\vy_{[k]})-\what{\vs}_{k+1}(\mX_{k+1},\vy_{[k]})\|_{2}^{4}\right],
\end{align*}
which completes the proof.
\end{proof}

\section{Convergence of Langevin Monte Carlo}
\label{appendix:LMC}

\par In this section, we aim to analyze the convergence of the Langevin Monte Carlo
\begin{equation*}
\|\pi_{k+1}(\cdot|\vy_{[k+1]})-\bar{\pi}_{k+1}^{T}(\cdot|\vy_{[k+1]})\|_{\tv}^{2}.
\end{equation*}
Indeed, we provide a stronger convergence result in $\chi^{2}$-divergence rather than the TV distance.

\begin{lemma}\label{lemma:appendix:discretization}
Suppose Assumptions~\ref{assumption:posterior:smooth} and~\ref{assumption:LSI:posterior} hold. Then
\begin{equation*}
\chi^{2}\big(\bar{\pi}_{k+1}^{T}(\cdot|\vy_{[k+1]})\|\pi_{k+1}(\cdot|\vy_{[k+1]})\big)\leq\exp\Big(-\frac{T}{5C_{\LSI}}\Big)\eta_{\chi}^{2}+140dC_{\LSI}\lambda^{2}h,
\end{equation*}
where $T=Kh$, and the step size $h>0$ satisfies $400dC_{\LSI}\lambda^{2}h\leq 1$.
\end{lemma}
\begin{remark}
The proof of Lemma~\ref{lemma:appendix:discretization} is inspired by~\cite[Theorem 2.1]{Lee2022Convergence}, and~\cite[Theorem 4]{Chewi2024Analysis}. We show the proof in this section for the sake of completeness. The first term in Lemma~\ref{lemma:appendix:discretization} converges to zero exponentially as the time $T$ increases, which corresponds to the convergence of the Langevin diffusion~\eqref{eq:LD}. The second term is linear with respect to the step size $h$, induced by the Euler-Maruyama approximation.
\end{remark}

\par Recall the Langevin Monte Carlo~\eqref{eq:LMC}. For each $0\leq\ell\leq K-1$ and $\ell h\leq t\leq(\ell+1)h$, 
\begin{align}
\bar{\mZ}_{t}
&=\bar{\mZ}_{\ell h}+\int_{\ell h}^{t}\nabla_{\vx}\log\pi_{k+1}(\bar{\mZ}_{\ell h}|\vy_{[k+1]})\ds+\sqrt{2}\int_{\ell h}^{t}\d\mB_{s} \nonumber \\
&=\bar{\mZ}_{\ell h}+(t-\ell h)\nabla_{\vx}\log\pi_{k+1}(\bar{\mZ}_{\ell h}|\vy_{[k+1]})+\sqrt{2}(\mB_{t}-\mB_{\ell h}), \label{eq:appendix:discretization:1}
\end{align}
where $\mB_{t}-\mB_{\ell h}\sim\mathcal{N}(\mathbf{0},(t-\ell h)\mI_{d})$ is independent of $\bar{\mZ}_{\ell h}$.

\subsection{Differential inequality for the chi-squared divergence}

\par The most crucial recipe in the proof of Lemma~\ref{lemma:appendix:discretization} is the following differential inequality for the $\chi^{2}$-divergence, which is inspired by~\cite[Theorem 4]{Chewi2024Analysis} and~\cite[Theorem 4.2]{Lee2022Convergence}. 

\par Before proceeding, we introduce some notations and properties. Define the Radon-Nikodym derivative of $\bar{\pi}_{k+1}^{t}$ with respect to $\pi_{k+1}$
\begin{equation}\label{eq:phi}
\phi_{k+1}^{t}(\vx|\vy_{[k+1]}):=\frac{\bar{\pi}_{k+1}^{t}(\vx|\vy_{[k+1]})}{\pi_{k+1}(\vx|\vy_{[k+1]})}, \quad \vx\in\bbR^{d}.
\end{equation}
Apparently, we find 
\begin{equation}\label{eq:phi:property}
\bbE_{\mZ}[\phi_{k+1}^{t}(\mZ|\vy_{[k+1]})^{2}]
=\bbE_{\bar{\mZ}_{t}}[\phi_{k+1}^{t}(\bar{\mZ}_{t}|\vy_{[k+1]})].
\end{equation}
Further, we define 
\begin{equation}\label{eq:psi}
\psi_{k+1}^{t}(\vx|\vy_{[k+1]}):=\frac{\phi_{k+1}^{t}(\vx|\vy_{[k+1]})}{\bbE_{\mZ}[\phi_{k+1}^{t}(\mZ|\vy_{[k+1]})^{2}]}, \quad \vx\in\bbR^{d}.
\end{equation}
Then it is straightforward that 
\begin{align}
\bbE_{\bar{\mZ}_{t}}\big[\psi_{k+1}^{t}(\bar{\mZ}_{t}|\vy_{[k+1]})\big]
&=\int\psi_{k+1}^{t}(\vx|\vy_{[k+1]})\phi_{k+1}^{t}(\vx|\vy_{[k+1]})\pi_{k+1}(\vx|\vy_{[k+1]})\d\vx \nonumber \\
&=\int\frac{\phi_{k+1}^{t}(\vx|\vy_{[k+1]})^{2}}{\bbE_{\mZ}[\phi_{k+1}^{t}(\mZ|\vy_{[k+1]})^{2}]}\pi_{k+1}(\vx|\vy_{[k+1]})\d\vx=1. \label{eq:psi:property}
\end{align}

\par The following lemmas shows that the derivative of the $\chi^{2}$-divergence can be bounded by two parts: Dirichlet energy and the discretization error.
\begin{lemma}\label{lemma:appendix:LMC:3}
For each $\ell h\leq t\leq(\ell+1)h$, it holds that 
\begin{align*}
&\frac{\d}{\dt}\chi^{2}\big(\bar{\pi}_{k+1}^{t}(\cdot|\vy_{[k+1]})\|\pi_{k+1}(\cdot|\vy_{[k+1]})\big) \\
&\leq-\underbrace{\bbE_{\mZ}\big[\|\nabla_{\vx}\phi_{k+1}^{t}(\mZ|\vy_{[k+1]})\|_{2}^{2}\big]}_{\text{Dirichlet energy}} \\
&\quad+\underbrace{\bbE_{\mZ}\big[\phi_{k+1}^{t}(\mZ|\vy_{[k+1]})^{2}\big]\bbE_{(\bar{\mZ}_{\ell h},\bar{\mZ}_{t})}\big[\|\ve_{k+1}(\bar{\mZ}_{\ell h},\bar{\mZ}_{t}|\vy_{[k+1]})\|_{2}^{2}\psi_{k+1}^{t}(\bar{\mZ}_{t}|\vy_{[k+1]})\big]}_{\text{discretization error}}
\end{align*}
where the pointwise discretization error is defined as 
\begin{equation*}
\ve_{k+1}(\vx_{\ell h},\vx|\vy_{[k+1]}):=\nabla_{\vx}\log\pi_{k+1}(\vx_{\ell h}|\vy_{[k+1]})-\nabla_{\vx}\log\pi_{k+1}(\vx|\vy_{[k+1]}).
\end{equation*}
Here the expectation $\bbE_{\mZ}[\cdot]$ is taken with respect to $\mZ\sim\pi_{k+1}(\cdot|\vy_{[k+1]})$, and the expectation $\bbE_{(\bar{\mZ}_{\ell h},\bar{\mZ}_{t})}[\cdot]$ is taken with respect to $(\bar{\mZ}_{\ell h},\bar{\mZ}_{t})\sim\bar{\pi}_{k+1}^{\ell h,t}(\cdot|\vy_{[k+1]})$.
\end{lemma}

\begin{remark}
According to~\cite[Lemma 6]{Vempala2019Rapid}, the law of the Langevin diffusion~\eqref{eq:LD} $\pi_{k+1}^{t}(\cdot|\vy_{[k+1]})$ satisfies 
\begin{equation*}
\frac{\d}{\dt}\chi^{2}\big(\pi_{k+1}^{t}(\cdot|\vy_{[k+1]})\|\pi_{k+1}(\cdot|\vy_{[k+1]})\big)\leq-2\bbE_{\mZ}\big[\|\nabla_{\vx}\phi_{k+1}^{t}(\mZ|\vy_{[k+1]})\|_{2}^{2}\big],
\end{equation*}
which can also be derived from~\eqref{eq:lemma:appendix:LMC:3:1}. Compared with this inequality, the differential inequality of the Langevin Monte Carlo in Lemma~\ref{lemma:appendix:LMC:3} has an additional term known as the discretization error.
\end{remark}

\par We first introduce the Fokker-Planck equation associated to the Langevin Monte Carlo~\eqref{eq:LMC}, which has appeared in~\cite[Proposition 17]{Chewi2024Analysis}.

\begin{lemma}\label{lemma:appendix:LMC:2}
For each $\ell h\leq t\leq(\ell+1)h$, the law of Langevin Monte Carlo~\eqref{eq:LMC} satisfies the Fokker-Planck equation
\begin{equation*}
\frac{\partial}{\partial t}\bar{\pi}_{k+1}^{t}(\vx|\vy_{[k+1]})=-\nabla_{\vx}\cdot\big(\bar{\pi}_{k+1}^{t}(\vx|\vy_{[k+1]})\bar{\vb}_{k+1}^{t}(\vx|\vy_{[k+1]})\big)+\Delta_{\vx}\bar{\pi}_{k+1}^{t}(\vx|\vy_{[k+1]}),
\end{equation*}
where the drift term is given as 
\begin{equation}\label{eq:lemma:appendix:LMC:2:1}
\bar{\vb}_{k+1}^{t}(\vx|\vy_{[k+1]})=\bbE\big[\nabla_{\vx}\log\pi_{k+1}(\bar{\mZ}_{\ell h}|\vy_{[k+1]})|\bar{\mZ}_{t}=\vx,\mY_{[k+1]}=\vy_{[k+1]}\big].
\end{equation}
\end{lemma}

\begin{proof}[Proof of Lemma~\ref{lemma:appendix:LMC:2}]
Let $\bar{\pi}_{k+1}^{t|\ell h}(\cdot|\vx_{\ell h},\vy_{[k+1]})$ denote the conditional distribution of $\bar{\mZ}_{t}$ given $\bar{\mZ}_{\ell h}=\vx_{\ell h}$ and $\mY_{[k+1]}=\vy_{[k+1]}$, which satisfies the Fokker-Planck equation~\cite[Theorem 2.2]{Pavliotis2014Stochastic}
\begin{align}
&\frac{\partial}{\partial t}\bar{\pi}_{k+1}^{t|\ell h}(\vx|\vx_{\ell h},\vy_{[k+1]}) \nonumber \\
&=\nabla_{\vx}\cdot\Big(-\bar{\pi}_{k+1}^{t|\ell h}(\vx|\vx_{\ell h},\vy_{[k+1]})\nabla_{\vx}\log\pi_{k+1}(\vx_{\ell h}|\vy_{[k+1]})\Big)+\Delta_{\vx}\bar{\pi}_{k+1}^{t|\ell h}(\vx|\vx_{\ell h},\vy_{[k+1]}). \label{eq:lemma:appendix:LMC:2:2}
\end{align}
Multiplying both sides of the equality by $\bar{\pi}_{k+1}^{\ell h}(\vx_{\ell h}|\vy_{[k+1]})$ and then integrating with respect to $\vx_{\ell h}\in\bbR^{d}$ deduces
\begin{align*}
&\frac{\partial}{\partial t}\bar{\pi}_{k+1}^{t}(\vx|\vy_{[k+1]}) \\
&=\frac{\partial}{\partial t}\int\bar{\pi}_{k+1}^{t|\ell h}(\vx|\vx_{\ell h},\vy_{[k+1]})\bar{\pi}_{k+1}^{\ell h}(\vx_{\ell h}|\vy_{[k+1]})\d\vx_{\ell h} \\
&=\int\frac{\partial}{\partial t}\bar{\pi}_{k+1}^{t|\ell h}(\vx|\vx_{\ell h},\vy_{[k+1]})\bar{\pi}_{k+1}^{\ell h}(\vx_{\ell h}|\vy_{[k+1]})\d\vx_{\ell h} \\
&=-\int\nabla_{\vx}\cdot\Big(\bar{\pi}_{k+1}^{t|\ell h}(\vx|\vx_{\ell h},\vy_{[k+1]})\nabla_{\vx}\log\pi_{k+1}(\vx_{\ell h}|\vy_{[k+1]})\Big)\bar{\pi}_{k+1}^{\ell h}(\vx_{\ell h}|\vy_{[k+1]})\d\vx_{\ell h} \\
&\quad+\int\Delta_{\vx}\bar{\pi}_{k+1}^{t|\ell h}(\vx|\vx_{\ell h},\vy_{[k+1]})\bar{\pi}_{k+1}^{\ell h}(\vx_{\ell h}|\vy_{[k+1]})\d\vx_{\ell h} \\
&=-\int\nabla_{\vx}\cdot\Big(\bar{\pi}_{k+1}^{t,\ell h}(\vx,\vx_{\ell h}|\vy_{[k+1]})\nabla_{\vx}\log\pi_{k+1}(\vx_{\ell h}|\vy_{[k+1]})\Big)\d\vx_{\ell h} \\
&\quad+\int\Delta_{\vx}\bar{\pi}_{k+1}^{t,\ell h}(\vx,\vx_{\ell h}|\vy_{[k+1]})\d\vx_{\ell h} \\
&=-\nabla_{\vx}\cdot\Big(\bar{\pi}_{k+1}^{t}(\vx|\vy_{[k+1]})\int\bar{\pi}_{k+1}^{\ell h|t}(\vx_{\ell h}|\vx,\vy_{[k+1]})\nabla_{\vx}\log\pi_{k+1}(\vx_{\ell h}|\vy_{[k+1]})\d\vx_{\ell h}\Big) \\
&\quad+\Delta_{\vx}\bar{\pi}_{k+1}^{t}(\vx|\vy_{[k+1]})\int\bar{\pi}_{k+1}^{\ell h|t}(\vx_{\ell h}|\vx,\vy_{[k+1]})\d\vx_{\ell h} \\
&=-\nabla_{\vx}\cdot\Big(\bar{\pi}_{k+1}^{t}(\vx|\vy_{[k+1]})\bar{\vb}_{k+1}^{t}(\vx|\vy_{[k+1]})\Big)+\Delta_{\vx}\bar{\pi}_{k+1}^{t}(\vx|\vy_{[k+1]}),
\end{align*}
where the first equality holds from Chapman-Kolmogorov identity, the third equality follows from~\eqref{eq:lemma:appendix:LMC:2:2}, and the last equality invokes~\eqref{eq:lemma:appendix:LMC:2:1}. This completes the proof.
\end{proof}

\par Now we are ready to prove Lemma~\ref{lemma:appendix:LMC:3}.
\begin{proof}[Proof of Lemma~\ref{lemma:appendix:LMC:3}]
According to the definition of the $\chi^{2}$-divergence, we have 
\begin{align}
&\frac{\d}{\dt}\chi^{2}\big(\bar{\pi}_{k+1}^{t}(\cdot|\vy_{[k+1]})\|\pi_{k+1}(\cdot|\vy_{[k+1]})\big) \nonumber \\
&=2\int\frac{\partial\bar{\pi}_{k+1}^{t}}{\partial t}(\vx|\vy_{[k+1]})\frac{\bar{\pi}_{k+1}^{t}(\vx|\vy_{[k+1]})}{\pi_{k+1}(\vx|\vy_{[k+1]})}\d\vx \nonumber \\
&=2\int\nabla_{\vx}\cdot\big(-\bar{\pi}_{k+1}^{t}(\vx|\vy_{[k+1]})\bar{\vb}_{k+1}^{t}(\vx|\vy_{[k+1]})\big)\phi_{k+1}^{t}(\vx|\vy_{[k+1]})\d\vx \nonumber \\
&\quad+2\int\Delta_{\vx}\bar{\pi}_{k+1}^{t}(\vx|\vy_{[k+1]})\phi_{k+1}^{t}(\vx|\vy_{[k+1]})\d\vx \nonumber \\
&=2\underbrace{\bbE_{\bar{\mZ}_{t}}\big[(\bar{\vb}_{k+1}^{t}-\nabla_{\vx}\log\pi_{k+1})(\bar{\mZ}_{t}|\vy_{[k+1]})\cdot\nabla_{\vx}\phi_{k+1}^{t}(\bar{\mZ}_{t}|\vy_{[k+1]})\big]}_{\text{($\star$)}} \nonumber \\
&\quad-2\bbE_{\mZ}\big[\|\nabla_{\vx}\phi_{k+1}^{t}(\mZ|\vy_{[k+1]})\|_{2}^{2}\big], \label{eq:lemma:appendix:LMC:3:1}
\end{align}
where the first inequality invokes the chain rule, the second equality holds from Fokker-Planck equation (Lemma~\ref{lemma:appendix:LMC:2}) and~\eqref{eq:phi}, and the last equation used the Green's formula and Lemma~\ref{lemma:Dirichlet:energy}. Here the expectation $\bbE_{\bar{\mZ}_{t}}[\cdot]$ is taken with respect to $\bar{\mZ}_{t}\sim\bar{\pi}_{k+1}^{t}(\cdot|\vy_{[k+1]})$, while the expectation $\bbE_{\mZ}[\cdot]$ is taken with respect to $\mZ\sim\pi_{k+1}(\cdot|\vy_{[k+1]})$. Now it remains to estimate the term ($\star$) in~\eqref{eq:lemma:appendix:LMC:3:1}. Notice that
\begin{align*}
&\bbE_{\bar{\mZ}_{t}}\big[\bar{\vb}_{k+1}^{t}(\bar{\mZ}_{t}|\vy_{[k+1]})\cdot\nabla_{\vx}\phi_{k+1}^{t}(\bar{\mZ}_{t}|\vy_{[k+1]})\big] \\
&=\int\bar{\pi}_{k+1}^{t}(\vx|\vy_{[k+1]})\bar{\vb}_{k+1}^{t}(\vx|\vy_{[k+1]})\cdot\nabla_{\vx}\phi_{k+1}^{t}(\vx|\vy_{[k+1]})\d\vx \\
&=\iint\bar{\pi}_{k+1}^{\ell h,t}(\vx_{\ell h},\vx|\vy_{[k+1]})\nabla_{\vx}\log\pi_{k+1}(\vx_{\ell h}|\vy_{[k+1]})\cdot\nabla_{\vx}\phi_{k+1}^{t}(\vx|\vy_{[k+1]})\d\vx_{\ell h}\d\vx \\
&=\bbE_{(\bar{\mZ}_{\ell h},\bar{\mZ}_{t})}\big[\nabla_{\vx}\log\pi_{k+1}(\bar{\mZ}_{\ell h}|\vy_{[k+1]})\cdot\nabla_{\vx}\phi_{k+1}^{t}(\bar{\mZ}_{t}|\vy_{[k+1]})\big],
\end{align*}
where we used the definition of $\bar{\vb}_{k+1}^{t}(\vx|\vy_{[k+1]})$ as~\eqref{eq:lemma:appendix:LMC:2:1}. As a consequence, 
\begin{align}
(\star)
&=\bbE_{\bar{\mZ}_{t}}\big[(\bar{\vb}_{k+1}^{t}-\nabla_{\vx}\log\pi_{k+1})(\bar{\mZ}_{t}|\vy_{[k+1]})\cdot\nabla_{\vx}\phi_{k+1}^{t}(\bar{\mZ}_{t}|\vy_{[k+1]})\big] \nonumber \\
&=\bbE_{(\bar{\mZ}_{\ell h},\bar{\mZ}_{t})}\big[\ve_{k+1}(\bar{\mZ}_{\ell h},\bar{\mZ}_{t}|\vy_{[k+1]})\cdot\nabla_{\vx}\phi_{k+1}^{t}(\bar{\mZ}_{t}|\vy_{[k+1]})\big] \nonumber \\
&=\bbE_{(\bar{\mZ}_{\ell h},\bar{\mZ}_{t})}\Big[\ve_{k+1}(\bar{\mZ}_{\ell h},\bar{\mZ}_{t}|\vy_{[k+1]})\sqrt{\phi_{k+1}^{t}(\bar{\mZ}_{t}|\vy_{[k+1]})}\cdot\frac{\nabla_{\vx}\phi_{k+1}^{t}(\bar{\mZ}_{t}|\vy_{[k+1]})}{\sqrt{\phi_{k+1}^{t}(\bar{\mZ}_{t}|\vy_{[k+1]})}}\Big] \nonumber \\
&\leq\bbE_{(\bar{\mZ}_{\ell h},\bar{\mZ}_{t})}^{\frac{1}{2}}\big[\|\ve_{k+1}(\bar{\mZ}_{\ell h},\bar{\mZ}_{t}|\vy_{[k+1]})\|_{2}^{2}\phi_{k+1}^{t}(\bar{\mZ}_{t}|\vy_{[k+1]})\big]\bbE_{\mZ}^{\frac{1}{2}}\big[\|\nabla_{\vx}\phi_{k+1}^{t}(\mZ|\vy_{[k+1]})\|_{2}^{2}\big] \nonumber \\
&\leq\frac{1}{2}\bbE_{(\bar{\mZ}_{\ell h},\bar{\mZ}_{t})}\big[\|\ve_{k+1}(\bar{\mZ}_{\ell h},\bar{\mZ}_{t}|\vy_{[k+1]})\|_{2}^{2}\phi_{k+1}^{t}(\bar{\mZ}_{t}|\vy_{[k+1]})\big] \nonumber \\
&\quad+\frac{1}{2}\bbE_{\mZ}\big[\|\nabla_{\vx}\phi_{k+1}^{t}(\mZ|\vy_{[k+1]})\|_{2}^{2}\big] \nonumber \\
&=\frac{1}{2}\bbE_{\mZ}\big[\phi_{k+1}^{t}(\mZ|\vy_{[k+1]})^{2}\big]\bbE_{(\bar{\mZ}_{\ell h},\bar{\mZ}_{t})}\big[\|\ve_{k+1}(\bar{\mZ}_{\ell h},\bar{\mZ}_{t}|\vy_{[k+1]})\|_{2}^{2}\psi_{k+1}^{t}(\bar{\mZ}_{t}|\vy_{[k+1]})\big] \nonumber \\
&\quad+\frac{1}{2}\bbE_{\mZ}\big[\|\nabla_{\vx}\phi_{k+1}^{t}(\mZ|\vy_{[k+1]})\|_{2}^{2}\big], \label{eq:lemma:appendix:LMC:3:2}
\end{align}
where the first inequality invokes the Cauchy-Schwarz inequality, the second inequality follows from $ab\leq(a^{2}+b^{2})/2$. Substituting~\eqref{eq:lemma:appendix:LMC:3:2} into~\eqref{eq:lemma:appendix:LMC:3:1} completes the proof.
\end{proof}

\subsection{Dirichlet energy and chi-squared divergence}

\par We relate the Dirichlet energy to $\chi^{2}$-divergence by the following lemma.
\begin{lemma}\label{lemma:appendix:LMC:4}
Suppose Assumption~\ref{assumption:LSI:posterior} holds. Then 
\begin{equation*}
\frac{1}{2C_{\LSI}}\chi^{2}\big(\bar{\pi}_{k+1}^{t}(\cdot|\vy_{[k+1]})\|\pi_{k+1}(\cdot|\vy_{[k+1]})\big)\leq\bbE_{\mZ}\big[\|\nabla_{\vx}\phi_{k+1}^{t}(\mZ|\vy_{[k+1]})\|_{2}^{2}\big].
\end{equation*}
\end{lemma}

\begin{proof}[Proof of Lemma~\ref{lemma:appendix:LMC:4}]
A direct conclusion of Lemma~\ref{lemma:LSI:chi:energy}.
\end{proof}

\subsection{Discretization error}

\par The main results for the discretization error is stated as follows.
\begin{lemma}\label{lemma:appendix:LMC:5}
Suppose Assumptions~\ref{assumption:posterior:smooth} and~\ref{assumption:LSI:posterior} hold. Then for each $\ell h\leq t\leq(\ell+1)h$, 
\begin{align*}
&\bbE_{\mZ}\big[\phi_{k+1}^{t}(\mZ|\vy_{[k+1]})^{2}\big]\bbE_{(\bar{\mZ}_{\ell h},\bar{\mZ}_{t})}\Big[\|\ve_{k+1}(\bar{\mZ}_{\ell h},\bar{\mZ}_{t}|\vy_{[k+1]})\|_{2}^{2}\psi_{k+1}^{t}(\bar{\mZ}_{t}|\vy_{[k+1]})\Big] \\
&\leq 80C_{\LSI}\lambda^{2}(t-\ell h)\bbE_{\mZ}[\|\nabla_{\vx}\phi_{k+1}^{t}(\mZ|\vy_{[k+1]})\|_{2}^{2}]+20d\lambda^{2}(t-\ell h)\bbE_{\mZ}\big[\phi_{k+1}^{t}(\mZ|\vy_{[k+1]})^{2}\big],
\end{align*}
where the step size $h>0$ satisfies $4\lambda h\leq 1$.
\end{lemma}

\par To verify Lemma~\ref{lemma:appendix:LMC:5}, we provide some auxiliary lemmas.
\begin{lemma}\label{lemma:appendix:LMC:6}
Suppose Assumption~\ref{assumption:posterior:smooth} holds. Then for each $\ell h\leq t\leq(\ell+1)h$, 
\begin{align*}
&\bbE_{(\bar{\mZ}_{\ell h},\bar{\mZ}_{t})}\Big[\|\ve_{k+1}(\bar{\mZ}_{\ell h},\bar{\mZ}_{t}|\vy_{[k+1]})\|_{2}^{2}\psi_{k+1}^{t}(\bar{\mZ}_{t}|\vy_{[k+1]})\Big]  \\
&\leq 8\lambda^{2}(t-\ell h)^{2}\bbE_{\bar{\mZ}_{t}}\Big[\|\nabla_{\vx}\log\pi_{k+1}(\bar{\mZ}_{t}|\vy_{[k+1]})\|_{2}^{2}\psi_{k+1}^{t}(\bar{\mZ}_{t}|\vy_{[k+1]})\Big] \\
&\quad+6\lambda^{2}\bbE_{(\mB_{\ell h},\mB_{t},\bar{\mZ}_{t})}\Big[\|\mB_{t}-\mB_{\ell h}\|_{2}^{2}\psi_{k+1}^{t}(\bar{\mZ}_{t}|\vy_{[k+1]})\Big],
\end{align*}
where the step size $h>0$ satisfies $4\lambda h\leq 1$.
\end{lemma}

\begin{proof}[Proof of Lemma~\ref{lemma:appendix:LMC:6}]
Recall the solution to the Langevin Monte Carlo~\eqref{eq:appendix:discretization:1}, which implies
\begin{equation}\label{eq:lemma:appendix:LMC:6:1}
\|\bar{\mZ}_{t}-\bar{\mZ}_{\ell h}\|_{2}^{2}\leq 2(t-\ell h)^{2}\|\nabla_{\vx}\log\pi_{k+1}(\bar{\mZ}_{\ell h}|\vy_{[k+1]})\|_{2}^{2}+4\|\mB_{t}-\mB_{\ell h}\|_{2}^{2},
\end{equation}
where the first inequality invokes the triangular inequality, and last inequality holds from the fact $\ell h\leq t\leq(\ell+1)h$. According to Assumption~\ref{assumption:posterior:smooth}, we have 
\begin{align}
&\|\nabla_{\vx}\log\pi_{k+1}(\bar{\mZ}_{t}|\vy_{[k+1]})-\nabla_{\vx}\log\pi_{k+1}(\bar{\mZ}_{\ell h}|\vy_{[k+1]})\|_{2}^{2} \nonumber \\
&\leq\lambda^{2}\|\bar{\mZ}_{t}-\bar{\mZ}_{\ell h}\|_{2}^{2} \nonumber \\
&\leq 2\lambda^{2}(t-\ell h)^{2}\|\nabla_{\vx}\log\pi_{k+1}(\bar{\mZ}_{\ell h}|\vy_{[k+1]})\|_{2}^{2}+4\lambda^{2}\|\mB_{t}-\mB_{\ell h}\|_{2}^{2}, \label{eq:lemma:appendix:LMC:6:2}
\end{align}
where the last inequality used~\eqref{eq:appendix:discretization:1}. As a consequence, for each step size $h$ with $4\lambda h\leq 1$, 
\begin{align*}
&\|\nabla_{\vx}\log\pi_{k+1}(\bar{\mZ}_{\ell h}|\vy_{[k+1]})\|_{2}^{2} \\
&\leq 2\|\nabla_{\vx}\log\pi_{k+1}(\bar{\mZ}_{t}|\vy_{[k+1]})-\nabla_{\vx}\log\pi_{k+1}(\bar{\mZ}_{\ell h}|\vy_{[k+1]})\|_{2}^{2}+2\|\nabla_{\vx}\log\pi_{k+1}(\bar{\mZ}_{t}|\vy_{[k+1]})\|_{2}^{2} \\
&\leq 4\lambda^{2}(t-\ell h)^{2}\|\nabla_{\vx}\log\pi_{k+1}(\bar{\mZ}_{\ell h}|\vy_{[k+1]})\|_{2}^{2}+8\lambda^{2}\|\mB_{t}-\mB_{\ell h}\|_{2}^{2}+2\|\nabla_{\vx}\log\pi_{k+1}(\bar{\mZ}_{t}|\vy_{[k+1]})\|_{2}^{2} \\
&\leq \frac{1}{4}\|\nabla_{\vx}\log\pi_{k+1}(\bar{\mZ}_{\ell h}|\vy_{[k+1]})\|_{2}^{2}+8\lambda^{2}\|\mB_{t}-\mB_{\ell h}\|_{2}^{2}+2\|\nabla_{\vx}\log\pi_{k+1}(\bar{\mZ}_{t}|\vy_{[k+1]})\|_{2}^{2},
\end{align*}
where the first inequality follows from the triangular inequality. Rearranging this inequality yields 
\begin{align}
&\|\nabla_{\vx}\log\pi_{k+1}(\bar{\mZ}_{\ell h}|\vy_{[k+1]})\|_{2}^{2} \nonumber \\
&\leq 16\lambda^{2}\|\mB_{t}-\mB_{\ell h}\|_{2}^{2}+4\|\nabla_{\vx}\log\pi_{k+1}(\bar{\mZ}_{t}|\vy_{[k+1]})\|_{2}^{2}. \label{eq:lemma:appendix:LMC:6:3}
\end{align}
Then substituting~\eqref{eq:lemma:appendix:LMC:6:3} into~\eqref{eq:lemma:appendix:LMC:6:2} yields
\begin{align}
&\|\nabla_{\vx}\log\pi_{k+1}(\bar{\mZ}_{t}|\vy_{[k+1]})-\nabla_{\vx}\log\pi_{k+1}(\bar{\mZ}_{\ell h}|\vy_{[k+1]})\|_{2}^{2} \nonumber \\
&\leq 8\lambda^{2}(t-\ell h)^{2}\|\nabla_{\vx}\log\pi_{k+1}(\bar{\mZ}_{t}|\vy_{[k+1]})\|_{2}^{2}+\{32\lambda^{4}h^{2}+4\lambda^{2}\}\|\mB_{t}-\mB_{\ell h}\|_{2}^{2} \nonumber \\
&\leq 8\lambda^{2}(t-\ell h)^{2}\|\nabla_{\vx}\log\pi_{k+1}(\bar{\mZ}_{t}|\vy_{[k+1]})\|_{2}^{2}+6\lambda^{2}\|\mB_{t}-\mB_{\ell h}\|_{2}^{2}, \label{eq:lemma:appendix:LMC:6:4}
\end{align}
where we used the inequality $4\lambda h\leq 1$. Multiplying both sides of~\eqref{eq:lemma:appendix:LMC:6:4} by $\psi_{k+1}^{t}$ and taking expectation complete the proof.
\end{proof}

\par We bound the first term in Lemma~\ref{lemma:appendix:LMC:5} by the following lemma, which is based on~\cite[Lemma 20]{Chewi2024Analysis}.
\begin{lemma}\label{lemma:appendix:LMC:8}
Suppose Assumption~\ref{assumption:posterior:smooth} holds. Then for each $\ell h\leq t\leq(\ell+1)h$, 
\begin{align*}
&\bbE_{\bar{\mZ}_{t}}\big[\|\nabla_{\vx}\log\pi_{k+1}(\bar{\mZ}_{t}|\vy_{[k+1]})\|_{2}^{2}\psi_{k+1}^{t}(\bar{\mZ}_{t}|\vy_{[k+1]})\big] \\
&\leq\frac{4\bbE_{\mZ}[\|\nabla_{\vx}\phi_{k+1}^{t}(\mZ|\vy_{[k+1]})\|_{2}^{2}]}{\bbE_{\mZ}[\phi_{k+1}^{t}(\mZ|\vy_{[k+1]})^{2}]}+2d\lambda,
\end{align*}
where the expectation $\bbE_{\mZ}[\cdot]$ is taken with respect to $\mZ\sim\pi_{k+1}(\cdot|\vy_{[k+1]})$.
\end{lemma}

\begin{proof}[Proof of Lemma~\ref{lemma:appendix:LMC:8}]
According to~\eqref{eq:psi:property}, we define a change of measure
\begin{equation*}
\mu_{k+1}^{t}(\vx|\vy_{[k+1]}):=\psi_{k+1}^{t}(\vx|\vy_{[k+1]})\bar{\pi}_{k+1}^{t}(\vx|\vy_{[k+1]}).
\end{equation*}
Then it suffices to consider the expectation under this change of measures
\begin{align*}
&\bbE_{\widetilde{\mZ}_{t}}\big[\|\nabla_{\vx}\log\pi_{k+1}(\widetilde{\mZ}_{t}|\vy_{[k+1]})\|_{2}^{2}\big] \\
&=\bbE_{\bar{\mZ}_{t}}\big[\|\nabla_{\vx}\log\pi_{k+1}(\bar{\mZ}_{t}|\vy_{[k+1]})\|_{2}^{2}\psi_{k+1}^{t}(\bar{\mZ}_{t}|\vy_{[k+1]})\big],
\end{align*}
where $\widetilde{\mZ}_{t}$ is a random variable with probability density $\mu_{k+1}^{t}(\cdot|\vy_{[k+1]})$. We first verify that 
\begin{align}
&\bbE_{\widetilde{\mZ}_{t}}\big[\|\nabla_{\vx}\log\pi_{k+1}(\widetilde{\mZ}_{t}|\vy_{[k+1]})\|_{2}^{2}\big] \nonumber \\
&\leq\underbrace{\bbE_{\widetilde{\mZ}_{t}}\big[\Delta_{\vx}\log\pi_{k+1}(\widetilde{\mZ}_{t}|\vy_{[k+1]})\big]}_{\text{(i)}}+\underbrace{\bbE_{\mZ}\Big[\nabla_{\vx}\log\pi_{k+1}(\mZ|\vy_{[k+1]})\cdot\nabla_{\vx}\frac{\mu_{k+1}^{t}(\mZ|\vy_{[k+1]})}{\pi_{k+1}(\mZ|\vy_{[k+1]})}\Big]}_{\text{(ii)}}. \label{eq:lemma:appendix:LMC:8:1}
\end{align}
Indeed, according to the Green's formula, we obtain 
\begin{align*}
&\bbE_{\widetilde{\mZ}_{t}}\big[\Delta_{\vx}\log\pi_{k+1}(\widetilde{\mZ}_{t}|\vy_{[k+1]})\big] \\
&=\int\Delta_{\vx}\log\pi_{k+1}(\vx|\vy_{[k+1]})\frac{\mu_{k+1}^{t}(\vx|\vy_{[k+1]})}{\pi_{k+1}(\vx|\vy_{[k+1]})}\pi_{k+1}(\vx|\vy_{[k+1]})\d\vx \\
&=-\int\nabla_{\vx}\log\pi_{k+1}(\vx|\vy_{[k+1]})\cdot\nabla_{\vx}\Big(\frac{\mu_{k+1}^{t}(\vx|\vy_{[k+1]})}{\pi_{k+1}(\vx|\vy_{[k+1]})}\pi_{k+1}(\vx|\vy_{[k+1]})\Big)\d\vx \\
&=-\bbE_{\mZ}\Big[\nabla_{\vx}\log\pi_{k+1}(\mZ|\vy_{[k+1]})\cdot\nabla_{\vx}\frac{\mu_{k+1}^{t}(\mZ|\vy_{[k+1]})}{\pi_{k+1}(\mZ|\vy_{[k+1]})}\Big]+\bbE_{\widetilde{\mZ}_{t}}\big[\|\nabla_{\vx}\log\pi_{k+1}(\widetilde{\mZ}_{t}|\vy_{[k+1]})\|_{2}^{2}\big].
\end{align*}
For the term (i) in~\eqref{eq:lemma:appendix:LMC:8:2}, it follows from Assumption~\ref{assumption:posterior:smooth} that 
\begin{equation}\label{eq:lemma:appendix:LMC:8:2}
\bbE_{\widetilde{\mZ}_{t}}\big[\Delta_{\vx}\log\pi_{k+1}(\widetilde{\mZ}_{t}|\vy_{[k+1]})\big]\leq d\lambda.
\end{equation}
For the term (ii) in~\eqref{eq:lemma:appendix:LMC:8:2}, we find
\begin{align}
&\bbE_{\mZ}\Big[\nabla_{\vx}\log\pi_{k+1}(\mZ|\vy_{[k+1]})\cdot\nabla_{\vx}\frac{\mu_{k+1}^{t}(\mZ|\vy_{[k+1]})}{\pi_{k+1}(\mZ|\vy_{[k+1]})}\Big] \nonumber \\
&=2\bbE_{\mZ}\Big[\nabla_{\vx}\log\pi_{k+1}(\mZ|\vy_{[k+1]})\sqrt{\frac{\mu_{k+1}^{t}(\mZ|\vy_{[k+1]})}{\pi_{k+1}(\mZ|\vy_{[k+1]})}}\cdot\nabla_{\vx}\sqrt{\frac{\mu_{k+1}^{t}(\mZ|\vy_{[k+1]})}{\pi_{k+1}(\mZ|\vy_{[k+1]})}}\Big] \nonumber \\
&\leq\frac{1}{2}\bbE_{\mZ}\Big[\big\|\nabla_{\vx}\log\pi_{k+1}(\mZ|\vy_{[k+1]})\big\|_{2}^{2}\frac{\mu_{k+1}^{t}(\mZ|\vy_{[k+1]})}{\pi_{k+1}(\mZ|\vy_{[k+1]})}\Big]+2\bbE_{\mZ}\Big[\big\|\nabla_{\vx}\sqrt{\frac{\mu_{k+1}^{t}(\mZ|\vy_{[k+1]})}{\pi_{k+1}(\mZ|\vy_{[k+1]})}}\big\|_{2}^{2}\Big] \nonumber \\
&=\frac{1}{2}\bbE_{\widetilde{\mZ}_{t}}\big[\|\nabla_{\vx}\log\pi_{k+1}(\widetilde{\mZ}_{t}|\vy_{[k+1]})\|_{2}^{2}\big]+2\bbE_{\mZ}\Big[\|\nabla_{\vx}\sqrt{\frac{\mu_{k+1}^{t}(\mZ|\vy_{[k+1]})}{\pi_{k+1}(\mZ|\vy_{[k+1]})}}\|_{2}^{2}\Big], \label{eq:lemma:appendix:LMC:8:3}
\end{align}
where the second equality holds from the chain rule, and the inequality follows from Young's inequality. Substituting~\eqref{eq:lemma:appendix:LMC:8:2} and~\eqref{eq:lemma:appendix:LMC:8:3} into~\eqref{eq:lemma:appendix:LMC:8:1} implies 
\begin{equation}\label{eq:lemma:appendix:LMC:8:4}
\bbE_{\widetilde{\mZ}_{t}}\big[\|\nabla_{\vx}\log\pi_{k+1}(\widetilde{\mZ}_{t}|\vy_{[k+1]})\|_{2}^{2}\big]\leq 4\bbE_{\mZ}\Big[\|\nabla_{\vx}\sqrt{\frac{\mu_{k+1}^{t}(\mZ|\vy_{[k+1]})}{\pi_{k+1}(\mZ|\vy_{[k+1]})}}\|_{2}^{2}\Big]+2d\lambda. 
\end{equation}
Applying Lemma~\ref{lemma:score:fisher} to~\eqref{eq:lemma:appendix:LMC:8:4} yields 
\begin{align*}
&\bbE_{\widetilde{\mZ}_{t}}\big[\|\nabla_{\vx}\log\pi_{k+1}(\widetilde{\mZ}_{t}|\vy_{[k+1]})\|_{2}^{2}\big] \\
&\leq \bbE_{\widetilde{\mZ}_{t}}\Big[\|\nabla_{\vx}\log\frac{\mu_{k+1}^{t}(\widetilde{\mZ}_{t}|\vy_{[k+1]})}{\pi_{k+1}(\widetilde{\mZ}_{t}|\vy_{[k+1]})}\big\|_{2}^{2}\Big]+2d\lambda \\
&=\bbE_{\bar{\mZ}_{t}}\Big[\|\nabla_{\vx}\log\frac{\mu_{k+1}^{t}(\bar{\mZ}_{t}|\vy_{[k+1]})}{\pi_{k+1}(\bar{\mZ}_{t}|\vy_{[k+1]})}\|_{2}^{2}\psi_{k+1}^{t}(\bar{\mZ}_{t}|\vy_{[k+1]})\Big]+2d\lambda.
\end{align*}
Finally, using Lemma~\ref{lemma:psi:phi} completes the proof.
\end{proof}

\par For the second term in Lemma~\ref{lemma:appendix:LMC:5}, we have the following result by a similar argument as~\cite[Lemma 19]{Chewi2024Analysis}.
\begin{lemma}\label{lemma:appendix:LMC:7}
Suppose Assumptions~\ref{assumption:posterior:smooth} and~\ref{assumption:LSI:posterior} hold. Then for each $\ell h\leq t\leq(\ell+1)h$, 
\begin{align*}
&\bbE\big[\|\mB_{t}-\mB_{\ell h}\|_{2}^{2}\psi_{k+1}^{t}(\bar{\mZ}_{t}|\vy_{[k+1]})\big] \\
&\leq 3d(t-\ell h)+8C_{\LSI}(t-\ell h)\frac{\bbE_{\mZ}[\|\nabla_{\vx}\phi_{k+1}^{t}(\mZ|\vy_{[k+1]})\|_{2}^{2}]}{\bbE_{\mZ}[\phi_{k+1}^{t}(\mZ|\vy_{[k+1]})^{2}]},
\end{align*}
where the expectation $\bbE_{\mZ}[\cdot]$ is taken with respect to $\mZ\sim\pi_{k+1}(\cdot|\vy_{[k+1]})$.
\end{lemma}

\begin{proof}[Proof of Lemma~\ref{lemma:appendix:LMC:7}]
According to Donsker-Varadhan variational principle (Lemma~\ref{lemma:DonskerVaradhan}), for each $s>0$, we have  
\begin{align*}
&\bbE\big[\|\mB_{t}-\mB_{\ell h}\|_{2}^{2}\psi_{k+1}^{t}(\bar{\mZ}_{t}|\vy_{[k+1]})\big]-\bbE[\|\mB_{t}-\mB_{\ell h}\|_{2}^{2}] \\
&=\frac{1}{s}\bbE\big[s(\|\mB_{t}-\mB_{\ell h}\|_{2}^{2}-\bbE[\|\mB_{t}-\mB_{\ell h}\|_{2}^{2}])\psi_{k+1}^{t}(\bar{\mZ}_{t}|\vy_{[k+1]})\big] \\
&\leq\frac{1}{s}\kl\big(\mu_{k+1}^{t}(\cdot|\vy_{[k+1]})\|\bar{\pi}_{k+1}^{t}(\cdot|\vy_{[k+1]})\big) \\
&\quad+\frac{1}{s}\log\bbE\big[\exp\{s(\|\mB_{t}-\mB_{\ell h}\|_{2}^{2}-\bbE[\|\mB_{t}-\mB_{\ell h}\|_{2}^{2}])\}\big].
\end{align*}
Rearranging the above inequality yields 
\begin{align}
&\bbE\big[\|\mB_{t}-\mB_{\ell h}\|_{2}^{2}\psi_{k+1}^{t}(\bar{\mZ}_{t}|\vy_{[k+1]})\big] \nonumber \\
&\leq\underbrace{\bbE[\|\mB_{t}-\mB_{\ell h}\|_{2}^{2}]}_{\textbf{(i)}}+\frac{1}{s}\underbrace{\kl\big(\mu_{k+1}^{t}(\cdot|\vy_{[k+1]})\|\bar{\pi}_{k+1}^{t}(\cdot|\vy_{[k+1]})\big)}_{\textbf{(ii)}} \nonumber \\
&\quad+\frac{1}{s}\underbrace{\log\bbE\big[\exp\{s(\|\mB_{t}-\mB_{\ell h}\|_{2}^{2}-\bbE[\|\mB_{t}-\mB_{\ell h}\|_{2}^{2}])\}\big]}_{\text{(iii)}}. \label{eq:lemma:appendix:LMC:7:1}
\end{align}
For the term (i) in~\eqref{eq:lemma:appendix:LMC:7:1}, it holds that 
\begin{equation}\label{eq:lemma:appendix:LMC:7:2}
\bbE[\|\mB_{t}-\mB_{\ell h}\|_{2}^{2}]=d(t-\ell h).
\end{equation}
For the term (ii) in~\eqref{eq:lemma:appendix:LMC:7:1}, we find 
\begin{align}
&\kl\big(\mu_{k+1}^{t}(\cdot|\vy_{[k+1]})\|\bar{\pi}_{k+1}^{t}(\cdot|\vy_{[k+1]})\big) \nonumber  \\
&=\int\mu_{k+1}^{t}(\vx|\vy_{[k+1]})\log\psi_{k+1}^{t}(\vx|\vy_{[k+1]})\d\vx \nonumber  \\
&=\frac{1}{2}\int\mu_{k+1}^{t}(\vx|\vy_{[k+1]})\log\frac{\phi_{k+1}^{t}(\vx|\vy_{[k+1]})^{2}}{\bbE_{\bar{\mZ}_{t}}[\phi_{k+1}^{t}(\bar{\mZ}_{t}|\vy_{[k+1]})]^{2}}\d\vx \nonumber  \\
&=\frac{1}{2}\int\mu_{k+1}^{t}(\vx|\vy_{[k+1]})\Big\{\log\frac{\phi_{k+1}^{t}(\vx|\vy_{[k+1]})^{2}}{\bbE_{\bar{\mZ}_{t}}[\phi_{k+1}^{t}(\bar{\mZ}_{t}|\vy_{[k+1]})]}-\log\bbE_{\bar{\mZ}_{t}}[\phi_{k+1}^{t}(\bar{\mZ}_{t}|\vy_{[k+1]})]\Big\}\d\vx \nonumber  \\
&\leq\frac{1}{2}\int\mu_{k+1}^{t}(\vx|\vy_{[k+1]})\log\frac{\phi_{k+1}^{t}(\vx|\vy_{[k+1]})^{2}}{\bbE_{\bar{\mZ}_{t}}[\phi_{k+1}^{t}(\bar{\mZ}_{t}|\vy_{[k+1]})]}\d\vx \nonumber  \\
&=\frac{1}{2}\int\mu_{k+1}^{t}(\vx|\vy_{[k+1]})\log\big\{\psi_{k+1}^{t}(\vx|\vy_{[k+1]})\phi_{k+1}^{t}(\vx|\vy_{[k+1]})\big\}\d\vx \nonumber  \\
&=\frac{1}{2}\kl\big(\mu_{k+1}^{t}(\cdot|\vy_{[k+1]})\|\pi_{k+1}(\cdot|\vy_{[k+1]})\big) \nonumber  \\
&\leq\frac{C_{\LSI}}{4}\bbE_{\widetilde{\mZ}_{t}}\Big[\big\|\nabla\log\big\{\psi_{k+1}^{t}(\widetilde{\mZ}_{t}|\vy_{[k+1]})\phi_{k+1}^{t}(\widetilde{\mZ}_{t}|\vy_{[k+1]})\big\}\big\|_{2}^{2}\Big] \nonumber \\
&=C_{\LSI}\frac{\bbE_{\mZ}[\|\nabla_{\vx}\phi_{k+1}^{t}(\mZ|\vy_{[k+1]})\|_{2}^{2}]}{\bbE_{\mZ}[\phi_{k+1}^{t}(\mZ|\vy_{[k+1]})^{2}]}, \label{eq:lemma:appendix:LMC:7:3}
\end{align}
where the second equality holds from~\eqref{eq:phi:property} and~\eqref{eq:psi}, the second inequality invokes Lemmas~\ref{lemma:kl:fisher} and~\ref{lemma:score:fisher}, and the last equality is due to Lemma~\ref{lemma:psi:phi}. Finally, we consider the term (iii) in~\eqref{eq:lemma:appendix:LMC:7:1}. Applying Lemma~\ref{lemma:chernoff:chi:square} deduces 
\begin{equation}\label{eq:lemma:appendix:LMC:7:4}
\log\bbE\big[\exp\{s\|\mB_{t}-\mB_{\ell h}\|_{2}^{2}-\bbE[\|\mB_{t}-\mB_{\ell h}\|_{2}^{2}]\}\big]\leq 2ds(t-\ell h),
\end{equation}
provided that $4s(t-\ell h)\leq 1$. Substituting~\eqref{eq:lemma:appendix:LMC:7:2},~\eqref{eq:lemma:appendix:LMC:7:3}, and~\eqref{eq:lemma:appendix:LMC:7:4} into~\eqref{eq:lemma:appendix:LMC:7:1} implies 
\begin{align*}
&\bbE\big[\|\mB_{t}-\mB_{\ell h}\|_{2}^{2}\psi_{k+1}^{t}(\bar{\mZ}_{t}|\vy_{[k+1]})\big] \\
&\leq d(t-\ell h)+\frac{C_{\LSI}}{s}\frac{\bbE_{\mZ}[\|\nabla_{\vx}\phi_{k+1}^{t}(\mZ|\vy_{[k+1]})\|_{2}^{2}]}{\bbE_{\mZ}[\phi_{k+1}^{t}(\mZ|\vy_{[k+1]})^{2}]}+2d(t-\ell h),
\end{align*}
for each $s>0$ such that $4s(t-\ell h)\leq 1$. Letting $8s(t-\ell h)=1$ completes the proof.
\end{proof}

\par With the help of the preceding three lemmas, we can now prove Lemma~\ref{lemma:appendix:LMC:5}.
\begin{proof}[Proof of Lemma~\ref{lemma:appendix:LMC:5}]
Applying Lemmas~\ref{lemma:appendix:LMC:8} and~\ref{lemma:appendix:LMC:7} into Lemma~\ref{lemma:appendix:LMC:6} implies 
\begin{align*}
&\bbE_{(\bar{\mZ}_{\ell h},\bar{\mZ}_{t})}\Big[\|\ve_{k+1}(\bar{\mZ}_{\ell h},\bar{\mZ}_{t}|\vy_{[k+1]})\|_{2}^{2}\psi_{k+1}^{t}(\bar{\mZ}_{t}|\vy_{[k+1]})\Big]  \\
&\leq 8\lambda^{2}(t-\ell h)^{2}\Big\{\frac{4\bbE_{\mZ}[\|\nabla_{\vx}\phi_{k+1}^{t}(\mZ|\vy_{[k+1]})\|_{2}^{2}]}{\bbE_{\mZ}[\phi_{k+1}^{t}(\mZ|\vy_{[k+1]})^{2}]}+2d\lambda\Big\} \\
&\quad+6\lambda^{2}\Big\{3d(t-\ell h)+8C_{\LSI}(t-\ell h)\frac{\bbE_{\mZ}[\|\nabla_{\vx}\phi_{k+1}^{t}(\mZ|\vy_{[k+1]})\|_{2}^{2}]}{\bbE_{\mZ}[\phi_{k+1}^{t}(\mZ|\vy_{[k+1]})^{2}]}\Big\} \\
&=\big\{32\lambda^{2}(t-\ell h)^{2}+48C_{\LSI}\lambda^{2}(t-\ell h)\big\}\frac{\bbE_{\mZ}[\|\nabla_{\vx}\phi_{k+1}^{t}(\mZ|\vy_{[k+1]})\|_{2}^{2}]}{\bbE_{\mZ}[\phi_{k+1}^{t}(\mZ|\vy_{[k+1]})^{2}]} \\
&\quad+16d\lambda^{3}(t-\ell h)^{2}+16d\lambda^{2}(t-\ell h) \\
&\leq 80C_{\LSI}\lambda^{2}(t-\ell h)\frac{\bbE_{\mZ}[\|\nabla_{\vx}\phi_{k+1}^{t}(\mZ|\vy_{[k+1]})\|_{2}^{2}]}{\bbE_{\mZ}[\phi_{k+1}^{t}(\mZ|\vy_{[k+1]})^{2}]}+20d\lambda^{2}(t-\ell h),
\end{align*}
where the last inequality holds from $4\lambda h\leq 1$. This completes the proof.
\end{proof}

\subsection{Proof of the convergence of Langevin Monte Carlo}

\par Combining Lemmas~\ref{lemma:appendix:LMC:3},~\ref{lemma:appendix:LMC:4}, and~\ref{lemma:appendix:LMC:5} achieves the following recursion of $\chi^{2}$-divergence.
\begin{lemma}\label{lemma:appendix:LMC:9}
Suppose Assumptions~\ref{assumption:posterior:smooth} and~\ref{assumption:LSI:posterior} hold. Then for each $0\leq\ell\leq K-1$,
\begin{align*}
&\chi^{2}\big(\bar{\pi}_{k+1}^{(\ell+1)h}(\cdot|\vy_{[k+1]})\|\pi_{k+1}(\cdot|\vy_{[k+1]})\big) \\
&\leq\exp\Big(-\frac{h}{5C_{\LSI}}\Big)\chi^{2}\big(\bar{\pi}_{k+1}^{\ell h}(\cdot|\vy_{[k+1]})\|\pi_{k+1}(\cdot|\vy_{[k+1]})\big)+20d\lambda^{2}h^{2},
\end{align*}
where the step size $h>0$ satisfies $400dC_{\LSI}\lambda^{2}h\leq 1$.
\end{lemma}

\begin{proof}[Proof of Lemma~\ref{lemma:appendix:LMC:9}]
Plugging Lemma~\ref{lemma:appendix:LMC:5} into Lemma~\ref{lemma:appendix:LMC:3} implies
\begin{align}
&\frac{\d}{\dt}\chi^{2}\big(\bar{\pi}_{k+1}^{t}(\cdot|\vy_{[k+1]})\|\pi_{k+1}(\cdot|\vy_{[k+1]})\big) \nonumber \\
&\leq-\bbE_{\mZ}\Big[\big\|\nabla_{\vx}\phi_{k+1}^{t}(\mZ|\vy_{[k+1]})\big\|_{2}^{2}\Big]+80C_{\LSI}\lambda^{2}h\bbE_{\mZ}[\|\nabla_{\vx}\phi_{k+1}^{t}(\mZ|\vy_{[k+1]})\|_{2}^{2}] \nonumber \\
&\quad+20d\lambda^{2}(t-\ell h)\bbE_{\mZ}[\phi_{k+1}^{t}(\mZ|\vy_{[k+1]})^{2}] \nonumber \\
&\leq-\frac{4}{5}\underbrace{\bbE_{\mZ}\Big[\big\|\nabla_{\vx}\phi_{k+1}^{t}(\mZ|\vy_{[k+1]})\big\|_{2}^{2}\Big]}_{\text{(i)}}+20d\lambda^{2}(t-\ell h)\underbrace{\bbE_{\mZ}[\phi_{k+1}^{t}(\mZ|\vy_{[k+1]})^{2}]}_{\text{(ii)}}, \label{eq:lemma:appendix:LMC:9:1}
\end{align}
where the second inequality invokes $400C_{\LSI}\lambda^{2}h\leq 1$. For the term (i) in~\eqref{eq:lemma:appendix:LMC:9:1}, it follows from Lemma~\ref{lemma:appendix:LMC:4} that 
\begin{equation}\label{eq:lemma:appendix:LMC:9:2}
\frac{1}{2C_{\LSI}}\chi^{2}\big(\bar{\pi}_{k+1}^{t}(\cdot|\vy_{[k+1]})\|\pi_{k+1}(\cdot|\vy_{[k+1]})\big)\leq\bbE_{\mZ}\Big[\big\|\nabla_{\vx}\phi_{k+1}^{t}(\mZ|\vy_{[k+1]})\big\|_{2}^{2}\Big].
\end{equation}
For the term (ii) in~\eqref{eq:lemma:appendix:LMC:9:1}, using the definition of $\chi^{2}$-divergence and~\eqref{eq:phi},
\begin{equation}\label{eq:lemma:appendix:LMC:9:3}
\bbE_{\mZ}[\phi_{k+1}^{t}(\mZ|\vy_{[k+1]})^{2}]\leq\chi^{2}\big(\bar{\pi}_{k+1}^{t}(\cdot|\vy_{[k+1]})\|\pi_{k+1}(\cdot|\vy_{[k+1]})\big)+1.
\end{equation}
Substituting~\eqref{eq:lemma:appendix:LMC:9:2} and~\eqref{eq:lemma:appendix:LMC:9:3} into~\eqref{eq:lemma:appendix:LMC:9:1} yields that for $h$ satisfying $100dC_{\LSI}\lambda^{2}h\leq 1$,
\begin{align}
&\frac{\d}{\dt}\chi^{2}\big(\bar{\pi}_{k+1}^{t}(\cdot|\vy_{[k+1]})\|\pi_{k+1}(\cdot|\vy_{[k+1]})\big) \nonumber \\
&\leq-\frac{1}{5C_{\LSI}}\chi^{2}\big(\bar{\pi}_{k+1}^{t}(\cdot|\vy_{[k+1]})\|\pi_{k+1}(\cdot|\vy_{[k+1]})\big)+20d\lambda^{2}h. \label{eq:lemma:appendix:LMC:9:4}
\end{align}
Multiplying both sides of~\eqref{eq:lemma:appendix:LMC:9:4} by $\exp(\frac{t}{5C_{\LSI}})$ deduces 
\begin{equation}\label{eq:lemma:appendix:LMC:9:5}
\frac{\d}{\dt}\Big(\exp\Big(\frac{t}{5C_{\LSI}}\Big)\chi^{2}\big(\bar{\pi}_{k+1}^{t}(\cdot|\vy_{[k+1]})\|\pi_{k+1}(\cdot|\vy_{[k+1]})\big)\Big)\leq 20d\lambda^{2}h\exp\Big(\frac{t}{5C_{\LSI}}\Big).
\end{equation}
Before proceeding, we verify a useful inequality
\begin{equation}\label{eq:lemma:appendix:LMC:9:6}
\exp\Big(\frac{h}{5C_{\LSI}}\Big)-1\leq\frac{h}{5C_{\LSI}}.
\end{equation}
In fact, since that $400C_{\LSI}\lambda^{2}h\leq 1$ and $C_{\LSI}\lambda\geq 1$ (Lemma~\ref{lemma:LSI:Lip}), it holds that 
\begin{equation*}
0<\frac{h}{5C_{\LSI}}\leq\frac{h}{5C_{\LSI}}\frac{1}{400C_{\LSI}\lambda^{2}h}<1,
\end{equation*}
which implies~\eqref{eq:lemma:appendix:LMC:9:6} directly. Then integrating both sides of~\eqref{eq:lemma:appendix:LMC:9:5} from $\ell h$ to $(\ell+1)h$ yields 
\begin{align*}
&\chi^{2}\big(\bar{\pi}_{k+1}^{(\ell+1)h}(\cdot|\vy_{[k+1]})\|\pi_{k+1}(\cdot|\vy_{[k+1]})\big) \\
&\leq\exp\Big(-\frac{h}{5C_{\LSI}}\Big)\chi^{2}\big(\bar{\pi}_{k+1}^{\ell h}(\cdot|\vy_{[k+1]})\|\pi_{k+1}(\cdot|\vy_{[k+1]})\big) \\
&\quad+100dC_{\LSI}\lambda^{2}h\exp\Big(-\frac{h}{5C_{\LSI}}\Big)\Big\{\exp\Big(\frac{h}{5C_{\LSI}}\Big)-1\Big\} \\
&\leq\exp\Big(-\frac{h}{5C_{\LSI}}\Big)\chi^{2}\big(\bar{\pi}_{k+1}^{\ell h}(\cdot|\vy_{[k+1]})\|\pi_{k+1}(\cdot|\vy_{[k+1]})\big)+20d\lambda^{2}h^{2},
\end{align*}
where the last inequality invokes~\eqref{eq:lemma:appendix:LMC:9:6}. This completes the proof.
\end{proof}

\begin{proof}[Proof of Lemma~\ref{lemma:appendix:discretization}]
It is straightforward from Lemma~\ref{lemma:appendix:LMC:9} that
\begin{align*}
&\chi^{2}\big(\bar{\pi}_{k+1}^{\ell h}(\cdot|\vy_{[k+1]})\|\pi_{k+1}(\cdot|\vy_{[k+1]})\big) \\
&\leq\exp\Big(-\frac{\ell h}{5C_{\LSI}}\Big)\chi^{2}\big(\pi_{k+1}^{0}(\cdot|\vy_{[k+1]})\|\pi_{k+1}(\cdot|\vy_{[k+1]})\big)+\frac{20d\lambda^{2}h^{2}}{1-\exp(-\frac{h}{5C_{\LSI}})} \\
&\leq\exp\Big(-\frac{\ell h}{5C_{\LSI}}\Big)\chi^{2}\big(\pi_{k+1}^{0}(\cdot|\vy_{[k+1]})\|\pi_{k+1}(\cdot|\vy_{[k+1]})\big)+140dC_{\LSI}\lambda^{2}h,
\end{align*}
where the last inequality used $1-\exp(-\frac{h}{5C_{\LSI}})\geq\frac{3h}{20C_{\LSI}}$ for $0<\frac{h}{5C_{\LSI}}\leq\frac{1}{4}$. Indeed, combining the condition $400dC_{\LSI}\lambda^{2}h\leq 1$ and Lemma~\ref{lemma:LSI:Lip} implies $4h\lambda\leq 1$, which deduces that $\frac{h}{5C_{\LSI}}\leq\frac{1}{20C_{\LSI}\lambda}\leq\frac{1}{4}$. This completes the proof.
\end{proof}

\section{Prior Error}
\label{appendix:prior}

\par In this section, we provide an error bound for prior mismatch in Lemma~\ref{lemma:appendix:convergence:decomposition}. Before proceeding, we introduce two auxiliary lemmas.
\begin{lemma}\label{appendix:prior:1}
Suppose Assumption~\ref{assumption:bounded} (i) holds. Then for any $\vx\in\bbR^{d}$, 
\begin{equation*}
|q_{k+1}(\vx|\vy_{[k]})-\what{q}_{k+1}(\vx|\vy_{[k]})|\leq 2B\|\pi_{k}(\cdot|\vy_{[k]})-\what{\pi}_{k}^{T}(\cdot|\vy_{[k]})\|_{\tv}.
\end{equation*}
\end{lemma}

\begin{proof}[Proof of Lemma~\ref{appendix:prior:1}]
It is straightforward that for any $\vx\in\bbR^{d}$, 
\begin{align*}
|q_{k+1}(\vx|\vy_{[k]})-\what{q}_{k+1}(\vx|\vy_{[k]})|
&\leq\int\rho_{k}(\vx|\vx_{k})|\pi_{k}(\vx_{k}|\vy_{[k]})-\what{\pi}_{k}^{T}(\vx_{k}|\vy_{[k]})|\d\vx_{k} \\
&\leq 2B\|\pi_{k}(\cdot|\vy_{[k]})-\what{\pi}_{k}^{T}(\cdot|\vy_{[k]})\|_{\tv},
\end{align*}
where the first inequality follows from the Jensen's inequality, and the second inequality invokes Assumption~\ref{assumption:bounded} (i) and the H{\"o}lder's inequality. This completes the proof.
\end{proof}

\begin{lemma}\label{appendix:prior:2}
Suppose Assumption~\ref{assumption:bounded} (i) holds. Then for any $\vx\in\bbR^{d}$, 
\begin{equation*}
\|\nabla_{\vx}q_{k+1}(\vx|\vy_{[k]})-\nabla_{\vx}\what{q}_{k+1}(\vx|\vy_{[k]})\|_{\infty}\leq 2B\|\pi_{k}(\cdot|\vy_{[k]})-\what{\pi}_{k}^{T}(\cdot|\vy_{[k]})\|_{\tv}.
\end{equation*}
\end{lemma}

\begin{proof}[Proof of Lemma~\ref{appendix:prior:2}]
It is straightforward that for any $\vx\in\bbR^{d}$ and any $i\in\{1,\ldots,d\}$, 
\begin{align*}
\Big|\frac{\partial q_{k+1}}{\partial x_{i}}(\vx|\vy_{[k]})-\frac{\partial \what{q}_{k+1}}{\partial x_{i}}(\vx|\vy_{[k]})\Big|
&\leq\int\Big|\frac{\partial \rho_{k}}{\partial x_{i}}(\vx|\vx_{k})\Big||\pi_{k}(\vx_{k}|\vy_{[k]})-\what{\pi}_{k}^{T}(\vx_{k}|\vy_{[k]})|\d\vx_{k} \\
&\leq 2B\|\pi_{k}(\cdot|\vy_{[k]})-\what{\pi}_{k}^{T}(\cdot|\vy_{[k]})\|_{\tv},
\end{align*}
where the first inequality follows from the Jensen's inequality, and the second inequality invokes Assumption~\ref{assumption:bounded} (i) and the H{\"o}lder's inequality. This completes the proof.
\end{proof}

\begin{proposition}\label{appendix:prior:3}
Suppose Assumptions~\ref{assumption:posterior:smooth},~\ref{assumption:LSI:posterior} and~\ref{assumption:bounded} hold. Then 
\begin{align*}
&\bbE^{\frac{1}{2}}\left[\|\nabla_{\vx}\log q_{k+1}(\mX_{k+1}|\vy_{[k]})-\nabla\log\what{q}_{k+1}(\mX_{k+1}|\vy_{[k]})\|_{2}^{4}\right] \\
&\leq CC_{\LSI}^{\frac{1}{4}}\|\pi_{k}(\cdot|\vy_{[k]})-\what{\pi}_{k}^{T}(\cdot|\vy_{[k]})\|_{\tv}^{2\gamma},
\end{align*}
where $C$ is a constant only depending on $d$ and $B$, and
\begin{equation*}
\gamma \coloneqq \frac{1+96BC_{\LSI}}{1+128BC_{\LSI}}.
\end{equation*}
\end{proposition}

\begin{proof}[Proof of Proposition~\ref{appendix:prior:3}]
For any $R\geq 1$, we have the decomposition using the truncation arguments, 
\begin{align}
&\bbE^{\frac{1}{2}}\left[\|\nabla_{\vx}\log q_{k+1}(\mX_{k+1}|\vy_{[k]})-\nabla\log\what{q}_{k+1}(\mX_{k+1}|\vy_{[k]})\|_{2}^{4}\right] \nonumber \\
&= \left(\int\|\nabla_{\vx}\log q_{k+1}(\vx|\vy_{[k]})-\nabla\log\what{q}_{k+1}(\vx|\vy_{[k]})\|_{2}^{4}\pi_{k+1}(\vx|\vy_{[k]})\d\vx\right)^{\frac{1}{2}} \nonumber \\
&\leq \left(\int\|\nabla_{\vx}\log q_{k+1}(\vx|\vy_{[k]})-\nabla\log\what{q}_{k+1}(\vx|\vy_{[k]})\|_{2}^{4}\mathbbm{1}\{\|\vx\|_{2}\leq R\}\pi_{k+1}(\vx|\vy_{[k]})\d\vx\right)^{\frac{1}{2}} \nonumber \\
&\quad +\left(\int\|\nabla_{\vx}\log q_{k+1}(\vx|\vy_{[k]})-\nabla\log\what{q}_{k+1}(\vx|\vy_{[k]})\|_{2}^{4}\mathbbm{1}\{\|\vx\|_{2}> R\}\pi_{k+1}(\vx|\vy_{[k]})\d\vx\right)^{\frac{1}{2}} \nonumber \\
&\eqqcolon E_{1}+E_{2}, \label{eq:proof:lemma:appendix:prior:3:1}
\end{align}
where the inequality is due to Jensen's inequality. We consider the summand $E_{1}$ in~\eqref{eq:proof:lemma:appendix:prior:3:1}. According to Assumption~\ref{assumption:bounded}, for any $\|\vx\|_{2}\leq R$, we have 
\begin{equation}\label{eq:proof:lemma:appendix:prior:3:2}
\Big|\Big(\frac{1}{\widehat{q}_{k+1}}\frac{\partial\widehat{q}_{k+1}}{\partial x_{i}}\Big)(\vx|\vy_{[k]})\Big|\leq \frac{B}{H}\exp\Big(\frac{R^{2}}{V^{2}}\Big),
\end{equation}
and 
\begin{equation}\label{eq:proof:lemma:appendix:prior:3:3}
\Big|\frac{1}{\widehat{q}_{k+1}}(\vx|\vy_{[k]})\Big|\leq \frac{1}{H}\exp\Big(\frac{R^{2}}{V^{2}}\Big).
\end{equation}
Here $H$ and $V$ are given in Proposition~\ref{proposition:lower:bound}. Therefore, for any $\|\vx\|_{2}\leq R$, it follows from the triangular inequality that 
\begin{align*}
&\Big|\Big(\frac{1}{q_{k+1}}\frac{\partial q_{k+1}}{\partial x_{i}}\Big)(\vx|\vy_{[k]})-\Big(\frac{1}{\what{q}_{k+1}}\frac{\partial\what{q}_{k+1}}{\partial x_{i}}\Big)(\vx|\vy_{[k]})\Big| \\
&\leq\Big|\frac{1}{q_{k+1}(\vx|\vy_{[k]})}\frac{\partial q_{k+1}}{\partial x_{i}}(\vx|\vy_{[k]})-\frac{1}{q_{k+1}(\vx|\vy_{[k]})}\frac{\partial\what{q}_{k+1}}{\partial x_{i}}(\vx|\vy_{[k]})\Big| \\
&\quad+\Big|\frac{1}{q_{k+1}(\vx|\vy_{[k]})}\frac{\partial\what{q}_{k+1}}{\partial x_{i}}(\vx|\vy_{[k]})-\frac{1}{\what{q}_{k+1}(\vx|\vy_{[k]})}\frac{\partial\what{q}_{k+1}}{\partial x_{i}}(\vx|\vy_{[k]})\Big| \\
&\leq\frac{1}{q_{k+1}(\vx|\vy_{[k]})}\Big|\frac{\partial q_{k+1}}{\partial x_{i}}(\vx|\vy_{[k]})-\frac{\partial\what{q}_{k+1}}{\partial x_{i}}(\vx|\vy_{[k]})\Big| \\
&\quad+\frac{1}{q_{k+1}(\vx|\vy_{[k]})}\Big|\Big(\frac{1}{\what{q}_{k+1}}\frac{\partial\what{q}_{k+1}}{\partial x_{i}}\Big)(\vx|\vy_{[k]})\Big||\what{q}_{k+1}(\vx|\vy_{[k]})-q_{k+1}(\vx|\vy_{[k]})| \\
&\leq \frac{B}{H}\exp\Big(\frac{R^{2}}{V^{2}}\Big)\Big|\frac{\partial q_{k+1}}{\partial x_{i}}(\vx|\vy_{[k]})-\frac{\partial\what{q}_{k+1}}{\partial x_{i}}(\vx|\vy_{[k]})\Big| \\
&\quad +\frac{B}{H^{2}}\exp\Big(\frac{2R^{2}}{V^{2}}\Big)|\what{q}_{k+1}(\vx|\vy_{[k]})-q_{k+1}(\vx|\vy_{[k]})| \\
&\leq \frac{2B}{H^{2}}\exp\Big(\frac{2R^{2}}{V^{2}}\Big)\|\pi_{k}(\cdot|\vy_{[k]})-\what{\pi}_{k}^{T}(\cdot|\vy_{[k]})\|_{\tv},
\end{align*}
where the second inequality is owing to~\eqref{eq:proof:lemma:appendix:prior:3:2} and~\eqref{eq:proof:lemma:appendix:prior:3:3}. As a consequence,
\begin{align}
E_{1}
&=\left(\int\|\nabla_{\vx}\log q_{k+1}(\vx|\vy_{[k]})-\nabla\log\what{q}_{k+1}(\vx|\vy_{[k]})\|_{2}^{4}\mathbbm{1}\{\|\vx\|_{2}\leq R\}\pi_{k+1}(\vx|\vy_{[k]})\d\vx\right)^{\frac{1}{2}} \nonumber \\
&\lesssim \frac{dB^{2}}{H^{4}}\exp\Big(\frac{4R^{2}}{V^{2}}\Big)\|\pi_{k}(\cdot|\vy_{[k]})-\what{\pi}_{k}^{T}(\cdot|\vy_{[k]})\|_{\tv}^{2}. \label{eq:proof:lemma:appendix:prior:3:4}
\end{align}
For the summand $E_{2}$ in~\eqref{eq:proof:lemma:appendix:prior:3:1}, 
\begin{align}
E_{2}
&=\left(\int\|\nabla_{\vx}\log q_{k+1}(\vx|\vy_{[k]})-\nabla\log\what{q}_{k+1}(\vx|\vy_{[k]})\|_{2}^{4}\mathbbm{1}\{\|\vx\|_{2}> R\}\pi_{k+1}(\vx|\vy_{[k]})\d\vx\right)^{\frac{1}{2}} \nonumber \\
&\lesssim B^{2}\left(\int(1+\|\vx\|_{2}^{4})\mathbbm{1}\{\|\vx\|_{2}> R\}\pi_{k+1}(\vx|\vy_{[k+1]})\d\vx\right)^{\frac{1}{2}} \nonumber \\
&\lesssim B^{2}\left(\int(1+\|\vx\|_{2}^{8})\pi_{k+1}(\vx|\vy_{[k+1]})\d\vx\right)^{\frac{1}{4}}\left(\int\mathbbm{1}\{\|\vx\|_{2}> R\}\pi_{k+1}(\vx|\vy_{[k+1]})\d\vx\right)^{\frac{1}{4}} \nonumber \\
&\lesssim B^{2}e^{\frac{d}{4}}C_{\LSI}\exp\Big(-\frac{R^2}{16C_{\LSI}}\Big), \label{eq:proof:lemma:appendix:prior:3:5}
\end{align}
where the first inequality holds from Assumption~\ref{assumption:bounded}, the second inequality invokes Cauchy-Schwarz inequality, and the last inequality is owing to Lemmas~\ref{lem:tail:bound} and~\ref{lem:k:moment}. Substituting~\eqref{eq:proof:lemma:appendix:prior:3:4} and~\eqref{eq:proof:lemma:appendix:prior:3:5} into~\eqref{eq:proof:lemma:appendix:prior:3:1} yields  
\begin{align*}
&\bbE^{\frac{1}{2}}\left[\|\nabla_{\vx}\log q_{k+1}(\mX_{k+1}|\vy_{[k]})-\nabla\log\what{q}_{k+1}(\mX_{k+1}|\vy_{[k]})\|_{2}^{4}\right] \\
&\leq \inf_{R\geq 1}\left\{\frac{dB^{2}}{H^{4}}\exp\Big(\frac{4R^{2}}{V^{2}}\Big)\|\pi_{k}(\cdot|\vy_{[k]})-\what{\pi}_{k}^{T}(\cdot|\vy_{[k]})\|_{\tv}^{2}+B^{2}e^{\frac{d}{4}}C_{\LSI}\exp\Big(-\frac{R^2}{16C_{\LSI}}\Big)\right\}.
\end{align*}
By setting 
\begin{equation*}
R^{2}=\frac{64V^{2}+C_{\LSI}}{V^{2}+64C_{\LSI}}\log\Big(\frac{H^{2}e^{\frac{d}{4}}C_{\LSI}}{d\|\pi_{k}(\cdot|\vy_{[k]})-\what{\pi}_{k}^{T}(\cdot|\vy_{[k]})\|_{\tv}^{2}}\Big),
\end{equation*}
we have 
\begin{align*}
&\bbE^{\frac{1}{2}}\left[\|\nabla_{\vx}\log q_{k+1}(\mX_{k+1}|\vy_{[k]})-\nabla\log\what{q}_{k+1}(\mX_{k+1}|\vy_{[k]})\|_{2}^{4}\right] \\
&\leq CB^{2}C_{\LSI}^{\frac{1}{4}}\|\pi_{k}(\cdot|\vy_{[k]})-\what{\pi}_{k}^{T}(\cdot|\vy_{[k]})\|_{\tv}^{2\gamma},
\end{align*}
where $C$ is a constant only depending on $d$ and $B$, and $\gamma \coloneqq \frac{V^{2}+48C_{\LSI}}{V^{2}+64C_{\LSI}}$. Substituting $H$ and $V$ in Proposition~\ref{proposition:lower:bound} completes the proof.
\end{proof}

\section{Score estimation error}
\label{appendix:score}

This section focuses on the error of the score matching in Lemma~\ref{lemma:appendix:convergence:decomposition}.

\begin{proposition}\label{proposition:sm:error}
Suppose Assumptions~\ref{assumption:posterior:smooth},~\ref{assumption:bounded},~\ref{assumption:LSI:posterior}, and~\ref{assumption:sm} hold. Then 
\begin{align*}
&\bbE^{\frac{1}{2}}\left[\|\nabla_{\vx}\log\widehat{q}_{k+1}(\mX_{k+1}|\vy_{[k]})-\what{\vs}_{k+1}(\mX_{k+1},\vy_{[k]})\|_{2}^{4}\right] \\
&\leq C^{\prime}C_{\LSI}(\kappa\Delta)^{\alpha}\log^{d+2}\Big(\frac{C_{\LSI}}{\kappa\Delta}\Big),
\end{align*}
where $C^{\prime}$ is a constant only depending on $d$, $H$ and $V$, and
\begin{equation*}
\alpha \coloneqq \frac{1}{2+8BC_{\LSI}}.
\end{equation*}
\end{proposition}

\begin{proof}[Proof of Proposition~\ref{proposition:sm:error}]
For any $R\geq 1$, we have the decomposition using the truncation arguments, 
\begin{align}
&\bbE^{\frac{1}{2}}\left[\|\nabla_{\vx}\log\widehat{q}_{k+1}(\mX_{k+1}|\vy_{[k]})-\what{\vs}_{k+1}(\mX_{k+1},\vy_{[k]})\|_{2}^{4}\right] \nonumber \\
&= \left(\int\|\nabla_{\vx}\log\widehat{q}_{k+1}(\vx|\vy_{[k]})-\what{\vs}_{k+1}(\vx,\vy_{[k]})\|_{2}^{4}\pi_{k+1}(\vx|\vy_{[k]})\d\vx\right)^{\frac{1}{2}} \nonumber \\
&\leq \left(\int\|\nabla_{\vx}\log\widehat{q}_{k+1}(\vx|\vy_{[k]})-\what{\vs}_{k+1}(\vx,\vy_{[k]})\|_{2}^{4}\mathbbm{1}\{\|\vx\|_{2}\leq R\}\pi_{k+1}(\vx|\vy_{[k]})\d\vx\right)^{\frac{1}{2}} \nonumber \\
&\quad +\left(\int\|\nabla_{\vx}\log\widehat{q}_{k+1}(\vx|\vy_{[k]})-\what{\vs}_{k+1}(\vx,\vy_{[k]})\|_{2}^{4}\mathbbm{1}\{\|\vx\|_{2}> R\}\pi_{k+1}(\vx|\vy_{[k]})\d\vx\right)^{\frac{1}{2}} \nonumber \\
&\eqqcolon E_{1}+E_{2}, \label{eq:proof:lemma:appendix:score:1}
\end{align}
where the inequality is due to Jensen's inequality. For the summand $E_{1}$ in~\eqref{eq:proof:lemma:appendix:score:1}, it follows from Assumption~\ref{assumption:bounded} that 
\begin{align}
E_{1}
&= \left(\int\|\nabla_{\vx}\log\widehat{q}_{k+1}(\vx|\vy_{[k]})-\what{\vs}_{k+1}(\vx,\vy_{[k]})\|_{2}^{4}\mathbbm{1}\{\|\vx\|_{2}\leq R\}\pi_{k+1}(\vx|\vy_{[k]})\d\vx\right)^{\frac{1}{2}} \nonumber \\
&\lesssim B^{\frac{3}{2}}R^{\frac{3}{2}}\left(\int\|\nabla_{\vx}\log\widehat{q}_{k+1}(\vx|\vy_{[k]})-\what{\vs}_{k+1}(\vx,\vy_{[k]})\|_{2}\mathbbm{1}\{\|\vx\|_{2}\leq R\}\pi_{k+1}(\vx|\vy_{[k]})\d\vx\right)^{\frac{1}{2}}. \label{eq:proof:lemma:appendix:score:2}
\end{align}
Under Assumptions~\ref{assumption:bounded} and~\ref{assumption:sm}, using Proposition~\ref{proposition:lower:bound} and Cauchy-Schwarz inequality implies 
\begin{align*}
&\int\|\nabla_{\vx}\log\widehat{q}_{k+1}(\vx|\vy_{[k]})-\what{\vs}_{k+1}(\vx,\vy_{[k]})\|_{2}\pi_{k+1}(\vx|\vy_{[k+1]})\mathbbm{1}\{\|\vx\|_{2}\leq R\}\d\vx \\
&\leq \kappa\int\|\nabla_{\vx}\log\widehat{q}_{k+1}(\vx|\vy_{[k]})-\what{\vs}_{k+1}(\vx,\vy_{[k]})\|_{2}\frac{q_{k+1}(\vx|\vy_{[k]})}{\widehat{q}_{k+1}(\vx|\vy_{[k]})}\widehat{q}_{k+1}(\vx|\vy_{[k]})\mathbbm{1}\{\|\vx\|_{2}\leq R\}\d\vx \\
&\leq \kappa\left(\int\|\nabla_{\vx}\log\widehat{q}_{k+1}(\vx|\vy_{[k]})-\what{\vs}_{k+1}(\vx,\vy_{[k]})\|_{2}^{2}\widehat{q}_{k+1}(\vx|\vy_{[k]})\d\vx\right)^{\frac{1}{2}} \\
&\quad \times\left(\int\Big(\frac{q_{k+1}(\vx|\vy_{[k]})}{\widehat{q}_{k+1}(\vx|\vy_{[k]})}\Big)^{2}\widehat{q}_{k+1}(\vx|\vy_{[k]})\mathbbm{1}\{\|\vx\|_{2}\leq R\}\d\vx\right)^{\frac{1}{2}} \\
&\leq \kappa\Delta\left(\int\Big(\frac{q_{k+1}(\vx|\vy_{[k]})}{\widehat{q}_{k+1}(\vx|\vy_{[k]})}\Big)^{2}\widehat{q}_{k+1}(\vx|\vy_{[k]})\mathbbm{1}\{\|\vx\|_{2}\leq R\}\d\vx\right)^{\frac{1}{2}} \\
&\leq B\kappa\Delta\left(\int\frac{\mathbbm{1}\{\|\vx\|_{2}\leq R\}}{\widehat{q}_{k+1}(\vx|\vy_{[k]})}\d\vx\right)^{\frac{1}{2}} \\
&\leq \sqrt{\frac{A_{d}R^{d}}{H}}B\kappa\exp\Big(\frac{R^{2}}{2V^{2}}\Big)\Delta,
\end{align*}
where $A_{d}$ is the volume of the $d$-dimensional unit ball. Substituting this into~\eqref{eq:proof:lemma:appendix:score:2} yields 
\begin{equation}\label{eq:proof:lemma:appendix:score:3}
E_{1} \lesssim B^{2}R^{\frac{3}{2}+\frac{d}{4}}\Big(\frac{A_{d}}{H}\Big)^{\frac{1}{4}}\kappa^{\frac{1}{2}}\exp\Big(\frac{R^{2}}{4V^{2}}\Big)\Delta^{\frac{1}{2}}.
\end{equation}
For the summand $E_{2}$ in~\eqref{eq:proof:lemma:appendix:score:1}, by similar arguments as~\eqref{eq:proof:lemma:appendix:prior:3:5}, we have  
\begin{equation}\label{eq:proof:lemma:appendix:score:4}
E_{2} \lesssim B^{2}e^{\frac{d}{4}}C_{\LSI}\exp\Big(-\frac{R^2}{16C_{\LSI}}\Big).
\end{equation}
By setting 
\begin{equation*}
R^{2}=\frac{16V^{2}C_{\LSI}}{V^{2}+4C_{\LSI}}\log\Big(\frac{e^{d/4}H^{1/4}C_{\LSI}}{A_{d}^{1/4}\sqrt{\kappa\Delta}}\Big)
\end{equation*}
in~\eqref{eq:proof:lemma:appendix:score:3} and~\eqref{eq:proof:lemma:appendix:score:4}, and substituting them into~\eqref{eq:proof:lemma:appendix:score:1} completes the proof.
\end{proof}

\section{Convergence Analysis for the Initial Step}
\label{appendix:init}

\par In this section, we consider the assimilation in the first time step. The Langevin diffusion is given as 
\begin{equation}\label{eq:LD:init}
\d\mZ_{t}=\nabla_{\vx}\log\pi_{1}(\mZ_{t}|\vy_{1})\dt+\sqrt{2}\d\mB_{t}, \quad \mZ_{0}\sim\pi_{1}^{0}(\cdot|\vy_{1}),~t\geq 0.
\end{equation}
Denote by $\pi_{1}^{t}$ the law of $\mZ_{t}$ for each $t\geq 0$. The Langevin Monte Carlo is defined as the Euler-Maruyama discretization of the Langevin diffusion. The interpolation of the Langevin Monte Carlo is given as, for each $0\leq\ell\leq K-1$,
\begin{equation}\label{eq:LMC:init}
\d\bar{\mZ}_{t}=\nabla_{\vx}\log\pi_{1}(\bar{\mZ}_{\ell h}|\vy_{[k+1]})\dt+\sqrt{2}\d\mB_{t}, \quad \ell h\leq t\leq(\ell+1)h,
\end{equation}
where $\bar{\mZ}_{0}\sim\pi_{1}^{0}(\cdot|\vy_{1})$. Denote by $\bar{\pi}_{1}^{t}$ the law of $\bar{\mZ}_{t}$ for each $0\leq t\leq Kh=T$. We next introduce the interpolation of the score-based Langevin Monte Carlo
\begin{equation}\label{eq:SLMC:init}
\d\what{\mZ}_{t}=\what{\vb}_{1}(\what{\mZ}_{\ell h}|\vy_{[k+1]})\dt+\sqrt{2}\d\mB_{t}, \quad \ell h\leq t\leq(\ell+1)h,
\end{equation}
where $\what{\mZ}_{0}\sim\pi_{1}^{0}(\cdot|\vy_{1})$, and the estimator of posterior score function is given as 
\begin{equation*}
\what{\vb}_{1}(\vx|\vy_{1})=\nabla_{\vx}\log g_{1}(\vy_{1}|\vx)+\what{\vs}_{1}(\vx).
\end{equation*}
Here $\what{\vs}_{1}$ is an estimator of $\nabla_{\vx}\log q_{1}$. Denote by $\what{\pi}_{1}^{t}$ the law of $\what{\mZ}_{t}$ for each $0\leq t\leq Kh=T$. 

\par By the same arguments as Theorem~\ref{theorem:section:convergence}, we have the error bounds for the initial time step.

\begin{lemma}
\label{lemma:appendix:init}
Suppose Assumptions~\ref{assumption:posterior:smooth},~\ref{assumption:LSI:posterior},~\ref{assumption:bounded}, and~\ref{assumption:sm} hold. Then for each $k\in\bbN$ and each terminal time $T=Kh$, 
\begin{align*}
&(\varepsilon_{\tv}^{1})^{2} \\
&\lesssim\underbrace{\exp\Big(-\frac{T}{5C_{\LSI}}\Big)\eta_{\chi}^{2}}_{\text{convergence of Langevin dynamics}}+\underbrace{dC_{\LSI}\lambda^{2}h}_{\text{discretization error}}+\underbrace{C^{\prime}(C_{\LSI}\eta_{\chi}+T)C_{\LSI}(\kappa\Delta)^{\alpha}\log^{d+2}\Big(\frac{C_{\LSI}}{\kappa\Delta}\Big)}_{\text{score estimation error}},
\end{align*}
where $C$ and $C^{\prime}$ are constants only depending on $d$ and $B$, and $\alpha \coloneqq \frac{1}{2+16BC_{\LSI}}$. Here the step size $h$ and the initial distribution $\pi_{1}^{0}(\cdot|\vy_{1})$ satisfies
\begin{equation*}
h\lesssim\frac{1}{dC_{\LSI}\lambda^{2}}, \quad 
\chi^{2}\big(\pi_{1}^{0}(\cdot|\vy_{1})\|\pi_{1}(\cdot|\vy_{1})\big)\leq\eta_{\chi}^{2}.
\end{equation*}
\end{lemma}

\section{Error Bounds in Wasserstein Distance}
\label{appendix:wasserstein}

\begin{lemma}\label{lemma:w1:tv}
Suppose Assumption~\ref{assumption:LSI:posterior} holds. Let $R_{0}\geq 1$. Then for all time step $k\in\mathbb{N}$, we have 
\begin{equation*}
W_{1}^{2}((\Pi_{R_{0}})_{\sharp}\what{\pi}_{k+1}(\cdot|\vy_{[k+1]}),\pi_{k+1}(\cdot|\vy_{[k+1]})) \lesssim C_{\LSI}(\varepsilon_{\tv}^{k+1})^{2}\log\Bigg(\frac{e^{\frac{d}{2}}C_{\LSI}}{(\varepsilon_{\tv}^{k+1})^{2}}\Bigg),
\end{equation*}
where the truncation radius $R_{0}$ is given as 
\begin{equation*}
R_{0}^{2}=C_{\LSI}\log\Bigg(\frac{e^{\frac{d}{2}}C_{\LSI}}{(\varepsilon_{\tv}^{k+1})^{2}}\Bigg).
\end{equation*}
\end{lemma}

\begin{proof}[Proof of Lemma~\ref{lemma:w1:tv}]
According to the triangular inequality, we have 
\begin{align}
&W_{1}((\Pi_{R_{0}})_{\sharp}\what{\pi}_{k+1}(\cdot|\vy_{[k+1]}),\pi_{k+1}(\cdot|\vy_{[k+1]})) \nonumber \\
&\leq \underbrace{W_{1}((\Pi_{R_{0}})_{\sharp}\what{\pi}_{k+1}(\cdot|\vy_{[k+1]}),(\Pi_{R_{0}})_{\sharp}\pi_{k+1}(\cdot|\vy_{[k+1]}))}_{E_{1}} \nonumber \\
&\quad +\underbrace{W_{1}((\Pi_{R_{0}})_{\sharp}\pi_{k+1}(\cdot|\vy_{[k+1]}),\pi_{k+1}(\cdot|\vy_{[k+1]}))}_{E_{2}}. \label{eq:lemma:w1:tv:1}
\end{align}
For the summand $E_{1}$ in~\eqref{eq:lemma:w1:tv:1}, since both $(\Pi_{R_{0}})_{\sharp}\what{\pi}_{k+1}(\cdot|\vy_{[k+1]})$ and $(\Pi_{R_{0}})_{\sharp}\pi_{k+1}(\cdot|\vy_{[k+1]})$ have compact support, we can bound the Wasserstein distance by TV distance. Specifically, it follows from~\cite[Theorem 6.15]{Villani2009Optimal} that 
\begin{align}
&W_{1}((\Pi_{R_{0}})_{\sharp}\what{\pi}_{k+1}(\cdot|\vy_{[k+1]}),(\Pi_{R_{0}})_{\sharp}\pi_{k+1}(\cdot|\vy_{[k+1]})) \nonumber \\
&\leq R_{0}\|(\Pi_{R_{0}})_{\sharp}\what{\pi}_{k+1}(\cdot|\vy_{[k+1]})-(\Pi_{R_{0}})_{\sharp}\pi_{k+1}(\cdot|\vy_{[k+1]})\|_{\tv} \nonumber \\
&\leq R_{0}\|\what{\pi}_{k+1}(\cdot|\vy_{[k+1]})-\pi_{k+1}(\cdot|\vy_{[k+1]})\|_{\tv} \leq R_{0}\varepsilon_{\tv}^{k+1}, \label{eq:lemma:w1:tv:2}
\end{align} 
where the second inequality holds from the data processing inequality. For the summand $E_{2}$ in~\eqref{eq:lemma:w1:tv:1}, we construct a coupling of $((\Pi_{R_{0}})_{\sharp}\pi_{k+1}(\cdot|\vy_{[k+1]}),\pi_{k+1}(\cdot|\vy_{[k+1]}))$ as $(\Pi_{R_{0}}\mX_{k+1},\mX_{k+1})$ with $\mX_{k+1}\sim\pi_{k+1}(\cdot|\vy_{[k+1]})$. Then we have 
\begin{align}
&W_{1}^{2}((\Pi_{R_{0}})_{\sharp}\pi_{k+1}(\cdot|\vy_{[k+1]}),\pi_{k+1}(\cdot|\vy_{[k+1]})) \nonumber \\
&\leq \mathbb{E}^{2}\big[\|\Pi_{R_{0}}\mX_{k+1}-\mX_{k+1}\|_{2}\big] \nonumber \\
&\leq \mathbb{E}\big[\|\Pi_{R_{0}}\mX_{k+1}-\mX_{k+1}\|_{2}^{2}\big] \nonumber \\
&= \int\|\vx\mathbbm{1}\{\|\vx\|_{2}\leq R_{0}\}-\vx\|_{2}^{2}\pi_{k+1}(\vx|\vy_{[k+1]})\d\vx \nonumber \\
&= \int\|\vx\|_{2}^{2}\mathbbm{1}\{\|\vx\|_{2}>R_{0}\}\pi_{k+1}(\vx|\vy_{[k+1]})\d\vx \nonumber \\
&\leq \Big(\int\|\vx\|_{2}^{4}\pi_{k+1}(\vx|\vy_{[k+1]})\d\vx\Big)^{\frac{1}{2}}\Big(\int\mathbbm{1}\{\|\vx\|_{2}>R_{0}\}\pi_{k+1}(\vx|\vy_{[k+1]})\d\vx\Big)^{\frac{1}{2}} \nonumber \\
&\lesssim \exp\Big(\frac{d}{2}\Big)C_{\LSI}\exp\Big(-\frac{R_{0}^{2}}{8C_{\LSI}}\Big), \label{eq:lemma:w1:tv:3}
\end{align}
where the first inequality holds from the definition of Wasserstein distance, the second inequality invokes Jensen's inequality, the third inequality is owing to Cauchy-Schwarz inequality, and the last inequality is due to Lemmas~\ref{lem:tail:bound} and~\ref{lem:k:moment}. Substituting~\eqref{eq:lemma:w1:tv:2} and~\eqref{eq:lemma:w1:tv:3} into~\eqref{eq:lemma:w1:tv:3} implies 
\begin{equation*}
W_{1}^{2}((\Pi_{R_{0}})_{\sharp}\what{\pi}_{k+1}(\cdot|\vy_{[k+1]}),\pi_{k+1}(\cdot|\vy_{[k+1]})) \lesssim R_{0}^{2}(\varepsilon_{\tv}^{k+1})^{2}+\exp\Big(\frac{d}{2}\Big)C_{\LSI}\exp\Big(-\frac{R_{0}^{2}}{8C_{\LSI}}\Big).
\end{equation*}
Setting
\begin{equation*}
R_{0}^{2}=C_{\LSI}\log\Bigg(\frac{e^{\frac{d}{2}}C_{\LSI}}{(\varepsilon_{\tv}^{k+1})^{2}}\Bigg)
\end{equation*}
completes the proof.
\end{proof}


\section{Auxiliary Definitions and Lemmas}
\label{appendix:auxiliary}

\begin{lemma}[Second moment bound under LSI]
\label{lem:second:moment}
Suppose Assumption~\ref{assumption:LSI:posterior} holds. Then for any $k\in\mathbb{N}$,
\begin{equation*}
\int\|\vx\|_{2}^{2}\pi_{k+1}(\vx|\vy_{[k+1]})\d\vx \leq dC_{\LSI}.
\end{equation*}
\end{lemma}

\begin{proof}[Proof of Lemma~\ref{lem:second:moment}]
Since the distribution $\pi_{k+1}(\cdot|\vy_{[k+1]})$ satisfies log-Sobolev inequality with constant $C_{\LSI}$, it strictly implies the Poincar\'e inequality with the same constant $C_{\LSI}$. This means, for any continuously differentiable function $g: \mathbb{R}^d \to \mathbb{R}$, 
\begin{equation*}
\int g^{2}(\vx)\pi_{k+1}(\vx|\vy_{[k+1]})\d\vx \leq C_{\LSI}\int \|\nabla g(\vx)\|_{2}^{2}\pi_{k+1}(\vx|\vy_{[k+1]})\d\vx.
\end{equation*}
Consider the coordinate projection functions $g_{i}(\vx)=x_{i}$ for $i\in\{1,\dots,d\}$, then $\|\nabla g_{i}(\vx)\|_{2}^{2}=1$ for any $\vx\in\mathbb{R}^d$. Applying the Poincar\'e inequality to each $g_{i}$ concludes the proof.
\end{proof}

\begin{lemma}[Tail concentration under LSI]
\label{lem:tail:bound}
Suppose Assumption~\ref{assumption:LSI:posterior} holds. Then, for any $k\in\mathbb{N}$ and for any radius $R>0$, 
\begin{equation*}
\int\mathbbm{1}\{\|\vx\|_{2}>R\}\pi_{k+1}(\vx|\vy_{[k+1]})\d\vx \leq \exp\Big(\frac{d}{2}\Big)\exp\Big(-\frac{R^2}{4C_{\LSI}}\Big).
\end{equation*}
\end{lemma}

\begin{proof}[Proof of Lemma~\ref{lem:tail:bound}]
Applying Herbst's argument to a $1$-Lipschitz function $f(\vx)=\|\vx\|_2$ gives:
\begin{equation*}
\int\mathbbm{1}\{\|\vx\|_{2}-\mathbb{E}\|\vx\|_{2}>t\}\pi_{k+1}(\vx|\vy_{[k+1]})\d\vx \leq \exp\left(-\frac{t^2}{2C}\right).
\end{equation*}
Using Jensen's inequality and Lemma~\ref{lem:second:moment}, we have
\begin{equation*}
\int\|\vx\|_{2}\pi_{k+1}(\vx|\vy_{[k+1]})\d\vx \leq \sqrt{dC_{\LSI}}.
\end{equation*}
For any $R > \sqrt{dC_{\LSI}}$, setting $t=R-\int\|\vx\|_{2}\pi_{k+1}(\vx|\vy_{[k+1]})\d\vx$ yields
\begin{align*}
&\int\mathbbm{1}\{\|\vx\|_{2}>R\}\pi_{k+1}(\vx|\vy_{[k+1]})\d\vx \\
&\leq \exp\left(-\frac{1}{2C_{\LSI}}\left(R-\int\|\vx\|_{2}\pi_{k+1}(\vx|\vy_{[k+1]})\d\vx\right)^{2}\right) \\
&\leq \exp\left(-\frac{(R-\sqrt{dC_{\LSI}})^{2}}{2C_{\LSI}}\right).
\end{align*}
Note that $2(a-b)^{2}\geq a^{2}-2b^{2}$. Setting $a = R$ and $b = \sqrt{dC_{\LSI}}$ implies 
\begin{equation*}
(R - \sqrt{dC_{\LSI}})^2 \geq \frac{1}{2}R^{2}-dC_{\LSI}.
\end{equation*}
Therefore,
\begin{align*}
&\int\mathbbm{1}\{\|\vx\|_{2}>R\}\pi_{k+1}(\vx|\vy_{[k+1]})\d\vx \\
&\leq \exp\left(-\frac{R^2}{4C_{\LSI}}+\frac{d}{2}\right)=\exp\Big(\frac{d}{2}\Big)\exp\Big(-\frac{R^2}{4C_{\LSI}}\Big).
\end{align*}
If $R\leq\sqrt{dC_{\LSI}}$, the desired inequality holds trivially. This completes the proof.
\end{proof}

\begin{lemma}[The $p$-th moment bound under LSI]
\label{lem:k:moment}
Suppose Assumption~\ref{assumption:LSI:posterior} holds. Then for any $k\in\mathbb{N}$ and $p\in\mathbb{N}_{+}$,
\begin{equation*}
\int\|\vx\|_{2}^{p}\pi_{k+1}(\vx|\vy_{[k+1]})\d\vx \leq c_{p}e^{\frac{d}{2}}C_{\LSI}^{\frac{p}{2}},
\end{equation*}
where $c_{p}$ is a constant only depending on $p$.
\end{lemma}

\begin{proof}[Proof of Lemma~\ref{lem:k:moment}]
For non-negative random variable, we can represent the expectation in terms of tail probabiltiy as 
\begin{align*}
&\int\|\vx\|_{2}^{p}\pi_{k+1}(\vx|\vy_{[k+1]})\d\vx \\
&= \int_{0}^{\infty}pt^{p-1}\left(\int\mathbbm{1}\{\|\vx\|_{2}>t\}\pi_{k+1}(\vx|\vy_{[k+1]})\d\vx\right)\d t \\
&\leq \int_{0}^{\infty}pt^{p-1}\exp\Big(\frac{d}{2}\Big)\exp\Big(-\frac{t^2}{4C_{\LSI}}\Big)\d t \\
&\leq c_{p}e^{\frac{d}{2}}C_{\LSI}^{\frac{p}{2}},
\end{align*}
where the first inequality holds from Lemma~\ref{lem:tail:bound}.
\end{proof}

\begin{definition}[Kullback-Leibler divergence]
The KL-divergence between two distributions $\mu$ and $\pi$ is defined as
\begin{equation*}
\kl(\mu\|\pi)=\int\mu(\vx)\log\frac{\mu(\vx)}{\pi(\vx)}\d\vx.
\end{equation*}
\end{definition}

\par We then show the relationships between them.
\begin{lemma}\label{lemma:appendix:auxiliary:tv:chi}
For two distributions $\mu$ and $\pi$, 
\begin{equation*}
\|\mu-\pi\|_{\tv}^{2} \leq \frac{1}{4}\chi^{2}(\mu\|\pi).
\end{equation*}
\end{lemma}

\begin{proof}[Proof of Lemma~\ref{lemma:appendix:auxiliary:tv:chi}]
It is straightforward that 
\begin{align*}
\|\mu-\pi\|_{\tv}^{2}
&=\frac{1}{4}\Big(\int|\mu(\vx)-\pi(\vx)|\d\vx\Big)^{2} \\
&\leq\frac{1}{4}\Big(\int\frac{(\mu(\vx)-\pi(\vx))^{2}}{\pi(\vx)}\d\vx\Big)\Big(\int \pi(\vx)\d\vx\Big) \\
&=\frac{1}{4}\int\Big(\frac{\mu(\vx)}{\pi(\vx)}-1\Big)^{2}\pi(\vx)\d\vx=\frac{1}{4}\chi^{2}(\mu\|\pi),
\end{align*}
where the inequality follows from the Cauchy-Schwarz inequality. The proof is complete.
\end{proof}

\par The proof of Lemmas~\ref{lemma:tv:kl} and~\ref{lemma:kl:chi} can be found in~\cite[Lemmas 2.5 and 2.7]{Tsybakov2009Introduction}.
\begin{lemma}[Pinsker's inequality]
\label{lemma:tv:kl}
For two distributions $\mu$ and $\pi$, 
\begin{equation*}
\|\mu-\pi\|_{\tv}^{2}\leq\frac{1}{2}\kl(\mu\|\pi).
\end{equation*}
\end{lemma}

\begin{lemma}
\label{lemma:kl:chi}
For two distributions $\mu$ and $\pi$, 
\begin{equation*}
\kl(\mu\|\pi)\leq\log(1+\chi^{2}(\mu\|\pi))\leq\chi^{2}(\mu\|\pi).
\end{equation*}
\end{lemma}

\begin{proof}[Proof of Lemma~\ref{lemma:kl:chi}]
It is straightforward that 
\begin{align*}
\kl(\mu\|\pi)
&=\int \mu(\vx)\log\frac{\mu(\vx)}{\pi(\vx)}\d\vx=\bbE_{\mX\sim\mu}\Big[\log\frac{\mu(\mX)}{\pi(\mX)}\Big] \\
&\leq\log\bbE_{\mX\sim\mu}\Big[\frac{\mu(\mX)}{\pi(\mX)}\Big]=\log\int\Big(\frac{\mu(\vx)}{\pi(\vx)}\Big)^{2}\pi(\vx)\d\vx \\
&=\log(1+\chi^{2}(\mu\|\pi))\leq\chi^{2}(\mu\|\pi),
\end{align*}
where the inequality follows from the Jensen's inequality. The proof is complete.
\end{proof}

\begin{definition}[Poincar{\'e} inequality]
A distributoin $\pi$ satisfies a Poincar{\'e} inequality with constant $C_{\PI}$, that is, for each function $f\in C_{0}^{\infty}(\bbR^{d})$,
\begin{equation*}
\var(f)\leq C_{\PI}\bbE\big[\|\nabla f\|_{2}^{2}\big],
\end{equation*}
where the expectation and variance are taken with respect to the distribution $\pi$.
\end{definition}

\par Notice that the log-Sobolev inequality implies a Poincar{\'e} inequality with the same constant. Thus~\cite[Lemma E.5]{Lee2022Convergence} gives the following lemma.

\begin{lemma}\label{lemma:LSI:Lip}
Let $\pi$ be a distribution such that $\log\pi$ is $C^{1}$ and $\lambda$-smooth. Further, $\pi$ satisfies the log-Sobolev inequality with constant $C_{\LSI}$. Then $\lambda C_{\LSI}\geq 1$. 
\end{lemma}

\begin{lemma}
\label{lemma:Dirichlet:energy}
For two distributions $\mu$ and $\pi$, it holds that 
\begin{equation*}
\int\Delta\mu(\vx)\frac{\mu(\vx)}{\pi(\vx)}\d\vx+\int\mu(\vx)\nabla\log\pi(\vx)\cdot\nabla\frac{\mu(\vx)}{\pi(\vx)}\d\vx=\bbE_{\pi}\Big[\big\|\nabla\frac{\mu}{\pi}\big\|_{2}^{2}\Big].
\end{equation*}
\end{lemma}

\begin{proof}[Proof of Lemma~\ref{lemma:Dirichlet:energy}]
It is straightforward that 
\begin{align*}
&-\int\Delta\mu(\vx)\frac{\mu(\vx)}{\pi(\vx)}\d\vx-\int\mu(\vx)\nabla\log\pi(\vx)\cdot\nabla\frac{\mu(\vx)}{\pi(\vx)}\d\vx \\
&=\int\nabla\mu(\vx)\cdot\nabla\frac{\mu(\vx)}{\pi(\vx)}\d\vx-\int\frac{\mu(\vx)}{\pi(\vx)}\nabla\pi(\vx)\cdot\nabla\frac{\mu(\vx)}{\pi(\vx)}\d\vx \\
&=\int\pi(\vx)\Big(\frac{\nabla\mu(\vx)}{\pi(\vx)}-\frac{\mu(\vx)\nabla\pi(\vx)}{\pi(\vx)^{2}}\Big)\cdot\nabla\frac{\mu(\vx)}{\pi(\vx)}\d\vx \\
&=\int\big\|\nabla\frac{\mu(\vx)}{\pi(\vx)}\big\|_{2}^{2}\pi(\vx)\d\vx=\bbE_{\pi}\Big[\big\|\nabla\frac{\mu}{\pi}\big\|_{2}^{2}\Big],
\end{align*}
where the first equality holds from Green's formula. This completes the proof.
\end{proof}

\begin{lemma}
\label{lemma:LSI:chi:energy}
Let $\pi$ be a distribution satisfies the log-Sobolev inequality with constant $C_{\LSI}$. Then for each distribution $\mu$, it holds that 
\begin{equation*}
\frac{1}{2C_{\LSI}}\chi^{2}(\mu\|\pi)\leq\bbE_{\pi}\Big[\big\|\nabla\frac{\mu}{\pi}\big\|_{2}^{2}\Big].
\end{equation*}
\end{lemma}

\begin{proof}[Proof of Lemma~\ref{lemma:LSI:chi:energy}]
Recall the log-Sobolev inequality
\begin{equation}\label{eq:proof:lemma:chi2:fisher:1}
\ent_{\pi}(f^{2})\leq 2C_{\LSI}\bbE_{\pi}\big[\|\nabla f\|_{2}^{2}\big].
\end{equation}
Substituting $f(\vx)=(\mu(\vx)/\pi(\vx))^{q/2}$ into the left-hand side of~\eqref{eq:proof:lemma:kl:fisher:1} deduces 
\begin{align*}
\ent_{\pi}\Big(\frac{\mu^{q}}{\pi^{q}}\Big)
&=q\int\frac{\mu(\vx)^{q}}{\pi(\vx)^{q}}\log\frac{\mu(\vx)}{\pi(\vx)}\pi(\vx)\d\vx-\int\frac{\mu(\vx)^{q}}{\pi(\vx)^{q}}\pi(\vx)\d\vx\log\int\frac{\mu(\vx)^{q}}{\pi(\vx)^{q}}\pi(\vx)\d\vx \\
&=q\partial_{q}\int\frac{\mu(\vx)^{q}}{\pi(\vx)^{q}}\pi(\vx)\d\vx-\int\frac{\mu(\vx)^{q}}{\pi(\vx)^{q}}\pi(\vx)\d\vx\log\int\frac{\mu(\vx)^{q}}{\pi(\vx)^{q}}\pi(\vx)\d\vx,
\end{align*}
where the last equality used the chain rule. As a consequence, 
\begin{align*}
&\Big(\int\frac{\mu(\vx)^{q}}{\pi(\vx)^{q}}\pi(\vx)\d\vx\Big)^{-1}\ent_{\pi}\Big(\frac{\mu^{q}}{\pi^{q}}\Big) \\
&=q\partial_{q}\log\int\frac{\mu(\vx)^{q}}{\pi(\vx)^{q}}\pi(\vx)\d\vx-\log\int\frac{\mu(\vx)^{q}}{\pi(\vx)^{q}}\pi(\vx)\d\vx \\
&=q\partial_{q}\Big\{(q-1)\Big(\frac{1}{q-1}\log\int\frac{\mu(\vx)^{q}}{\pi(\vx)^{q}}\pi(\vx)\d\vx\Big)\Big\}-\log\int\frac{\mu(\vx)^{q}}{\pi(\vx)^{q}}\pi(\vx)\d\vx \\
&=\frac{q}{q-1}\log\int\frac{\mu(\vx)^{q}}{\pi(\vx)^{q}}\pi(\vx)\d\vx-\log\int\frac{\mu(\vx)^{q}}{\pi(\vx)^{q}}\pi(\vx)\d\vx \\
&\quad+q(q-1)\partial_{q}\Big(\frac{1}{q-1}\log\int\frac{\mu(\vx)^{q}}{\pi(\vx)^{q}}\pi(\vx)\d\vx\Big) \\
&\geq\frac{1}{q-1}\log\int\frac{\mu(\vx)^{q}}{\pi(\vx)^{q}}\pi(\vx)\d\vx,
\end{align*}
where the inequality invokes the fact that the R{\'e}nyi divergence is monotonic in the order $q$~\cite[Lemma 11]{Vempala2019Rapid}. Hence by setting $q=2$, 
\begin{align}
\ent_{\pi}\Big(\frac{\mu^{2}}{\pi^{2}}\Big)
&\geq\Big(\int\frac{\mu(\vx)^{2}}{\pi(\vx)^{2}}\pi(\vx)\d\vx\Big)\log\int\frac{\mu(\vx)^{2}}{\pi(\vx)^{2}}\pi(\vx)\d\vx \nonumber \\
&\geq\big(\chi^{2}(\mu\|\pi)+1\big)\log\big(\chi^{2}(\mu\|\pi)+1\big)\geq\chi^{2}(\mu\|\pi), \label{eq:proof:lemma:chi2:fisher:2}
\end{align}
where the last inequality holds from $(x+1)\log(x+1)\geq x$ for each $x\geq 0$. Combining~\eqref{eq:proof:lemma:chi2:fisher:1} and~\eqref{eq:proof:lemma:chi2:fisher:2} completes the proof.
\end{proof}

\par We next introduce Donsker-Varadhan variational principle~\cite[Theorem 5.4]{Rassoul2015Course}.
\begin{lemma}\label{lemma:DonskerVaradhan}
Let $\mu$ and $\pi$ be two distributions. Then for each function $\phi:\bbR^{d}\rightarrow\bbR$, 
\begin{equation*}
\bbE_{\mu}[\phi]\leq\kl(\mu\|\pi)+\log\bbE_{\pi}[\exp(\phi)].
\end{equation*}
\end{lemma}

\par The following lemma provides the Chernoff bound for $\chi^{2}$-distribution, which can be found in~\cite[Example 2.8]{Wainwright2019high} and~\cite[Example 4.1.13]{Duchi2024Information}.
\begin{lemma}\label{lemma:chernoff:chi:square}
Let Let $\mX=(X_{1},\ldots,X_{d})$ be a vector of independent Gaussian random variables with zero mean and $\sigma^{2}$-variance. Then 
\begin{equation*}
\log\bbE\big[\exp\{s(\|\mX\|_{2}^{2}-\bbE[\|\mX\|_{2}^{2}])\}\big]\leq 2ds\sigma^{2}.
\end{equation*}
\end{lemma}

\begin{proof}[Proof of Lemma~\ref{lemma:chernoff:chi:square}]
Before proceeding, we consider the Chernoff bound for the $Z^{2}$ with $Z\sim\mathcal{N}(0,1)$. For $4\lambda<1$, we have 
\begin{align}
\bbE[\exp(\lambda(Z^{2}-\bbE[Z^{2}]))]
&=\frac{1}{\sqrt{2\pi}}\int\exp(\lambda(z^{2}-1))\exp\Big(-\frac{z^{2}}{2}\Big)\dz \nonumber \\
&=\frac{\exp(-\lambda)}{\sqrt{1-2\lambda}}\leq\exp(2\lambda^{2}), \label{eq:lemma:chernoff:chi:square:1}
\end{align}
where the inequality holds from $-\log(1-2\lambda)\leq 2\lambda+4\lambda^{2}$ for $4\lambda\leq 1$. 

\par We next turn to verify the Chernoff bound for the $\chi^{2}$ random variable with $d$ degrees of freedom, denoted by $Y\sim\chi_{n}^{2}$. Note that $Y\stackrel{\d}{=}\sum_{i=1}^{d}Z_{i}^{2}$ where $Z_{1},\ldots,Z_{d}\sim^{\iid}\mathcal{N}(0,1)$. Then for $4\lambda<1$,
\begin{align*}
&\log\bbE\big[\exp\{\lambda(Y-\bbE[Y])\}\big] \\
&=\log\bbE\Big[\exp\Big\{\sum_{i=1}^{d}\lambda(Z_{i}^{2}-\bbE[Z_{i}^{2}])\Big\}\Big]=\sum_{i=1}^{d}\log\bbE[\lambda(Z_{i}^{2}-\bbE[Z_{i}^{2}])]\leq 2d\lambda^{2},
\end{align*}
where the inequality follows from~\eqref{eq:lemma:chernoff:chi:square:1}. Setting $\lambda=s\sigma^{2}$ completes the proof.
\end{proof}

\par The next lemma shows that the KL divergence can be bounded by the Fisher information.
\begin{lemma}
\label{lemma:kl:fisher}
Suppose the distribution $\pi$ satisfies the log-Sobolev inequality with constant $C_{\LSI}$. Then for each distribution $\mu$, 
\begin{equation*}
\kl(\mu\|\pi)\leq 2C_{\LSI}\bbE_{\pi}\Big[\|\nabla\sqrt{\frac{\mu}{\pi}}\|_{2}^{2}\Big].
\end{equation*}
\end{lemma}

\begin{proof}[Proof of Lemma~\ref{lemma:kl:fisher}]
Recall the log-Sobolev inequality
\begin{equation}\label{eq:proof:lemma:kl:fisher:1}
\ent_{\pi}(f^{2})\leq 2C_{\LSI}\bbE_{\pi}\big[\|\nabla f\|_{2}^{2}\big].
\end{equation}
Substituting $f^{2}(\vx)=\mu(\vx)/\pi(\vx)$ into the left-hand side of~\eqref{eq:proof:lemma:kl:fisher:1} deduces 
\begin{align*}
\ent_{\pi}\Big(\frac{\mu}{\pi}\Big)
&=\int\frac{\mu(\vx)}{\pi(\vx)}\log\frac{\mu(\vx)}{\pi(\vx)}\pi(\vx)\d\vx-\int\frac{\mu(\vx)}{\pi(\vx)}\pi(\vx)\d\vx\log\int\frac{\mu(\vx)}{\pi(\vx)}\pi(\vx)\d\vx \\
&=\int\mu(\vx)\log\frac{\mu(\vx)}{\pi(\vx)}\d\vx=\kl(\mu\|\pi).
\end{align*}
As a consequence, 
\begin{equation*}
\kl(\mu\|\pi)=\ent_{\pi}\Big(\frac{\mu}{\pi}\Big)\leq2C_{\LSI}\bbE_{\pi}\Big[\|\nabla\sqrt{\frac{\mu}{\pi}}\|_{2}^{2}\Big].
\end{equation*}
This completes the proof.
\end{proof}

\begin{lemma}
\label{lemma:score:fisher}
For two distributions $\mu$ and $\pi$, it holds that 
\begin{equation*}
\bbE_{\mu}\Big[\big\|\nabla\log\frac{\mu}{\pi}\big\|_{2}^{2}\Big]
=4\bbE_{\pi}\Big[\big\|\nabla\sqrt{\frac{\mu}{\pi}}\big\|_{2}^{2}\Big].
\end{equation*}
\end{lemma}

\begin{proof}[Proof of Lemma~\ref{lemma:score:fisher}]
It is straightforward that
\begin{align*}
\bbE_{\mu}\Big[\big\|\nabla\log\frac{\mu}{\pi}\big\|_{2}^{2}\Big]
&=\int\big\|\nabla\log\frac{\mu(\vx)}{\pi(\vx)}\big\|_{2}^{2}\mu(\vx)\d\vx=\int\frac{\pi(\vx)}{\mu(\vx)}\big\|\nabla\frac{\mu(\vx)}{\pi(\vx)}\big\|_{2}^{2}\pi(\vx)\d\vx \\
&=4\int\big\|\frac{1}{2}\sqrt{\frac{\pi(\vx)}{\mu(\vx)}}\nabla\frac{\mu(\vx)}{\pi(\vx)}\big\|_{2}^{2}\pi(\vx)\d\vx=4\int\big\|\nabla\sqrt{\frac{\mu(\vx)}{\pi(\vx)}}\big\|_{2}^{2}\pi(\vx)\d\vx \\
&=4\bbE_{\pi}\Big[\big\|\nabla\sqrt{\frac{\mu}{\pi}}\big\|_{2}^{2}\Big],
\end{align*}
which completes the proof.
\end{proof}

\begin{lemma}\label{lemma:psi:phi}
Let $\mu$ and $\pi$ be two distributions. Define $\phi$ as $\d\mu=\phi\d\pi$, and define $\psi\bbE_{\pi}[\phi^{2}]=\phi$. Then the following equality holds 
\begin{equation*}
\bbE_{\mu}\big[\psi\|\nabla\log(\psi\phi)\|_{2}^{2}\big]=\frac{4\bbE_{\pi}[\|\nabla_{\vx}\phi\|_{2}^{2}]}{\bbE_{\pi}[\phi^{2}]}.
\end{equation*}
\end{lemma}

\begin{proof}[Proof of Lemma~\ref{lemma:psi:phi}]
It is straightforward that 
\begin{equation*}
\|\nabla\log(\psi\phi)\|_{2}^{2}=\|\nabla\log\frac{\phi^{2}}{\bbE_{\pi}[\phi^{2}]}\|_{2}^{2}=4\|\nabla\log\phi\|_{2}^{2}=\frac{4\|\nabla_{\vx}\phi\|_{2}^{2}}{\phi^{2}}.
\end{equation*}
As a consequence, 
\begin{equation*}
\bbE_{\mu}\big[\psi\|\nabla\log(\psi\phi)\|_{2}^{2}\big]=\frac{4}{\bbE_{\pi}[\phi^{2}]}\int\frac{\|\nabla_{\vx}\phi(\vx)\|_{2}^{2}}{\phi(\vx)}\mu(\vx)\d\vx=\frac{4\bbE_{\pi}[\|\nabla_{\vx}\phi\|_{2}^{2}]}{\bbE_{\pi}[\phi^{2}]},
\end{equation*}
which completes the proof.
\end{proof}


\section{Additional Numerical Experiments}
\label{section:experiments:appendix}

\begin{enumerate}[(1)]
\item Section~\ref{section:experiment:appendix:linearGaussian} showcases the application of SSLS to a linear Gaussian state-space model, highlighting its capability to accurately estimate the posterior distribution even in the presence of initial prior distribution shifts.
\item Section~\ref{section:experiments:appendix:Lorenz} focuses on applying SSLS to the Lorenz-96 model and comparing it with APF. The primary objective here is to assess the robustness of SSLS against initial prior distribution shifts, while also evaluating the impact of ensemble size on its performance across various metrics.
\end{enumerate}

\subsection{Linear Gaussian state-space model}
\label{section:experiment:appendix:linearGaussian}

\par To begin with, we look into a one-dimensional linear Gaussian state-space model, for which the ground truth posterior can be computed by the Kalman filter~\cite{Sarkka20232023Bayesian}. The state-space model is defined as
\begin{equation}\label{eq:LinearSSM}
\begin{aligned}
X_{k+1}&=X_{k}+V_{k}, && V_{k}\sim N(0,5), \\
Y_{k}&=X_{k}+W_{k}, && W_{k}\sim N(0,0.2),
\end{aligned}
\end{equation}
where $k\in\bbN$, and the initial prior distribution is set as $X_{1}\sim N(0,1)$. 
The SSLS ensemble size is 500.
\begin{figure}[htbp]
\centering
\includegraphics[width=0.75\linewidth]{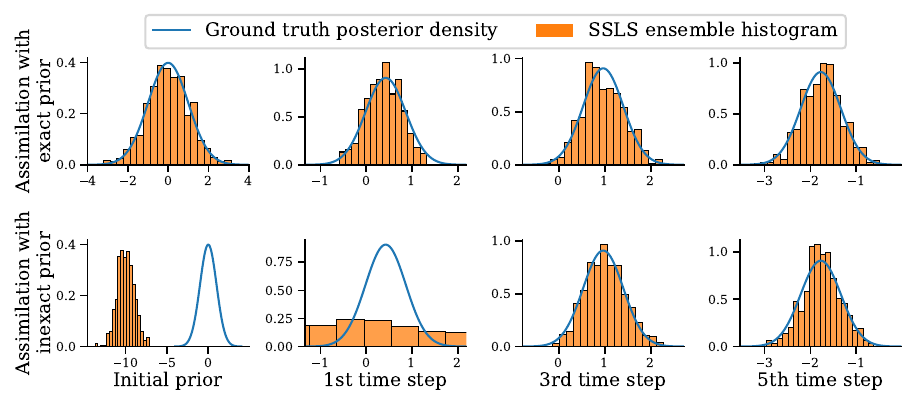}
\caption{Posterior distributions estimated by SSLS in a linear Gaussian state-space model~\eqref{eq:LinearSSM}. (a) The top row shows the histogram of the SSLS ensemble with an exact initial prior distribution. (b) The bottom row demonstrates the histogram of the SSLS ensemble with an inexact initial prior distribution.}
\label{fig:Linear}
\end{figure}

\par\textbf{Assimilation with exact initial prior.}
We begin by considering the case where the SSLS is carried out with the exact initial prior. The experimental results are shown in the top row of Figure~\ref{fig:Linear}, indicating that the distribution of the SSLS ensemble closely aligns with the ground truth posterior throughout all time steps, given that the SSLS is conducted without the initial prior distribution shift. This empirical finding validates the theoretical conclusions outlined in Theorem~\ref{theorem:section:convergence:assimilation}.

\par\textbf{Assimilation with inexact initial prior.}
In practical scenarios, the initial prior distribution is typically intractable. Therefore, it is essential to assess the robustness of the SSLS against the initial prior distribution shift. In this experiment, the SSLS is initialized by an inexact prior of $N(-10,1)$, and the outcomes are presented in the bottom row of Figure~\ref{fig:Linear}. The results demonstrate that, even in the presence of an initial prior distribution shift, the SSLS ensemble closely matches the ground truth posterior distribution after a small number of time steps, despite prominent estimation errors in initial few time steps.

\par Recall that Theorem~\ref{theorem:section:convergence:assimilation} provides an error bound that increases with the number of time step. Nonetheless, this error bound may be too loose to accurately depict empirical findings, as it solely characterizes a worst-case scenario. Therefore, establishing a tighter error bound that precisely reflects experimental observations remains an open question. This gap between theoretical understandings and experimental observations will be explored in greater depth in future work.

\subsection{Lorenz-96} 
\label{section:experiments:appendix:Lorenz}

\par Lorenz-96 is a widely-used benchmark in the field of numerical weather forecasting~\cite{Majda2012Filtering,Reich2015Probabilistic,Evensen2022Data,Spantini2022Coupling}, which is defined by a set of nonlinear ODEs representing the spatial discretization of a time-dependent PDE
\begin{equation}\label{eq:Lorenz}
\frac{\d}{\dt}Z_{t,i}=(Z_{t,i+1}-Z_{t,i-2})Z_{t,i-1}-Z_{t,i}+F, \quad 1\leq i\leq d.
\end{equation}
In this experiment, we consider the twenty-dimensional Lorenz-96 system. We set a constant forcing parameter $F=8$, resulting in a fully chaotic dynamic~\cite{Majda2012Filtering}, where slightly different initial conditions lead to extremely different trajectories. 

\par The dynamics model is defined by discretizing~\eqref{eq:Lorenz} using the fourth-order explicit Runge-Kutta method with time step $\delta t$. The states at discrete times are denoted by $(\mX_{k})_{k\in\bbN}$ with $\mX_{k}=(Z_{k\delta t,1},\ldots,Z_{k\delta t,d})^{\top}$. At each time step $k\in\bbN$, we employ a linear measurement model with Gaussian additive noise
\begin{equation}\label{eq:Lorenz:measurement}
\mY_{k}=\mX_{k}+\sigma_{\obs}\mW_{k},
\end{equation}
where $\mW_{k}\sim N(\mathbf{0},\mI_{d})$ denotes the measurement noise. 

\par\textbf{Baseline.}
To mitigate the degeneracy of the APF, a small amount of Gaussian noise $N(0,10^{-1}\mI_{d})$ is incorporated into the state at each iteration of the Runge-Kutta method~\cite{Spantini2022Coupling}. To ensure a fair comparison, the ensemble size for both APF and SSLS is set to 500. To showcase the robustness of SSLS against the initial prior distribution, we intentionally set the initial distribution of both SSLS and APF away from the ground truth initial prior. See Appendix~\ref{appendix:experiment:details} for more details on training parameters.

\par\textbf{Metrics for assimilation.}
To quantitatively measure the performance of SSLS and study the impact of ensemble size on the assimilation performance, we focus on four metrics as~\cite{Spantini2022Coupling}, including
\begin{enumerate}[(i)]
\item RMSE: the root mean squared error,
\item spread: the root mean trace of the ensemble covariance matrix,
\item coverage probability: the coverage probability of the intervals given by the empirical 2.5\% and 97.5\% quantiles of each marginal of the ensemble, and 
\item CRPS: the continuous ranked probability score~\cite{Gneiting2007Probabilistic,Brocker2012Evaluating}.
\end{enumerate}
The RMSE quantifies the discrepancy between estimated states and reference states, while the spread indicates the concentration of the ensemble particles. The coverage probability assesses the likelihood that a marginal confidence interval includes the reference states. The CRPS is a statistical metric used to assess the accuracy of the estimated posterior by comparing it to the observed data. It measures the discrepancy between the cumulative distribution function of the estimated posterior and the cumulative distribution function of the observations. A lower CRPS value indicates a better alignment between the estimated posterior and observed distributions, indicating a more accurate estimate. 

\begin{figure}[htbp]
\centering
\includegraphics[width=0.75\linewidth]{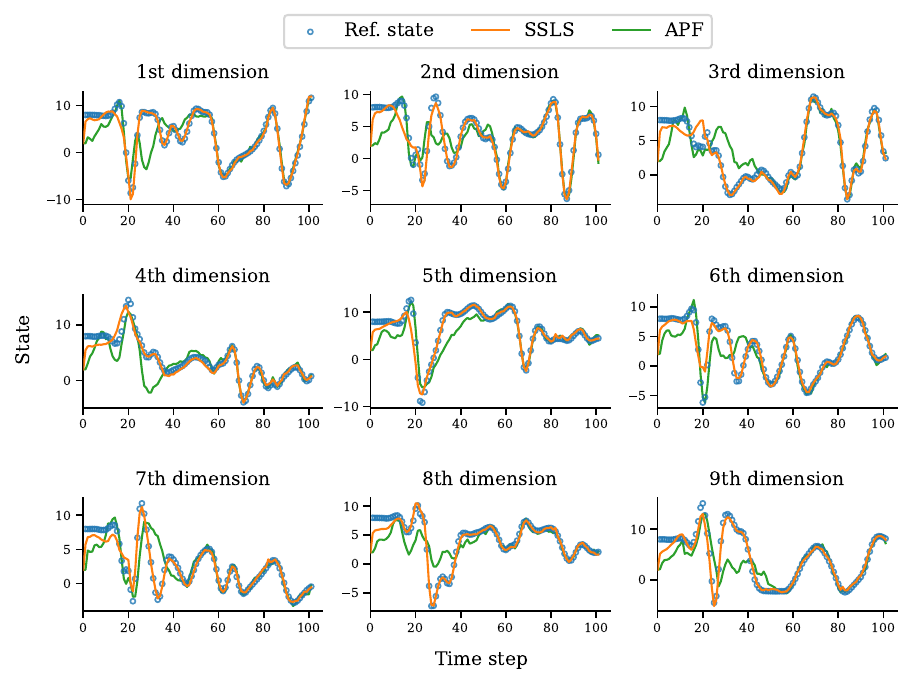}
\caption{Evolution of the reference states, the SSLS ensemble mean, and the APF ensemble mean for Lorenz-96~\eqref{eq:Lorenz}.}
\label{fig:LorenzEvolution}
\end{figure}

\begin{figure}[htbp]
\centering
\includegraphics[width=0.75\linewidth]{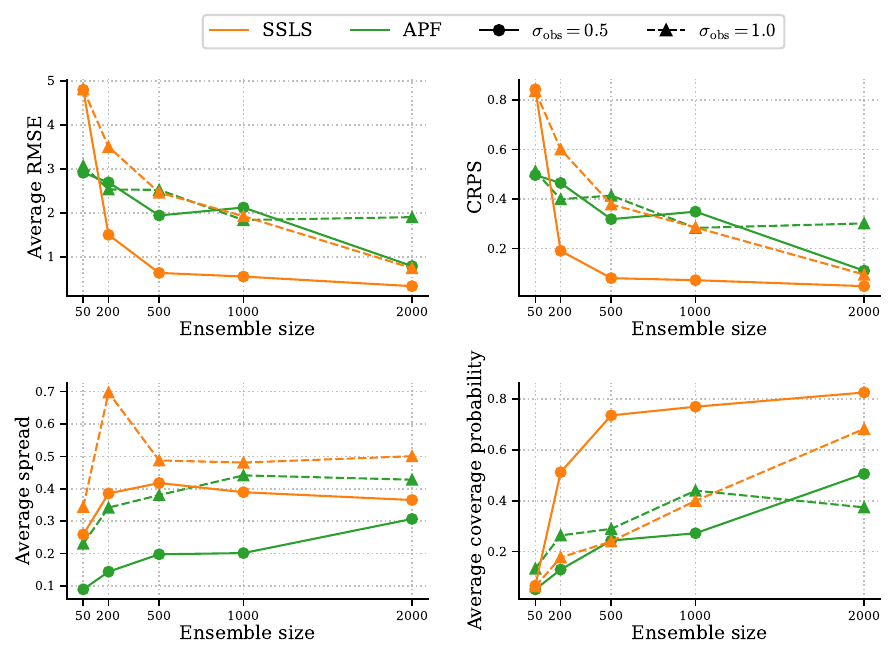}
\caption{Performance metrics of SSLS and APF for Lorenz-96~\eqref{eq:Lorenz}. For each ensemble size, metrics are averaged over elements of the estimated states and time steps.}
\label{fig:LorenzMetrics}
\end{figure}

\par\textbf{Experimental results.}
Figure~\ref{fig:LorenzEvolution} demonstrates the evolution of the first nine elements of the states estimated by SSLS and APF, respectively. Despite an initial prior distribution shift, SSLS effectively corrects this error within a few subsequent assimilation steps, whereas APF requires a significantly longer assimilation time to compensate for the initial prior distribution shift. These empirical observations are consistent with those of the previous experiment. Figure~\ref{fig:LorenzMetrics} presents four metrics for both SSLS and APF with different ensemble size.
\begin{enumerate}[(i)]
\item \textbf{RMSE}: The average RMSE of both SSLS and APF decreases as the ensemble size increases, which is consistent with our theoretical findings. Specifically, the discussions below Assumption~\ref{assumption:sm} indicates that the score matching error in SSLS approaches zero as the ensemble size increases towards infinity. Consequently, the assimilation error decreases as the ensemble size increases, as evidenced by the results in Theorem~\ref{theorem:section:convergence:assimilation}. Additionally, Figure \ref{fig:LorenzMetrics} illustrates that SSLS performs better than APF for ensemble sizes exceeding 200.
\item \textbf{CRPS}: As depicted in the upper right portion of Figure~\ref{fig:LorenzMetrics}, the average CRPS of SSLS decreases significantly as the ensemble size increases. This trend suggests an enhanced alignment between the estimated posterior distribution and the observation data. Furthermore, for ensemble sizes exceeding 500, SSLS demonstrates lower CRPS values compared to the APF, underscoring the superior effectiveness of SSLS in posterior estimation.
\item \textbf{Spread and coverage probability}: The bottom row of Figure~\ref{fig:LorenzMetrics} demonstrates that the coverage probability of SSLS noticeably grows with the increase in ensemble size, while the spread remains relatively stable. Moreover, Figure~\ref{fig:LorenzMetrics} indicates that at a low noise level of $\sigma_{\obs}=0.5$, SSLS exhibits a much higher coverage probability compared to APF, despite having a larger average spread. This disparity can be attributed to particle degeneracy in APF.
\end{enumerate}
For further discussion on this experiment, please refer to Appendix \ref{appendix:experiment:details}.

\subsection{Kolmogorov Flow baseline}
\label{section:experiments:appendix:Kolmogorov}

We evaluate SSLS against the Ensemble Kalman Diffusion Guidance (EnKG) \cite{zheng2025ensemble}. This comparison is conducted on the Kolmogorov Flow experiment with partial observations.

The original EnKG was developed as an alternative to standard Diffusion Posterior Sampling (DPS) \cite{chung2023diffusion} for general inverse problems. It differs from the sequential assimilation framework presented in this work. To facilitate a fair comparison in our scenario, we adapt its core contribution—a derivative-free estimation of the likelihood guidance for black-box measurement models—as a substitute for our likelihood calculation. Recall in SSLS, the posterior sampling is implemented using the Langevin diffusion:
\[
  \mathrm{d} \mathbf{Z}_{t}=\nabla_{\mathbf{x}}\log\pi_{k+1}(\mathbf{Z}_{t}|\mathbf{y}_{[k+1]})\mathrm{d}t+\sqrt{2}\mathrm{d}\mathbf{B}_{t},  \quad \mathbf{Z}_{0}\sim\pi_{k+1}^{m-1}(\cdot|\mathbf{y}_{[k+1]}).
\]
As for EnKG, the posterior sampling follows a prediction-correction scheme. The prediction step employs Langevin diffusion based on the prior:
\[
  \mathbf{Z}_{t+\delta t}^\prime = \mathbf{Z}_{t} + \nabla_{\mathbf{x}}\log q_{k+1}(\mathbf{Z}_{t}|\mathbf{y}_{[k]}) \delta t +\sqrt{2\delta t} \xi_t,  \quad \mathbf{Z}_{0}\sim\pi_{k+1}^{m-1}(\cdot|\mathbf{y}_{[k+1]}), \quad \xi_t \sim N(\mathbf{0}, \mathbf{I}). \tag{prediction step}
\]
The correction step then incorporates likelihood guidance through the derivative-free estimation method proposed in EnKG \cite[Algorithm 2]{zheng2025ensemble}:
\[
  \mathbf{Z}_{t+\delta t} \in \argmin_{\mathbf{Z}} \frac{\|\mathbf{Z} - \mathbf{Z}_{t+\delta t}^\prime\|^2_2}{2 w_i} - \log g_{k+1} (\mathbf{y}_{k+1} | \mathbf{Z}). \tag{correction step}
\]
To ensure a fair comparison, both SSLS and EnKG utilize the same diffusion prior architecture, training epochs, and total Langevin iterations. Since the computational budgets for likelihood guidance in both methods are comparable, their overall computational costs are nearly identical (refer to Table \ref{tab:ns-baseline} for details).

The assimilation results are presented in Figure \ref{fig:ns-baseline}. In unobserved regions, EnKG fails to capture the latent state patterns, whereas SSLS demonstrates superior performance. This superiority indicates that in data assimilation tasks, exact likelihood computation can significantly enhance assimilation accuracy.

\begin{figure}[htbp]
  \centering
  \includegraphics[width=0.6\linewidth]{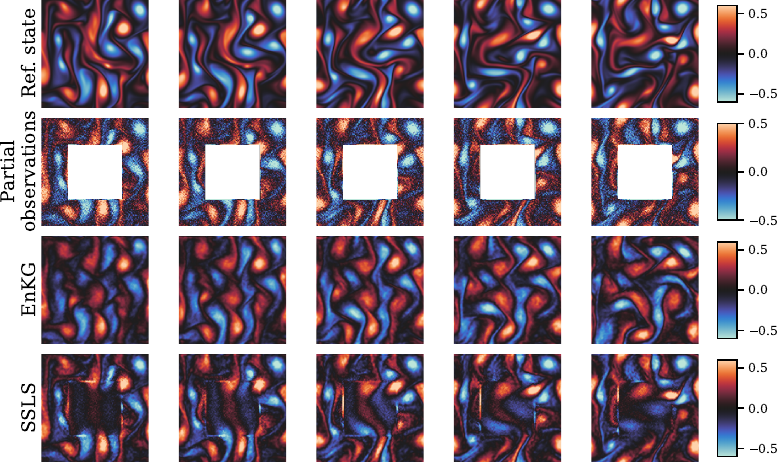}
  \caption{Performance comparison between SSLS and EnKG for the Kolmogorov Flow under partial observations.}
  \label{fig:ns-baseline}
\end{figure}

\begin{table}[htbp]
  \centering
  \caption{Computational budgets (per dynamical step) and assimilation error of SSLS and baseline EnKG, on the Kolmogorov Flow equation with partial observations.}
  \label{tab:ns-baseline}
  \begin{tabular}{ccccc}
    \toprule
    & training (sec) & update (sec) & memory load (GB) & RMSE\\
    \midrule
    EnKG &  55.88 $\pm$ 0.23 & 37.89 $\pm$ 0.19 & 11.78 $\pm$ 0.03 & 0.763 \\
    SSLS & 55.83 $\pm$ 0.24 & 37.32 $\pm$ 0.22 & 11.80 $\pm$ 0.03 & 0.300 \\
    \bottomrule
  \end{tabular}
\end{table}

\subsection{Computational complexity}
\label{section:experiments:appendix:budget}

In practice, computational efficiency is a critical factor in evaluating data assimilation methods. For the double-well problem, we report the execution time and memory consumption of the proposed SSLS, APF, and EnKF, all utilizing a consistent ensemble size of $n=1000$. For SSLS, each dynamic step involves 500 training epochs and a total of $K=400$ Langevin iterations. To ensure a fair comparison, APF and EnKF are also implemented on the same GPU hardware (NVIDIA A800, 80G). The detailed computational budgets are summarized in Table \ref{tab:computational-budget}.

\begin{table}[htbp]
  \centering
  \caption{Computational Time Analysis per dynamic step in the double-well experiment.}
  \label{tab:computational-budget}
  \begin{tabular}{lccc}
    \toprule
    & SSLS & APF & EnKF \\
    \midrule
    Training (sec) &  3.665 $\pm$ 0.185 & --- & --- \\
    Update (sec) & 0.109 $\pm$ 0.012 & 0.001 $\pm$ 0 & 0.0003 $\pm$ 0 \\
    Memory Load (MB) & 21.261 $\pm$ 0.002 & 0.260 $\pm$ 0.001 & 8.346 $\pm$ 0.001 \\
    Average RMSE & 0.194 & 0.388 & 0.225 \\
    \bottomrule
  \end{tabular}
\end{table}

These results reflect a trade-off between precision and efficiency. Specifically, SSLS achieves superior accuracy at the expense of higher temporal and memory requirements compared to the baseline methods. However, it is important to note that for APF and EnKF, simply increasing the ensemble size fails to resolve fundamental issues such as particle degeneracy and severe nonlinearity (see Figure \ref{fig:LorenzMetrics}). In contrast, the proposed SSLS effectively overcomes these challenges while maintaining a moderate ensemble size (see Figure \ref{fig:LorenzMetrics} and Figure \ref{fig:sensitivity} (f)) and manageable computational overhead.

\subsection{Hyper-parameter sensitivity analysis}
\label{appendix:experiment:sensitivity}

We conduct sensitivity analysis on hyper-parameters:
Langevin iterations $K$, annealing levels $M$, step size $h$,
Gaussian smoothing level $\sigma$, annealing parameter $\rho$, and ensemble size $n$.
Performance is evaluated using RMSE and CRPS. The results, presented in Figure \ref{fig:sensitivity}, yield the following observations:
\begin{itemize}
  \item \textbf{Scaling Effects}: Increasing the number of Langevin steps, temperature levels, and ensemble size consistently reduces assimilation errors.
  \item \textbf{Step Size Trade-off}: A ``sweet spot'' for the step size $h$ is identified at $h=0.0005$. This aligns with theoretical expectations: given a fixed number of iterations, an excessively small step size increases sampling error due to insufficient exploration, while an overly large step size introduces significant discretization error.
  \item \textbf{Smoothing and Stability}: Insufficient Gaussian smoothing ($\sigma$) leads to instability in denoising score matching (DSM), which degrades performance. Specifically, at $\sigma = 0.02$, DSM underperforms compared to sliced score matching ($\sigma=0$). However, as $\sigma$ increases from 0.02 to 0.15, the error decreases and stabilizes, with DSM eventually outperforming the sliced approach.
  \item \textbf{Annealing Efficacy}: SSLS with annealing ($\rho \neq 0$) consistently outperforms the non-annealed version ($\rho = 0$). The best annealing strategy allocates more updates at lower inverse temperatures ($\rho > 1$).
\end{itemize}

In conclusion, SSLS is particularly effective when the posterior exhibits multi-modality or when the system undergoes drastic transitions. While proper annealing mitigates the challenges posed by multi-modality, the generalization capabilities of the neural network enable the model to adapt robustly to sudden state changes.

\begin{figure}[htpb]
  \centering
  \begin{minipage}{0.32\textwidth}
    \centering
    \includegraphics[width=\linewidth]{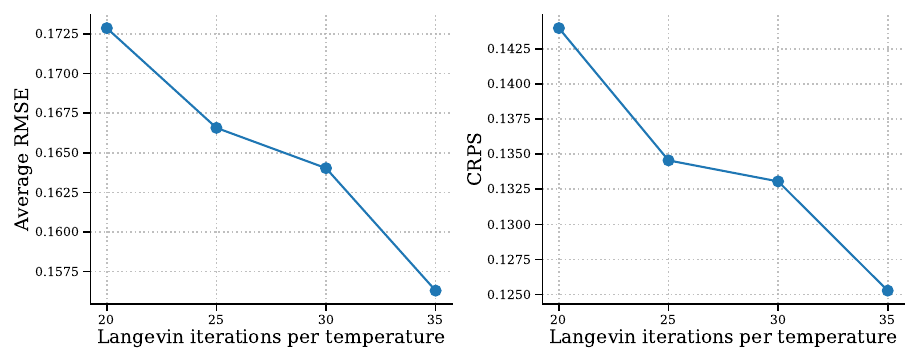}\\
    \makebox[\linewidth]{\tiny (a) SSLS with different Langevin steps.} 
  \end{minipage}
  \hfill
  \begin{minipage}{0.32\textwidth}
    \centering
    \includegraphics[width=\linewidth]{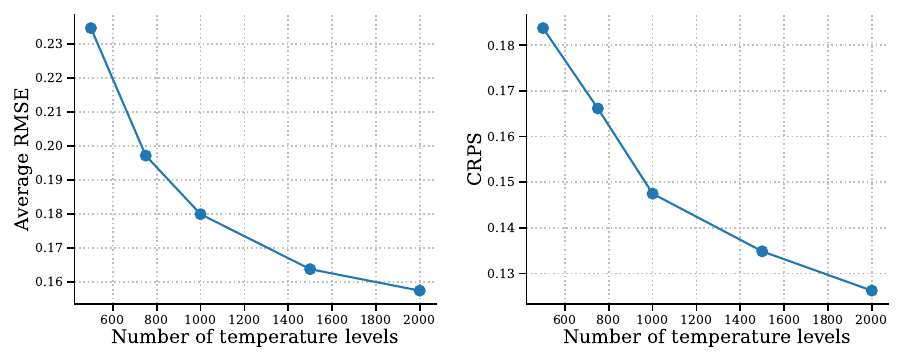}\\
    \makebox[\linewidth]{\tiny (b) SSLS with different annealing levels.} 
  \end{minipage}
  \hfill
  \begin{minipage}{0.32\textwidth}
    \centering
    \includegraphics[width=\linewidth]{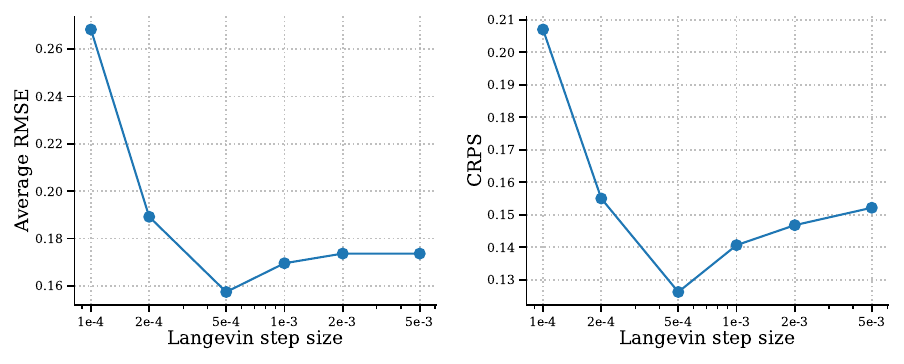}\\
    \makebox[\linewidth]{\tiny (c) SSLS with different Langevin step sizes.} 
  \end{minipage}

  \vspace{1em}

  \begin{minipage}{0.3\textwidth}
    \centering
    \includegraphics[width=\linewidth]{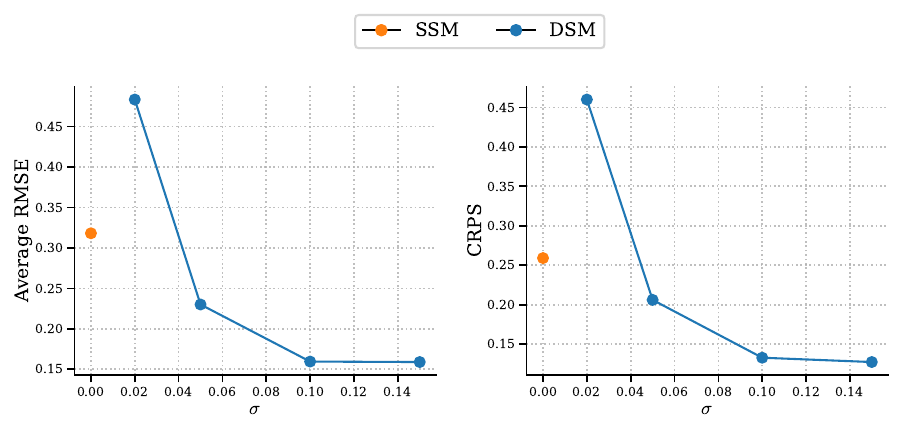}\\
    \makebox[\linewidth]{\tiny (d) SSLS with different Gaussian smoothing levels.} 
  \end{minipage}
  \hfill
  \begin{minipage}{0.3\textwidth}
    \centering
    \includegraphics[width=\linewidth]{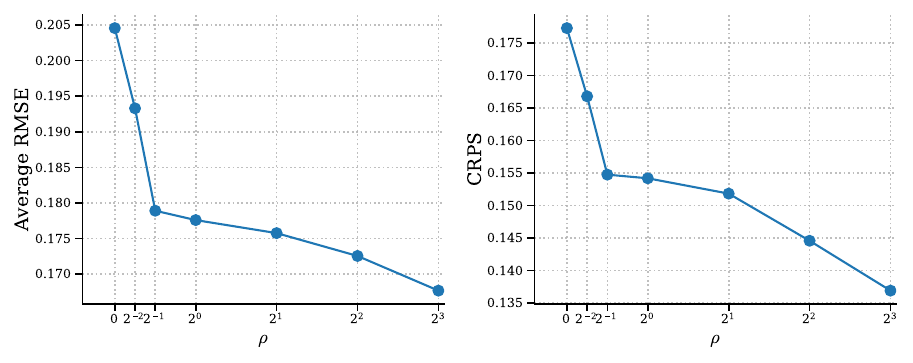}\\
    \makebox[\linewidth]{\tiny (e) SSLS with different annealing schedule.} 
  \end{minipage}
  \hfill
  \begin{minipage}{0.3\textwidth}
    \centering
    \includegraphics[width=\linewidth]{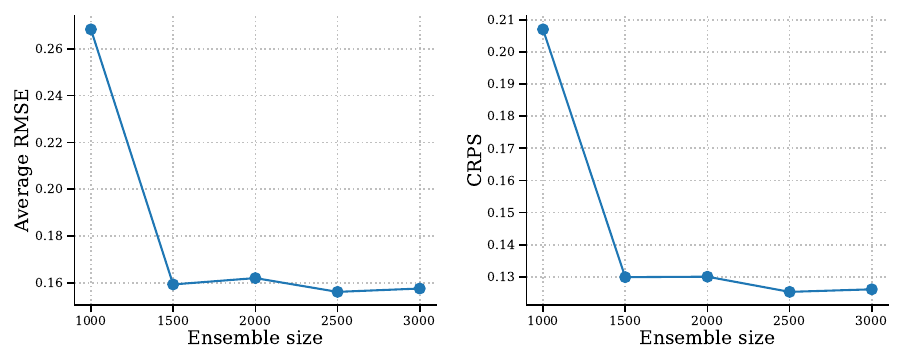}\\
    \makebox[\linewidth]{\tiny (f) SSLS with different ensemble size.} 
  \end{minipage}

  \caption{
    Sensitivity analysis of SSLS with different hyper-parameters.
    Metrics are averaged over elements of the estimated states and time steps.
  }
  \label{fig:sensitivity}
\end{figure}

\section{Experimental details}
\label{appendix:experiment:details}

\subsection{Double-well problem}
For the first problem, we adopt a residual neural network with 2 hidden layers to learn for the prior score. The width of each hidden layer is set as 128, and the activation functions is chosen as the sigmoid function. During the learning procedure, we apply the denoising score match method \cite{Vincent2011Connection} with a noise level 0.1 to improve the training efficiency. After that, the gradient of log posterior can be explicitly evaluated.

In the implementation of APF and EnKF, the ensemble size is set as $n=1000$, same to the sample size used in SSLS. Among all the three methods, the initial state distribution is defined as the normal distribution $N(-1, 0.15^2)$.

\subsection{Lorenz 96 model}

For the Lorenz 96 problem, we adopt a 1D UNet to learn for the prior score. The channels are 32, 64 and 128, and the activation functions is chosen as the ReLU function. During the learning procedure, we apply the denoising score match method \cite{Vincent2011Connection} with a noise level 0.1 to improve the training efficiency. After that, the gradient of log posterior can be explicitly evaluated.

In the implementation of APF, if unspecified, the ensemble size is set as $n=500$, same to the sample size used in SSLS. The initial state distribution is defined as the normal distribution $N(\mathbf{0}, \mathbf{I}_{20})$.

\begin{figure}
\centering
\centering
\includegraphics[width=0.50\linewidth]{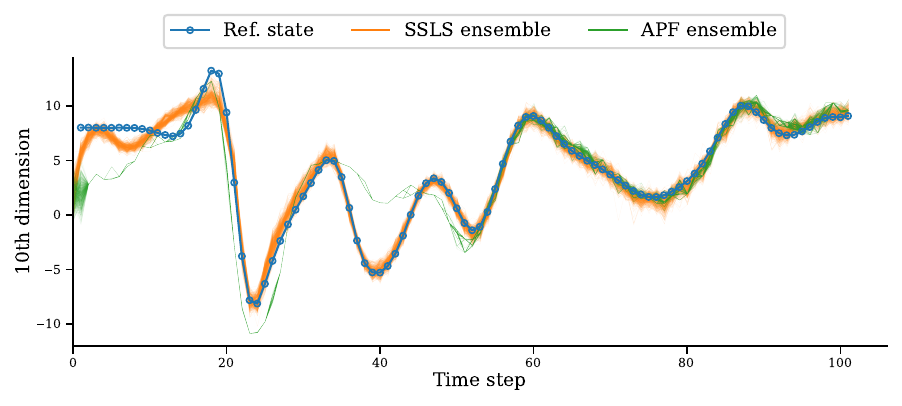}
\caption{Evolution of $x_{10}$ of the true states, the SSLS ensemble and the APF ensemble on REFERENCE LORENZ.}
\label{fig:LorenEnsemble}
\end{figure}

\begin{figure}
\centering
\includegraphics[width=0.50\linewidth]{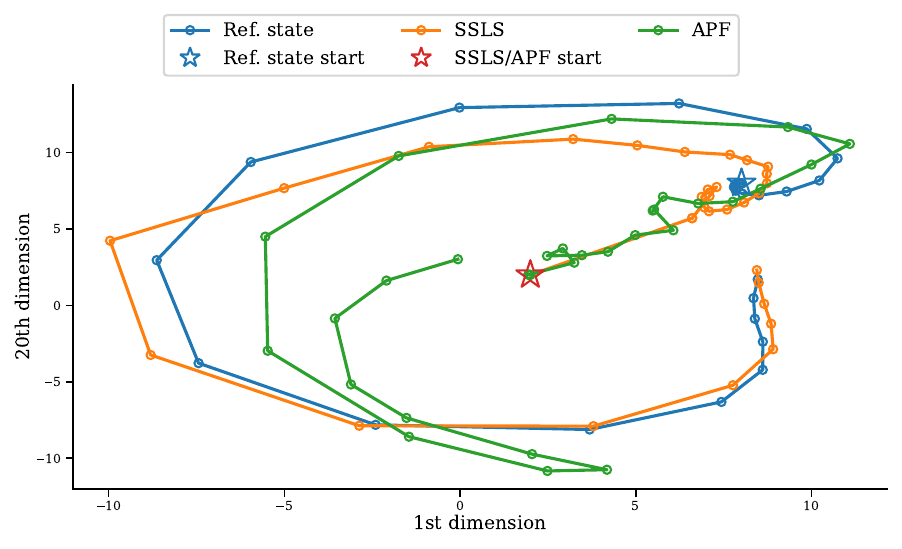}
\caption{Trajectory of the true states, the SSLS estimation and the APF estimation for a Lorenz 96 system. The trajectory is visualized in the $x_1$-$x_{20}$ space.}
\label{fig:LorenzTrajectory}
\end{figure}

The superiority of SSLS in Figure \ref{fig:LorenzEvolution} can also be understood using Figure \ref{fig:LorenEnsemble}. Starting from a bad initial guess, the width of assimilation band of the APF ensemble (the maximum difference between any two samples in the ensemble) becomes narrowing rapidly, due to the imbalanced distribution of the likelihood value. This phenomenon greatly reduces the efficiency of APF. On the contrary, SSLS adopts a continues network function to approximate the prior distribution, the better generalization ability increases the coverage probability for the true state. In Figure \ref{fig:LorenzTrajectory}, we compare the predicted trajectories of SSLS and APF with true state in the $x_1$-$x_{20}$ space, which again verifies this advantage.

\subsection{Kolmogorov flow}

For the Kolmogorov flow problem, we adopt a 2D UNet to learn for the prior score. During the learning procedure, we apply the denoising score match method \cite{Vincent2011Connection} with a noise level 0.2 to improve the training efficiency. After that, the gradient of log posterior can be explicitly evaluated.
The sample size used in SSLS is set as $n=500$.

\begin{figure}
\centering
\includegraphics[width=0.50\linewidth]{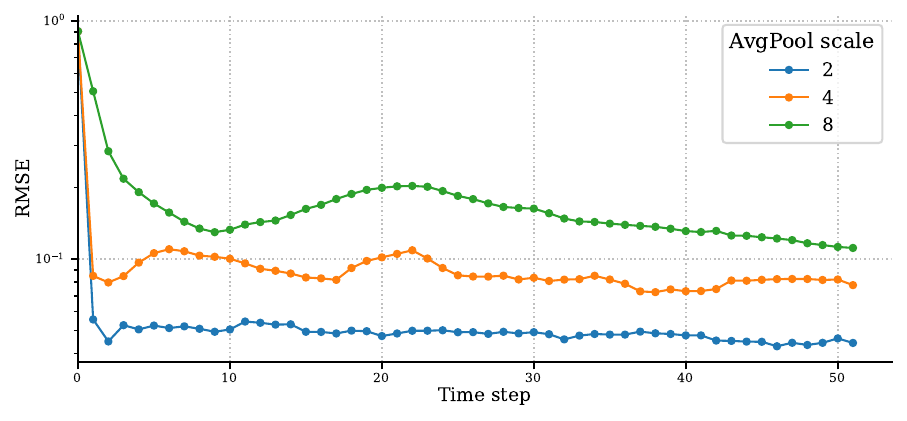}
\caption{RMSE of SSLS assimilated states at different average pooling scale. Time 0 corresponds to RMSE from the expectation of the prior distribution, when the assimilation has not taken place. The three lines share the same starting RMSE as they share the same guess on the prior distribution.}
\label{fig:KolmogorovMetrics}
\end{figure}

To further study the evolution of error, we also compare the RMSE of SSLS under different average pooling scale in Figure \ref{fig:KolmogorovMetrics}, which demonstrates at the early stage of assimilation (eliminating the effect of initial lag error), the RMSE would decreases rapidly to a small value, and then maintain stable. Furthermore, a larger pooling scale would make the problem more difficult , resulting a slower decrease of RMSE, converging to a higher value.

\subsection{Numerical stability improvements}

Throughout our numerical experiments, we mainly adopt two numerical improvements on the original algorithm for stability.
\begin{enumerate}[(i)]
\item The first improvement is that, before matching the score function of the prior distribution at each step, we normalize the samples to zero mean and identity covariance. Then we match the score function on the normalized distribution, from which we obtain the original score function after affine transformation.
\item Another improvement is that, we manually clip the score function of estimated posterior score by its $L^2$-norm to ensure the stability of the score-based Langevin sampling.
\end{enumerate}

\end{document}